\documentclass[11pt]{amsart}
\usepackage{amsfonts, amssymb, amscd, latexsym, graphicx, psfrag, color, float, enumitem}
\usepackage[all]{xy}
\usepackage{mathrsfs}

\usepackage{tikz}

\usepackage[hyphens]{url}

\usepackage{todonotes}
\usepackage{hyperref}

\usepackage[hyphenbreaks]{breakurl}

\usepackage{tkz-euclide}
\usetkzobj{all} 
\usepackage{caption}

\newtheorem{dummy}{dummy}[section]
\newtheorem{lemma}[dummy]{Lemma}
\newtheorem{theorem}[dummy]{Theorem}
\newtheorem{innercustomthm}{Theorem}
\newenvironment{customthm}[1]
{\renewcommand\theinnercustomthm{#1}\innercustomthm}
  {\endinnercustomthm}
  
  \newtheorem{innercustomcor}{Corollary}
\newenvironment{customcor}[1]
{\renewcommand\theinnercustomcor{#1}\innercustomcor}
  {\endinnercustomcor}

\newtheorem{corollary}[dummy]{Corollary}
\newtheorem{proposition}[dummy]{Proposition}
\theoremstyle{definition}
\newtheorem{definition}[dummy]{Definition}
\newtheorem*{definition*}{Definition}
\newtheorem{example}[dummy]{Example}

\newtheorem{remark}[dummy]{Remark}

\newcommand{\bA}{\mathbb{A}}
\newcommand{\bC}{\mathbb{C}}
\newcommand{\bG}{\mathbb{G}}

\newcommand{\bN}{\mathbb{N}}
\newcommand{\bP}{\mathbb{P}}
\newcommand{\bQ}{\mathbb{Q}}

\newcommand{\bZ}{\mathbb{Z}}

\newcommand{\cA}{\mathcal{A}}
\newcommand{\cB}{\mathcal{B}}
\newcommand{\cC}{\mathcal{C}}
\newcommand{\cD}{\mathcal{D}}

\newcommand{\cF}{\mathcal{F}}

\newcommand{\cI}{\mathcal{I}}

\newcommand{\cO}{\mathcal{O}}
\newcommand{\cP}{\mathcal{P}}

\newcommand{\cS}{\mathcal{S}}
\newcommand{\cT}{\mathcal{T}}
\newcommand{\cU}{\mathcal{U}}

\newcommand{\Spec}{\mathrm{Spec}\,}

\newcommand{\Hom}{\mathrm{Hom}}

\newcommand{\Perf}{\mathrm{Perf}}

\newcommand{\SCat}{\mathrm{Cat}^{\mathrm{perf}}_\infty}

\newcommand{\Div}{\mathrm{Div}}

\newcommand{\gp}{\mathrm{gp}}
\newcommand{\Qcoh}{\mathrm{Qcoh}}

\newcommand\radice[2][\relax]{\hspace{-1.5pt}\sqrt[\uproot{2}#1]{#2}}

\newcommand{\Coh}{\mathrm{Coh}}
\newcommand\Db[1]{\Perf(#1)}

\newcommand*\cocolon{%
        \nobreak
        \mskip6mu plus1mu
        \mathpunct{}%
        \nonscript
        \mkern-\thinmuskip
        {:}%
        \mskip2mu
        \relax
}

\usepackage{stackengine}
\usepackage{calc}
\newlength\shlength
\newcommand\xshlongvec[2][0]{\setlength\shlength{#1pt}%
  \stackengine{-5.6pt}{$#2$}{\smash{$\kern\shlength%
    \stackengine{7.55pt}{$\mathchar"017E$}%
      {\rule{\widthof{$#2$}}{.57pt}\kern.4pt}{O}{r}{F}{F}{L}\kern-\shlength$}}%
      {O}{c}{F}{T}{S}}

\begin{document}

\author[Scherotzke]{Sarah Scherotzke}
\address{Sarah Scherotzke, 
	Universit\'e du Luxembourg\\
	Maison du Nombre\\
	6, Avenue de la Fonte\\
	L-4364 Esch-sur-Alzette\\ Luxembourg}
\email{\href{mailto:sarah.scherotzke@uni.lu}{sarah.scherotzke@uni.lu}}

\author[Sibilla]{Nicol\`o Sibilla}
\address{Nicol\`o Sibilla, SMSAS\\ 
	University of Kent\\ 
	Canterbury, Kent CT2 7NF\\UK and SISSA\\ Via Bonomea 265 \\34136 Trieste (TS) \\Italy}
\email{\href{mailto:N.Sibilla@kent.ac.uk}{N.Sibilla@kent.ac.uk}}

\author[Talpo]{Mattia Talpo}
\address{Mattia Talpo, Dipartimento di Matematica\\ Universit\`{a} di Pisa \\ Largo Bruno Pontecorvo 5 \\ 56127 Pisa (PI) \\ Italy}
\email{\href{mailto:mattia.talpo@unipi.it}{mattia.talpo@unipi.it}}

\title[Parabolic semi-orthogonal decompositions 
and Kummer flat invariants]{Parabolic semi-orthogonal decompositions and Kummer flat  invariants of log schemes}


\begin{abstract} 
 We construct semi-orthogonal decompositions on   triangulated categories of parabolic sheaves on certain kinds of logarithmic schemes. This provides a categorification of   the decomposition theorems in Kummer flat K-theory due to Hagihara and Nizio{\l}. Our techniques allow us to generalize Hagihara and Nizio{\l}'s  results to a much larger class of invariants in addition to K-theory, and also to extend them to more general logarithmic  stacks. 
 \end{abstract}

\maketitle

\tableofcontents

\section{Introduction} 
 In this paper we carry forward the study of the derived category of parabolic sheaves we initiated   in \cite{scherotzke2016logarithmic}, where we established the Morita invariance of   parabolic sheaves under logarithmic blow-ups. Our main result is the construction of a special kind of  semi-orthogonal decompositions on derived categories of parabolic sheaves. This provides in particular a   categorification of  structure theorems for the Kummer flat  K-theory of log  schemes  
due to Hagihara \cite{hagihara}   and Nizio{\l} \cite{Ni1}. Additionally we generalize Hagihara  and Nizio{\l}'s result in two ways:
\begin{itemize}
\item  we obtain  uniform structure theorems which hold across all (Kummer flat) invariants of logarithmic  schemes, including Hochschild 
and cyclic homology;
\item our techniques  allow us to extend these 
results  to a much larger class of log  schemes (and log stacks)  than those considered by Hagihara  and Nizio{\l}.
\end{itemize}
Our main results, Theorem \ref{main1}, Theorem \ref{main2} and Corollary \ref{corC}, hold over an arbitrary ground ring.

\subsection{Parabolic sheaves, log schemes, and infinite root stacks}
\emph{Parabolic sheaves} were defined by Mehta and Seshadri in the 80's, for Riemann  surfaces with marked points, as coherent sheaves equipped with flags at the marked points. They 
are the key ingredient to extend  the Narasimhan--Seshadri correspondence to the non-compact setting. The theory over the last fifty years has undergone massive generalizations. It was extended first to pairs $(X, D)$ where $D$ is a normal crossing divisor in any dimension, and more recently to an even broader class of \emph{logarithmic schemes}. The work of Borne, Vistoli and the third author shows that parabolic structures are best viewed within the framework of log geometry \cite{borne-vistoli}, \cite{TV}, 
\cite{talpo2014moduli}  and this is the perspective that we will adopt throughout the paper.

Logarithmic (log)  geometry emerged in the 80's through the collective efforts of several authors  including Deligne, Faltings, Fontaine, Illusie and Kato \cite{kato}. The theory was initially designed 
for applications to arithmetic geometry, but over the last twenty years it has become  a key organizing principle in areas as diverse as algebraic geometry, symplectic geometry,   and homotopy theory. 
Log geometric techniques lie at the core of the Gross--Siebert program  in mirror symmetry \cite{gross2006mirror}, and feature prominently in recent approaches to 
Gromov-Witten theory, see \cite{GS} and references therein. 

One of the main hurdles in working with log schemes is that they encode both classical geometric 
and combinatorial data. For this reason 
transporting familiar geometric constructions to the log setting is often delicate: see for instance \cite{Ol}, \cite{hagihara2016structure} for the definition of the cotangent complex and the Chow groups of log schemes. A definition of K-theory for log schemes  was first proposed by Hagihara \cite{hagihara}   and Nizio{\l} \cite{Ni1}. We will refer to it as  \emph{Kummer flat (resp. \'etale)  K-theory}, and it is  the algebraic K-theory of the Kummer flat (resp.  \'etale) topos,  a logarithmic analogue  
of the classical flat (resp.  \'etale) topos.

Our main result is a construction of infinite   semi-orthogonal decompositions    on categories of parabolic sheaves. 
This can be viewed as a categorification of an important structure theorem due to Hagihara and Nizio{\l} for Kummer flat K-theory: if $X$ is a regular scheme   equipped with a simple normal crossings divisor $D \subset X$,   the Kummer flat   K-theory of $(X,D)$ splits as an an explicit  direct sum  indexed by the strata of $D$ \cite[Theorem 1.1]{Ni1}. We will show that our methods yield, in particular, substantial generalizations of Hagihara and Nizio{\l}'s results. Before stating our main result we review its two key ingredients: the \emph{infinite root stack}, and \emph{semi-orthogonal decompositions}.

\subsubsection*{\textbf{Infinite root stacks}} 
 The infinite root stack of a log scheme 
 was 
introduced in \cite{TV}: it is a limit of tame Artin 
stacks (Deligne--Mumford in characteristic $0$) which encodes 
log information as stacky data.  
The infinite root stack captures the geometry of 
the underlying log scheme, and this point of view informs 
several recent works by the authors and their collaborators  
\cite{CSST, scherotzke2016logarithmic, TaV, talpo2017parabolic}, 
see also  \cite{sala2017hall} for recent applications  
to Hall algebras and 
quantum groups. 
If $X$ is a log scheme, we denote 
its infinite root stack by $\radice[\infty]{X}$. One of the key properties of the infinite root stack is that the Kummer flat topos of $X$ is  equivalent, as a ringed topos, to the fppf topos of  $\radice[\infty]{X}$ \cite[Theorem 6.16]{TV}.  In particular  Hagihara and Nizio{\l}'s logarithmic algebraic K-theory of $X$ coincides with the ordinary algebraic K-theory of the infinite root stack $\radice[\infty]{X}.$ Thus we can study logarithmic algebraic K-theory by probing the geometry and the sheaf theory of infinite root stacks.  

\subsubsection*{\textbf{Semi-orthogonal decompositions}}
We view the algebraic K-theory of a stack as an invariant of its $\infty$-category of perfect complexes.  We study the category of perfect complexes of the infinite root stack, and show that Nizio{\l}'s direct sum decomposition of  Kummer flat K-theory   is the shadow of a factorization that holds directly at the categorical level. The appropriate concept of factorization for categories 
is given by semi-orthogonal decompositions (\emph{sod}-s, for short). These were   
introduced by Bondal and Orlov in \cite{bondal1990representable}.   
In the setting of $\infty$-categories, semi-orthogonal decompositions (of length two) were considered in  \cite{BGT} under the name of \emph{split-exact sequences} of 
$\infty$-categories. 
\subsubsection*{\textbf{The main theorem}}
Recent work of Ishii and Ueda \cite{ishii2011special}   and Bergh, Lunts, and Schn\"urer \cite{bergh2016geometricity} shows that the categories of perfect complexes of  finite root stacks  of pairs  $(X,D),$ where $X$ is a scheme (or stack) equipped with a simple normal crossings divisor $D \subset X,$ admit a canonical sod where the summands are labelled by the strata of $D.$   The infinite root stack is the limit of all finite root stacks, however the pull-back functors along root maps do not preserve  the canonical sod-s. This issue can be obviated via a recursion that gives rise a sequence of nested sod-s as the root index grows, and that ultimately yields an infinite sod  on $\Perf(\radice[\infty]{X})$.

Below we formulate our first main result   
for log scheme of the form $(X,D),$ where $X$ is a scheme and $D$ is a simple normal crossings divisor. We refer the reader to  Theorem \ref{mainsgncdiv} in the main text for  a sharper and  more general statement, that applies for instance to all finite type tame algebraic stacks.  Let $\{D_i\}_{i \in I}$ be the set of irreducible components of $D$. The divisor $D$ determines a  stratification of $X$ where strata are intersections of the irreducible components of $D$. Strata are in bijection with the subsets of $I$: if  $S$ is a stratum  $S=\cap_{j \in J} D_j$ for some $J \subset I$, and $\overline{S}$ is the closure, we set $|\overline{S}|:=|J|$. 
If $N$ is a natural number 
we set 
$
(\mathbb{Q}/\mathbb{Z})^{N,*} := (\mathbb{Q}/\mathbb{Z}\setminus \{0\})^N $.
\begin{customthm}{A}[Theorem \ref{mainsgncdiv}]
\label{main1}
The $\infty$-category of perfect complexes 
of the infinite root stack 
$\radice[\infty]{(X,D)}$ admits a semi-orthogonal decomposition 
$$
\Perf(\radice[\infty]{(X,D)}) = \langle   \cA_{\overline S},   \overline{S} \in S_D   \rangle 
$$
such that all objects in $ \cA_{\overline S}$ are supported on 
$\overline S$. 
Additionally, for all $\overline S$ the category $\cA_{\overline S}$ carries a semi-orthogonal decomposition 
indexed by $(\mathbb{Q}/\mathbb{Z})^{|\overline{S}|,*}$ whose factors are equivalent to $\Perf(\overline{S}).$ 
\end{customthm}

\subsection{\textbf{Additive invariants}}
Theorem \ref{main1} recovers in particular Hagihara and Nizio{\l}'s results, but is a much stronger statement. In order to clarify this point let us refer to the notion of   \emph{additive  invariants} of $\infty$-categories. Let us denote by  $\SCat$ the $\infty$-category of stable $\infty$-categories.  
A functor 
$ 
\mathrm{H}\colon \SCat \to \cP 
 $ , where $\cP$ is a stable presentable $\infty$-category,
is an \emph{additive invariant} if it preserves zero objects and filtered colimits, and it maps \emph{split exact sequences} to cofiber sequences (split exact sequences are the analogue in the $\infty$-setting of a  sod with two factors). Most homological invariants of algebras and categories are  additive: algebraic K-theory and non-connective K-theory, (topological) Hochschild homology and negative cyclic homology are all additive invariants.

The theory of \emph{non-commutative motives}  was developed by Tabuada  and others \cite{tabuada2008, cisinski2011non, BGT, robalo2015k} in analogy with the classical theory of motives. Non-commutative (additive) motives encode the universal additive invariant, exactly as classical motives are universal among Weil cohomologies. Noncommutative motives form a  presentable and stable $\infty$-category 
$ \,  
\mathrm{Mot}^{\mathrm{add}}$  
which is  the recipient  of  
the universal additive invariant  
$$
\cU\colon \SCat \longrightarrow \mathrm{Mot}^{\mathrm{add}}.      
$$
Every additive invariant $\, \mathrm{H}\colon \SCat \to \cP$  factors uniquely as a composition $$
 \xymatrix{
 \SCat \ar[r]^-{\mathrm{H}} \ar[d]_-{\cU} & \cP \\ 
 \mathrm{Mot}^{\mathrm{add}} \ar[ur]_-{\mathrm{\overline H}}
 }
 $$
If $\mathrm{H}$ is an additive  invariant  and $X$ is a stack  we set 
$ \,  
\cU(X):=\cU(\Perf(X)), 
\;
\mathrm{H}(X):=\mathrm{H}(\Perf(X)) $. 
Let $(X,D)$ be a scheme equipped with a normal crossings divisor. Then Theorem \ref{main1} yields   a 
direct product decomposition of the  non-commutative motive of $\radice[\infty]{X}$.
\begin{customthm}{B}[Corollary \ref{ncmotdirectsum}]
\label{main2}
There is a canonical direct sum decomposition  
$$
\cU (\radice[\infty]{(X,D)}) \simeq 
\bigoplus_{\overline{S} \in S_D} \Big (
\bigoplus_{(\mathbb{Q}/\mathbb{Z})^{|\overline{S}|,*}} \cU(\overline{S}) \Big ).
$$
\end{customthm}
The K-theory of root stacks was also studied in \cite{dhillon2015g}, although from a different perspective. 
\subsubsection*{\textbf{Kummer \'etale K-theory}}
Theorem \ref{main2} implies uniform direct product decompositions across all  additive invariants   
of infinite root stacks, and recovers in particular  Hagihara and Nizio{\l}'s structure theorems for Kummer flat K-theory.   
Let $X$ be a log scheme, and denote $X_{\mathrm{Kfl}}$ 
its Kummer flat topos. Let $\Perf(X_{\mathrm{Kfl}})$ be its $\infty$-category of perfect complexes. 
If $\mathrm{H}$ is an additive invariant, we set
$ \, 
\mathrm{H}_{\mathrm{Kfl}}(X):= \mathrm{H}(\Perf(X_{\mathrm{Kfl}})) \,$.  When $\mathrm{H}(-) = K(-)$ is  algebraic K-theory, this definition  recovers  Hagihara and Nizio{\l}'s Kummer flat K-theory. 

Work of Vistoli and the third author \cite{TV} 
identifies the Kummer 
flat topos with the ``small fppf topos'' of the infinite root stack. As a consequence, under suitable assumptions, there is an equivalence of stable $\infty$-categories
$
\Perf(X_{\mathrm{Kfl}}) \simeq \Perf(\radice[\infty]{X}). 
$
This, together with 
 Theorem \ref{main2}, yields the following immediate  
corollary. 
\begin{customcor}{C}
\label{corC}
If $\mathrm{H}\colon \SCat \to \cP$ is an additive  invariant then there is a direct sum decomposition 
$$
\mathrm{H}_{\mathrm{Kfl}}(X,D) \simeq 
\bigoplus_{\overline{S} \in S_D} \Big (
\bigoplus_{(\mathbb{Q}/\mathbb{Z})^{|\overline{S}|,*}} \mathrm{H}(\overline{S}) \Big ). 
$$
In particular, the Kummer flat  
K-theory of $(X,D)$ decomposes as a direct sum of spectra 
$$
K_{\mathrm{Kfl}}(X,D) \simeq 
\bigoplus_{\overline{S} \in S_D} \Big (
\bigoplus_{(\mathbb{Q}/\mathbb{Z})^{|\overline{S}|,*}} K(\overline{S}) \Big ). 
$$
\end{customcor}
The second half of the statement recovers the first part of Nizio{\l}'s  \cite[Theorem 1.1]{Ni1} (see Remark \ref{rmk:ket} for some comments about the second part). Nizio{\l}'s result holds under the  restrictive assumption that 
$(X, D)$ is a log smooth pair given by a regular scheme $X$ and a simple normal crossings divisor $D$. Corollary \ref{corC} holds with milder   smoothness assumption on $X$: indeed in the main body of the paper we work with a finite type algebraic stacks $X$  equipped with a simple normal crossings divisor $D$. In particular $X$ needs not be  regular outside of $D$.

In fact  we can  extend the the decomposition  given by Corollary \ref{corC} to even more general log stacks. We clarify this by  explaining three applications of our techniques. They  require working over a field $\kappa$ of characteristic zero. Additionally for 
the second one we need  to assume  $\kappa=\mathbb{C}$. 
\begin{itemize}[leftmargin=*]
\item \textbf{\emph{General normal crossing divisors}}. We extend the decomposition of Theorem \ref{main1} and Corollary \ref{corC} to \emph{general normal crossing divisors}, removing the simplicity assumption required by Hagihara and Nizio{\l}. An 
interesting new feature emerges in this setting. The  semi-orthogonal summands appearing in the analogue of Theorem \ref{main1} for general normal crossing log stacks $(X,D)$ are no longer equivalent to  the category of perfect complexes on the strata: instead, 
they are equivalent to perfect complexes on the \emph{normalization  of the strata} (see Theorem \ref{mainncdivsod} in the main text).   This is reflected by the   Kummer flat K-theory of a general normal crossing log stack: it  breaks up as a direct sum indexed by the strata, where the summands compute the K-theory of the normalization of the strata (see Theorem \ref{ncmotdirectsum1}).

\item \textbf{\emph{Simplicial log structures}}. We  generalize Corollary \ref{corC} to log smooth schemes with \emph{simplicial log structure}, i.e. pairs 
$(X, D)$ where $D$ 
is a divisor with \emph{simplicial singularities}.   
The decomposition formula holds only for the complexification of $K_{\mathrm{Kfl}}(X),$ and under suitable additional  assumptions on $(X, D).$ 
The main differences with the normal crossings case  
is that the formula  has additional summands keeping track of the singularities of $D$, and that   it depends on the \emph{G-theory}, rather than  the K-theory, of the strata. As the statement is somewhat technical   we do not include it in this  introduction, but refer the reader directly to Proposition \ref{propKKalg} in the main text.  
\item \textbf{\emph{Logarithmic Chern character.}} Having at our disposal a general definition of additive invariants of log schemes, we  introduce a construction of the \emph{logarithmic Chern character}. The availability of structure theorems valid across all additive invariants of log schemes allows us to  study some of its fundamental properties. This includes a 
Grothendieck--Riemann--Roch statement, which  we establish in the restrictive setting of strict maps of log schemes, leaving generalizations to future work. 
These results are contained in Section \ref{logchchar} of the text.  
\end{itemize} 

\subsection{\textbf{Towards logarithmic DT invariants}}
Donaldson--Thomas invariants are part of the rich array of enumerative invariants inspired by string theory.  One of the outstanding open questions in the area is to construct  a theory of log DT invariants, analogous to   the theory of log GW invariants developed in \cite{GS, C, AC}. This would have applications to  degeneration formulas  for DT invariants.  DT theory counts \emph{Bridgeland stable} objects in triangulated categories.  Thus building a theory of log DT invariants requires, first,  
 to introduce a viable concept of  derived category for log schemes; 
and, second, to define and study stability conditions over it. 
 For the first requirement, one of the viable options is to try to use parabolic sheaves on  $(X,D)$, which in turn are equivalent to sheaves on  $\radice[\infty]{X}$. 
Thus a first step towards defining log DT invariants consists in constructing \emph{Bridgeland stability condition} on $\Db{\radice[\infty]{X}}$. Results obtained in \cite{collins2010gluing} give a means to  glue  stability conditions  across semi-orthogonal decompositions. Adapting these techniques to the 
sod-s on $\Db{\radice[\infty]{X}}$ obtained in  Theorem \ref{main1}, we can already obtain stability conditions in some cases, such as many  toric log pairs $(X,D)$.  
It is too early to tell wether they will be relevant from the viewpoint of log DT theory, but this seems an interesting avenue for future investigation. 

\subsection{Relation to work of other authors}
Several different approaches to the definition of log motives 
and log invariants have been considered in the literature.
\begin{itemize}[leftmargin=*]
\item A definition of log motives has been proposed in  \cite{ito2017log}, and in  \cite{howell}. 
\item Constructions of Hochschild homology and topological Hochschild homology in the log setting have been proposed by  Hesselholt and Madsen \cite{hesselholt2003k} Rognes, Sagave, and Schlichtkrull \cite{rognes2015localization}, Leip \cite{leip} and Olsson \cite{ollsonhh}.
\end{itemize}
 It would be very interesting to compare these approaches to the one  we pursue in this paper, however there is a key difference in perspective between 
 some of these works and our own. The  constructions of log 
 Hochschild homology considered in \cite{hesselholt2003k}, \cite{rognes2015localization}, and \cite{leip} are closely related to the log de Rham complex, and therefore to the cohomology of the \emph{complement} of the locus where the log structure is concentrated. In this paper, on the other hand, we investigate log schemes through the lenses of their \emph{Kummer flat} topos and their infinite root stack.

\subsection*{Acknowledgments:} We thank Daniel Bergh, David Carchedi, Olaf Schn\"urer, Greg Stevenson and Angelo Vistoli for their interest in this project and for useful exchanges, and the anonymous referee for useful comments. M.T. was partially supported by EPSRC grant EP/R013349/1.

\subsection*{Conventions}

We work over an arbitrary noetherian commutative ring $\kappa$ (that could be the ring of integers $\bZ$). In later parts of the paper (Sections \ref{sec:simplicial.applications} and \ref{logchchar}) we will impose more restrictive assumptions on $\kappa$. All algebraic stacks (in the sense of \cite[Tag 026O]{stacks-project}) will be of finite type over $\kappa$. All monoids will be commutative and ``toric'', i.e. finitely generated, sharp, integral and saturated.

\section{Preliminaries}
\label{sec:outline}

\subsection{Log structures from boundary divisors and root stacks}\label{sec:logstr}

In this section we briefly recall how certain boundary divisors $D$ on a scheme or algebraic stack $X$ give rise to log structures and root stacks. 

 \begin{definition}\label{def:nc}
Let $X$ be a scheme over $\kappa$, and $D\subset X$ an effective Cartier divisor. Recall that the divisor $D$ is: \begin{itemize}
\item \emph{simple normal crossings} if for every $x\in D$ the local ring $\cO_{X,x}$ is regular, and there is a system of parameters $a_1,\hdots, a_n \in \cO_{X,x}$ such that the ideal of $D$ in $\cO_{X,x}$ is generated by $a_1,\hdots, a_k$ for some $1\leq k\leq n$, and
\item \emph{normal crossings} if every $x\in D$ has an \'etale neighbourhood where $D$ becomes simple normal crossings.
\end{itemize}
 \end{definition}
 
 Note that if $D$ is a normal crossings divisor on $X$, then every point of $D$ is a regular point of $X$, but away from $D$ the scheme $X$ can well be singular.

If $X$ is an algebraic stack and $D\subset X$ is an effective Cartier divisor, we say that $D$ is {(simple) normal crossings} if the pull-back of $D$ to some smooth presentation $U\to X$, where $U$ is a scheme, is a (simple) normal crossings divisor on $U$.

\begin{remark}
If $\kappa$ is a field, then the divisor $D\subset X$ is normal crossings if and only if \'etale locally around every point $x \in D$, the pair $(X,D)$ is isomorphic to the pair $(\bA^n,\{x_1\cdots x_k=0\})$ for some $n\in \bN$ and $0\leq k \leq n$.

This is not necessarily true if $\kappa$ is not a field. For example, if $\kappa=\bZ$ the divisor $V(p)\cup V(x)$ for a prime number $p$ is simple normal crossings on $X=\bA^1=\Spec \bZ[x]$, but X is not \'etale locally isomorphic to $\bA^2$ around the point $(p,x)$.
\end{remark}

Later on (Section \ref{reducetoncviacan}) we will consider a generalization of this notion, where the divisor $D$ is allowed to have \emph{simplicial singularities}. Over a field $\kappa$ one can define this by asking that \'etale locally around every point $x \in D$, the pair $(X,D)$ is isomorphic to the pair $(\Spec \kappa[P] \times \bA^n,\Delta_P\times \bA^n)$ for some simplicial monoid $P$ and $n\in \bN$. Recall that a sharp fine saturated monoid $P$ is \emph{simplicial} if the extremal rays of the rational cone $P_\bQ\subset P^\gp\otimes_\bZ \bQ$ are linearly independent, and we denote by $\Delta_P$ the toric boundary in the affine toric variety $\Spec \kappa[P]$, i.e. the complement of the torus $\Spec \kappa[P^\gp]$, equipped with the reduced subscheme structure. This weaker notion allows for some kinds of singularities along the divisor $D$ itself.

It is not clear to us how to formulate this notion in the ``absolute'' case (i.e. for schemes over $\Spec \bZ$), so we will circumvent this problem by using a canonically defined root stack of a log scheme with simplicial log structure, and reducing to Definition \ref{def:nc} on this root stack (see Definition \ref{def:simpl.sing}).

Next we recall how ((simple) normal crossings) divisors induce associated log structures and root stacks. We focus on the construction of root stacks, since the full formalism of logarithmic geometry will not play an important role in the paper. We refer the reader to \cite{ogus} for an extensive introduction on log geometry, and to the appendix of \cite{CSST} for a quick overview of the basic concepts. For more on root stacks, the reader can consult \cite{cadman, borne-vistoli, TV, talpo2014moduli}.

Any effective Cartier divisor $D$ 
in a scheme $X$ induces a log structure, usually called the \emph{compactifying log structure} of the open embedding $X\setminus D\subset X$, as follows. The sheaf
$
M_D=\{f\in \cO_X\mid f|_{X\setminus D}\in \cO_{X\setminus D}^\times\}
$
is a sheaf of submonoids of $\cO_X$ (seen as a sheaf on the small \'etale site of $X$), where the monoid operation is multiplication of regular functions. The inclusion $\alpha\colon M_D\to \cO_X$ gives rise to a log structure on $X$. If $D$ is normal crossings, this log structure admits local charts, and in fact will also be fine and saturated. Morally, the sheaf $M_D$ keeps track of how many branches of $D$ intersect at a point of $X$, and how their local equations fit together in the local ring $\cO_{X,x}$ (more precisely, in its strict henselization, since we are using the \'etale topology).
 
If $X$ is a stack rather than a scheme, the above procedure gives a log structure on any given smooth presentation of $X$, and descent for fine log structures \cite[Appendix]{Ols} gives a fine saturated log structure on $X$ itself. We will denote the resulting log scheme or stack by $(X,D)$.
 
\subsubsection{{Root stacks along a single regular divisor}}\label{sec:smooth}

Assume that $D\subset X$ is a simple normal crossings divisor with only one component, i.e. a regular connected divisor, and $X$ is an algebraic stack. Recall that the quotient stack $[\bA^1/\bG_m]$ functorially parametrizes pairs $(L,t)$ where $L$ is a line bundle and $t$ is a global section of $L$.

As specified above, the effective Cartier divisor $D$ on 
$X$ induces a canonical log structure, that in this simple case can be described as follows: the divisor $D$ determines a line bundle $\cO_X(D)$  on 
$X$ together with a section $\sigma \in \Gamma(\cO_X(D))$ having  
$D$ as its zero locus, and this yields a  tautological morphism
$
s\colon X \longrightarrow [\bA^1/ \mathbb{G}_m]. 
$
This equips $X$ with a log structure by pulling back the canonical log structure of $[\bA^1/ \mathbb{G}_m]$ (corresponding to the regular divisor $[\{0\}/\bG_m]\subset [\bA^1/\bG_m]$) via $s$. 

In this case, for $r\in \bN$ the $r$-th root stack of the pair $(X,D)$ can be described as the functor that associates to a scheme $T\to X$ over $X$ the groupoid $\radice[r]{(X,D)}(T)$ of pairs $(L,t)$ consisting of a line bundle on $T$ with a global section, and with an isomorphism $(\cO_X(D), \sigma)|_T\cong (L^{\otimes r},t^{\otimes r})$.

The stack $\radice[r]{(X,D)}$ fits in a fiber square 
\begin{equation}
\begin{gathered}
\label{rootsquare}
\xymatrix{
\radice[r]{(X,D)} \ar[r] \ar[d]_-{g_{r,1}} & [\bA^1/ \mathbb{G}_m] \ar[d]^-{(-)^ r}\\ 
X \ar[r]^-s & [\bA^1/ \mathbb{G}_m]
}
\end{gathered}
\end{equation}
where $(-)^r$ is induced by the $r$-th power maps on $\bA^1$ and on $ \mathbb{G}_m$, or equivalently is the functor sending a pair $(L,t)$ of a line bundle with a global section to the pair $(L^{\otimes r}, t^{\otimes r})$. The reason for the notation $g_{r,1}$ will be apparent later (see Section \ref{sec:irs}). The construction of the stack $\radice[r]{(X,D)}$ is compatible with respect to pull-back along smooth morphisms towards $X$ (in particular with respect to Zariski and \'etale localization on $X$), i.e. if $Y\to X$ is smooth, we have a canonical isomorphism $\radice[r]{(Y,D|_Y)}\simeq \radice[r]{(X,D)}\times_X Y$.

We denote by $D_r$ the effective Cartier divisor on $\radice[r]{(X,D)}$ obtained by taking the reduction of the 
closed substack $g_{r,1}^{-1}(D) \subset \radice[r]{(X,D)},$ and we denote 
$ 
i_r\colon D_r \to\radice[r]{(X,D)}
$ the inclusion. We will refer to $D_r$ as the \emph{universal effective Cartier divisor} on $\radice[r]{(X,D)}$, since it is the universal $r$-th root of the divisor $D$ on $X$.

Consider the commutative (non-cartesian) diagram
\begin{equation}
\begin{gathered}
\label{rootd}
\xymatrix{
\mathscr{B}\mathbb{G}_m \ar[d]_-{(-)^r} \ar[r] & [\mathbb{A}^1/\mathbb{G}_m] 
\ar[d]^-{(-)^r} \\ 
\mathscr{B}\mathbb{G}_m \ar[r] & [\mathbb{A}^1/\mathbb{G}_m] 
}
\end{gathered}
\end{equation}
where the left vertical arrow is induced by the $r$-th power map
$
(-)^r\colon \mathbb{G}_m \longrightarrow \mathbb{G}_m. 
$
  As explained in \cite{bergh2016geometricity}, $D_r$ can be defined equivalently as the top left vertex of the base change of diagram 
  (\ref{rootd}) along the tautological map 
  $s\colon X \to [\mathbb{A}^1/\mathbb{G}_m]$: in particular, there is a fiber product 
   \[ \xymatrix{ 
D_r \ar[r] \ar[d]_-{f_{r,1}} & \mathscr{B} \mathbb{G}_m \ar[d]^{(-)^r}\\
   D  \ar[r] & \mathscr{B} \mathbb{G}_m.  }
        \] 
This implies that $D_r \to D$ is a $\mu_r$-gerbe.

Zariski locally on $X$, where the line bundle $\cO_X(D)$ is trivial, the stack $\radice[r]{(X,D)}$ admits the following explicit description: assume also that $X=\Spec A$ is affine, and let $f\in A$ correspond to the section $\sigma$ of $\cO_X(D)$ (so that $D$ has equation $f=0$). Then we have an isomorphism
$$
\radice[r]{(X,D)}\simeq [\Spec ( A[x]/(x^r-f))\, /\, \mu_r]
$$
where $\mu_r$ acts by multiplication on $x$. The divisor $D_r\subset \radice[r]{(X,D)}$ is given by the global equation $x=0$, and is therefore isomorphic to $[\Spec (A/f) /\mu_r]\simeq D\times \mathscr{B}\mu_r$.

\subsubsection{Root stacks along a simple normal crossings divisor}\label{sec:preliminaries.NC}

Assume now that $D$ is a simple normal crossings divisor on $X$, and denote by $D_1,\hdots, D_N$ the irreducible components of $D$. In this case root stacks of $(X,D)$ are indexed by elements of $\bN^N$. For $\vec{r}=(r_1,\hdots, r_N)\in \bN^N$, the root stack $\radice[\vec{r}]{(X,D)}$ parametrizes tuples $((L_1,t_1),\hdots, (L_N,t_N))$, where $(L_i,t_i)$ is a $r_i$-th root of $(\cO_X(D_i),\sigma_i)$. Each pair $(\cO_X(D_i),\sigma_i)$ determines a morphism $s_i\colon X\to [\bA^1/\bG_m]$, and the stack $\radice[\vec{r}]{(X,D)}$ is the fiber product of the diagram
$$
\xymatrix{
\radice[\vec{r}]{(X,D)}\ar[d] \ar[r] & [\bA^1/ \mathbb{G}_m]^N \ar[d]^-{(-)^{\vec{r}}}\\ 
X \ar[r]^-s & [\bA^1/ \mathbb{G}_m]^N
}
$$
where $s\colon X\to [\bA^1/ \mathbb{G}_m]^N$ is determined by the $s_i$ and $(-)^{\vec{r}}$ is the map induced by $(-)^{r_i}\colon [\bA^1/\bG_m]\to [\bA^1/\bG_m]$ on the $i$-th component.

Equivalently, $\radice[\vec{r}]{(X,D)}$ can be constructed by iteration from the previous case: from $X$ we first construct the stack $\radice[r_1]{(X,D_1)}$ as in the previous section. The preimages $\widetilde{D_2}, \hdots, \widetilde{D_N}$ of $D_2,\hdots, D_N$ to this stack give a simple normal crossings divisor $\widetilde{D}$, and we can replace $(X,D)$ by the log stack $(\radice[r_1]{(X,D_1)}, \widetilde{D})$, and continue the process.

Finally, the stack $\radice[\vec{r}]{(X,D)}$ can also be identified with the fibered product of the diagram
$$
\xymatrix{
\radice[r_1]{(X,D_1)} \ar[rrd] & \radice[r_2]{(X,D_2)} 
\ar[rd] & \ldots & \radice[r_{N-1}]{(X,D_{N-1})} \ar[ld] & \radice[r_N]{(X,D_N)} 
\ar[lld] \\
&& X &&
}
$$
As in the case of one divisor, the construction of $\radice[\vec{r}]{(X,D)}$ is compatible with smooth base change. 

On the stack $\radice[\vec{r}]{(X,D)}$ there are $N$ universal effective Cartier divisors $D_{i, r_i}$, obtained as the reduction of the preimage of $D_i$ to $\radice[\vec{r}]{(X,D)}$  via the projection to $X$, or equivalently as the preimage of the corresponding divisor $D_{i, r_i}$ on $\radice[r_i]{(X,D_i)}$ constructed in the previous section (we abuse notation slightly and denote the two by the same symbol).

Zariski locally where $X=\Spec A$ is affine and each $D_i$ has a global equation $f_i=0$ we have an isomorphism
$$
\radice[\vec{r}]{(X,D)}\simeq [\Spec A[x_1,\hdots, x_N]/(f_1-x_1^{r_1},\hdots, f_N-x_N^{r_n})/\prod_i \mu_{r_i}]
$$
where each factor $\mu_{r_i}$ acts on $x_i$ by multiplication and trivially on the other variables.

\subsubsection{Root stacks for non-simple normal crossings
}\label{sec:pre.non.simple}
 
If $D$ is only normal crossings, then the ``correct'' notion of root stack of $X$ along $D$ is not the naive generalization of the simple normal crossings case, but rather 
the subtler one put forward in \cite{borne-vistoli}. We refer to that paper and to \cite[Section 2.2]{talpo2014moduli} for details about what follows.

We start by explaining with an example why the naive notion of root stacks along the lines of the previous two cases is not the one that we want to work with.

\begin{example}\label{ex:nodal.cubic}
Consider an irreducible nodal cubic  $D\subset \bP^2$, and the pair $(\bP^2,D)$. We could, as in Section \ref{sec:smooth}, start with the morphism $\bP^2\to[\bA^1/\bG_m]$ that classifies the divisor $D$ (i.e. determined by  the pair ($\cO_{\bP^2}(D),\sigma_D)$, where $\sigma_D$ is the tautological section of $\cO_{\bP^2}(D)$ cutting out $D$), and set $\radice[r]{(\bP^2,D)}$ to be the fibered product of diagram (\ref{rootsquare}).
The unpleasant feature of this construction is that it adds a stabilizer  $\mu_r$ along all of $D$, including at the node, where one would expect to  have a $\mu_r\times \mu_r$ accounting for the two branches of $D$.

The point here is that the morphism $\bP^2\to[\bA^1/\bG_m]$  does not induce to the ``correct'' logarithmic structure on $\bP^2$ corresponding to the pair $(\bP^2,D)$, which is the compactifying log structure briefly mentioned in Section \ref{sec:logstr}, and \emph{does} distinguish the branches of $D$ at the node.
\end{example}

The trouble with the previous example is remedied by working with sheaves of weights (as explained in \cite{borne-vistoli}), that allow to locally distinguish the branches of $D$ at the node, and to locally take roots separately along the branches.

We briefly recall what this entails: in general, the log structure of a fine saturated log scheme $X$ can be seen as a ``Deligne--Faltings'' structure, a symmetric monoidal functor $L\colon A\to \Div_X$ from a sheaf of saturated sharp monoids $A$ on the small \'etale site of $X$ to the symmetric monoidal stack $\Div_X$ of pairs $(L,t)$ of line bundles with global section. The monoidal operation of $\Div_X$ is given by tensor product.

In the case of \emph{simple} normal crossings, the canonical (compactifying) log structure $M_D\to \cO_X$ induced by $D$ can be completely described in the language of Deligne--Faltings structures by the symmetric monoidal functor $\bN^N\to \Div_X(X)$ sending the generator $e_i$ to the pair $(\cO_X(D_i),\sigma_i)$ (this is a ``chart'' for the log structure, in the sense of \cite[Section 3.3]{borne-vistoli}). If $D$ is only normal crossings though, the \emph{global} components of the divisor $D$ will not capture the geometry of the situation faithfully (as for the nodal cubic of Example \ref{ex:nodal.cubic}), and the compactifying log structure is \emph{not} described by the symmetric monoidal functor $\bN^N\to \Div_X(X)$ defined above.

In both cases (simple and non-simple normal crossings), the sheaf $A$ is the quotient $M_D/\cO_X^\times$ of the sheaf $M_D\subset \cO_X$ consisting of regular functions that only possibly vanish along $D$ (mentioned in Section \ref{sec:logstr}) by the subgroup $\cO_X^\times$ of invertible functions, and the symmetric monoidal functor $L\colon A\to \Div_X$ is induced by the inclusion $M_D\subseteq \cO_X$ by passing to (stacky) quotients by $\cO_X^\times$, obtaining $$A=M_D/\cO_X\to [\cO_X/\cO_X^\times]=[\bA^1/\bG_m]_X=\Div_X.$$ If $D$ has simple normal crossings, then the sheaf $A$ is isomorphic to the direct sum $\bigoplus_{i=1}^N {j_i}_*\bN_{D_i}$, where $D_i$ are the irreducible components of $D$, $j_i\colon D_i \to X$ is the inclusion and $\bN_{D_i}$ is the constant (\'etale) sheaf of monoids with stalk $\bN$ on $D_i$. The map $A\to \Div_X$ in this case sends the generator $1_{D_i}$ of the summand ${j_i}_*\bN_{D_i}$ to the pair $(\cO_X(D_i),\sigma_i)$ (where as usual $\sigma_i$ is the section that cuts out $D_i$). If $D$ is only normal crossings but non-simple, this description does not hold anymore (it only holds \'etale locally where $D$ becomes simple normal crossings). For the nodal cubic of Example \ref{ex:nodal.cubic}, the stalk of the sheaf $A$ at the node is $\bN^2$ (accounting for the two branches) rather than just $\bN$.


For a general log scheme $X$ and $r\in \bN$, the root stack $\radice[r]{X}$ (in the sense of Borne and Vistoli) parametrizes lifts of the symmetric monoidal functor $L\colon A\to \Div_X$ to a symmetric monoidal functor $\frac{1}{r}A\to \Div_X$, where we are embedding $A$ into $\frac{1}{r}A\cong A$ via the map $\cdot r\colon A\to A$. The image of the section $\frac{1}{r}a$ via this lift is an $r$-th root of the pair $L(a)=(L_a,s_a)$. In the case of the nodal cubic $D$ of Example \ref{ex:nodal.cubic}, i.e. of $\bP^2$ equipped with the compactifying log structure induced by $D$, this  results in trivial stabilizers outside of $D$, a stabilizer $\mu_r$ over points of $D$ that are not the node, and a stabilizer $\mu_r\times\mu_r$ at the node.

If $D\subset X$ is a normal crossings divisor and $U\to X$ is a surjective \'etale morphism such that the pull-back $D|_U$ is simple normal crossings, then the stack $\radice[{r}]{(X,D)}$ can be obtained from the root stack $\radice[\vec{r}]{(U, D|_U)}$ constructed in the previous section (where $\vec{r}$ is the vector $(r,\hdots, r)$) by descent.

Moreover, In place of $\frac{1}{r}A$ one can use an arbitrary \emph{Kummer extension} $A\to B$ of sheaves of monoids, an injective morphism such that every section of $B$ locally has a multiple coming from $A$. This more general construction of the root stack $\radice[B]{X}$  for Kummer extensions $A\to  B$, where $B$ is not of the form  $\frac{1}{r}A$, will come up only in Section \ref{reducetoncviacan}. In general, the stack $\radice[B]{X}$ is a tame Artin stack (Deligne--Mumford in characteristic $0$), and the projection $\radice[B]{X}\to X$ is a coarse moduli space morphism.

\subsubsection{The infinite root stack and the Kummer flat topos}\label{sec:irs}

In all the cases considered above, the various root stacks of $(X,D)$ form an inverse system. Let us temporarily use the letter $r$ to denote either a natural number, or a vector of natural numbers $\vec{r}\in \bN^N$, depending on the context that we are considering. We write $r\mid r'$ to denote divisibility in the first case, and that $r_1 \mid r_1', \ldots, r_N \mid r_N'$ in the second case. We also write   
$\frac{r'}{r}$ for the vector $\left(\frac{r'_1}{r_1}, \ldots, \frac{r'_N}{r_N}\right)$.

With these conventions, if $(X,D)$ is a pair where $D$ is normal crossings and $r\mid r'$, there is a natural projection $g_{r',r}\colon \radice[r']{(X,D)}\to \radice[r]{(X,D)}$, roughly defined by raising the roots parametrized by the source stack to the $r'/r$-th power. Since $\radice[1]{(X,D)}\simeq X$, the map $g_{r,1}\colon \radice[r]{(X,D)}\to X$ is the natural projection. If $D$ is a regular divisor, then the maps $g_{r',r}$ restrict to maps $f_{r',r}\colon D_{r'}\to D_r$ between the universal Cartier divisors on the two stacks (and analogously in the case where $D$ is simple normal crossings).

Moreover, $g_{r',r}$ is a ``relative root stack'' morphism, in the sense that for a morphism $T\to \radice[r]{(X,D)}$ where $T$ is a scheme, the pull-back of $g_{r',r}$ can be seen as the projection from a root stack of the scheme $T$. Consequently, the maps $g_{r',r}$ have all the properties of a projection to a coarse moduli space of a tame algebraic stack. More precisely, $\radice[r']{(X,D)}$ is  canonically isomorphic to the $\frac{r'}{r}$-th root stack of $\radice[r]{(X,D)}.$ Under this isomorphism $g_{r',r}$ is identified with the projection
$$
\radice[r']{(X,D)} \simeq \radice[\frac{r'}{r}]{(\radice[r]{(X,D)},D_r)} \to \radice[r]{(X,D)}. 
$$
The maps $g_{r',r}$ equip the stacks $\{\radice[r]{(X,D)}\}_r$ with the structure of an inverse system. The inverse limit
$
\radice[\infty]{(X,D)}:=\varprojlim_r \radice[r]{(X,D)}
$
is the \emph{infinite root stack} of $(X,D)$ \cite{TV}. This, contrarily to the finite root stacks, is not algebraic, but has a local description as a quotient stack, that allows some control over quasi-coherent sheaves (and in particular perfect complexes) on it.

It will be convenient for us to work with a directed subsystem of root stacks which is cofinal in $\{\radice[r]{(X,D)}\}_r.$ Namely we will consider the subsystem $\{\radice[n!]{(X,D)}\}_{n \in \mathbb{N}},$ where $\vec{n}! :=(n !, \hdots, n !)$ if $D$ has more than one component. Note that the restriction of the ordering given by divisibility on $\bN$ to the subset $\{n!\}_{n\in \bN}\subseteq \bN$ coincides with the usual ordering of the naturals (i.e. $r!\mid s!$ if and only if $r\leq s$), and that this subsystem is cofinal, since given an index $\vec{r}=(r_1,\hdots, r_N)$, we have $\vec{r} \mid \vec{M}!$ with $M=\mathrm{lcm}(r_i)$. Therefore we have a canonical isomorphism $\radice[\infty]{(X,D)} \simeq \varprojlim_{n} \radice[n!]{(X,D)}.$

Let us also recall that every log scheme $X$ has an associated Kummer flat topos $X_\mathrm{Kfl}$. We omit a discussion of this construction, since for the present paper it can be safely taken as a black box. We refer the reader to \cite[Section 6.2]{TV} or \cite[Section 2]{Ni1} for the definition and basic properties. One can also define a ``small fppf site'' of the infinite root stack $\radice[\infty]{X}$, and as proven in \cite[Theorem 6.16]{TV}, the resulting topos $\radice[\infty]{X}_\mathrm{fppf}$ is isomorphic  to the Kummer flat topos $X_\mathrm{Kfl}$ of the fine saturated log scheme $X$.
 
 We will also consider the Kummer flat topos $X_\mathrm{Kfl}$ for $X$ a log algebraic stack. To the best of our knowledge this has not been discussed in the literature before. The conscientious reader can take the equivalence with the small fppf topos of the infinite root stack $\radice[\infty]{X}$ as a \emph{definition} for $X_\mathrm{Kfl}$ in this setting. One can also write down a definition for the Kummer flat topos in analogy with the one for schemes, and use  \cite[Theorem 6.16]{TV} to prove that this is indeed equivalent to the small fppf topos of the infinite root stack.

\subsection{ $\infty$-categories and categories of sheaves}
\label{sec:dgcategories}
Throughout the paper we 
will use the formalism of 
$(\infty,1)$-categories, the standard reference is  Lurie's work \cite{lurie2009higher, Lu2}. 
The main reason for working with $\infty$-categories is that additive invariants, and in particular algebraic K-theory, cannot be computed from the underlying  triangulated categories  alone: they are not invariants of triangulated categories but rather of their enhancements. We will work with \emph{stable $(\infty, 1)$-categories} as an enhancement of triangulated categories.  We will need very little from the theory of $\infty$-categories, and the reader could replace stable $(\infty, 1)$-categories with $\kappa$-linear dg categories throughout.   

From now on we will refer to $(\infty,1)$-categories just as $\infty$-categories. 
If $\cC$ is an $\infty$-category, and $A$ and $B$ are objects in $\cC,$ we denote by $\Hom_\cC(A, B)$ the \emph{mapping space}  between $A$ and $B$. All limits and colimits of $\infty$-categories appearing in the paper 
are to be understood in the $\infty$-categorical sense.  We  say that a diagram of $\infty$-categories 
$$
\xymatrix{
\cC_1 \ar[r]^-{F} \ar[d]_-{G} \ar[d] & \cC_2 \ar[d]^-{H} \\
\cC_3 \ar[r]^-{K} & \cC_4}
$$
is \emph{commutative} if there 
is a natural transformation $\alpha\colon HF \Rightarrow KG$ which is an equivalence when passing to homotopy categories. 

We will be mostly interested in \emph{stable} \emph{idempotent-complete} $\infty$-categories. Stable $\infty$-categories are an enhancement of triangulated categories: if $\cC$ is a stable $\infty$-category its homotopy category  $\mathrm{Ho}(\cC)$ is  triangulated. An important source of stable $\infty$-categories are the $\infty$-categorical derived categories of Grothendieck abelian categories, see \cite{Lu2} Definition 1.3.5.8. We refer the reader to Section 2.2 of \cite{BGT} for a summary of the theory of stable $\infty$-categories which contains  most of the facts that we will need   in this paper. Small  stable idempotent-complete $\infty$-categories form a presentable  
$\infty$-category which is denoted $\SCat.$  In particular $\SCat$ has all small limits and colimits. As explained in \cite{BGT} there is a well-defined notion of tensor product of  stable idempotent-complete   $\infty$-categories. This endows $\SCat$ with a symmetric monoidal structure. 
If $\kappa$ is a commutative ring we denote by $\Perf(\kappa)$ the symmetric monoidal stable $\infty$-category of perfect $\kappa$-modules: if $\kappa-mod$ is the abelian category of 
$\kappa$-modules, $\Perf(\kappa)$ is the subcategory of  compact objects in the the $\infty$-categorical derived category $\mathrm{D}(\kappa-mod)$ (see Definition 1.3.5.8 of \cite{Lu2}). 
  We denote 
$\mathrm{Cat}_{\infty, \kappa}^{\mathrm{perf}}$ 
the symmetric monoidal $\infty$-category of idempotent-complete stable $\infty$-categories tensored over $\Perf(\kappa)$,  see Section 4.1 of \cite{HSS} for more details on this construction. We will refer to objects in $\mathrm{Cat}_{\infty, \kappa}^{\mathrm{perf}}$ as \emph{$\kappa$-linear}  $\infty$-categories. 

For later reference we recall from \cite{R} a general result on colimits of  $\infty$-categories.
\begin{proposition}\label{colim} 
Let $I$ be a filtered category, $\{\cC_i\}_{i \in I}$ be a filtered system of $\infty$-categories and assume that all the structure maps $\alpha_{i \to j}\colon\cC_i \to \cC_j$ are fully faithful. 
Then the colimit $\cC := \varinjlim \cC_i$ is the  $\infty$-category  with
\begin{itemize}
\item  objects given by 
the union $\bigcup_{i\in I} \mathrm{Ob}(\cC_i)$, and 
\item the Hom complex between  $A_i \in \cC_i$ and  $A_j\in \cC_j$ is given  by 
$$\Hom_\cC(A_i, A_j) = \Hom_{\cC_l} (\alpha_{l \to i}(A_i),\alpha_{l \to j}(A_j))$$ where $l$ is any object of $I$ that is the source of morphisms $ j \leftarrow l \rightarrow i$. 
\end{itemize}
\end{proposition}
   
We will be  working with stable  categories of quasi-coherent sheaves on schemes and stacks. A survey of all basic facts and definitions on $\infty$-categories of quasi-coherent sheaves can be found in Section 2 and 3 of \cite{BFN}. Let $X$ be a  stack. We denote:
\begin{itemize}
\item by $\mathrm{qcoh}(X)$ the abelian category of quasi-coherent sheaves on $X$,    
and by $\Qcoh(X),$ the 
 stable $\infty$-category of quasi-coherent sheaves on $X$;
 \item by $\mathrm{coh}(X)$ the abelian category of coherent sheaves on $X$,    
and by $\Coh(X)$ the 
 stable $\infty$-category of coherent sheaves on $X$.
 \end{itemize}
 \begin{remark}
 \label{GR}
Recall  from \cite{BFN} that $\Qcoh(X)$ is  defined as  the limit
 $$
 \Qcoh(X) := \varprojlim_{\mathrm{Spec}(A) \to X}  \Qcoh(\mathrm{Spec}(A))
 $$
 where $\mathrm{Spec}(A) \to X$ runs over all maps from derived affine schemes with target $X$.  In general $\Qcoh(X)$ is different from the $\infty$-categorical derived category $\mathrm{D}(\mathrm{qcoh}(X))$.  By \cite[Proposition 2.4.3]{gaitsgory2017study},  however, if $X$ is a quasi-compact and quasi-separated  classical  Artin stack the respective 
  subcategories of objects that are \emph{bounded below} are equivalent
  $$ 
  \Qcoh(X)^+ \simeq \mathrm{D}(\mathrm{qcoh}(X))^+$$  
 \end{remark}
 
The tensor product of quasi-coherent sheaves equips  $\Qcoh(X)$ with a symmetric monoidal structure. We define $\Db{X}$, the $\infty$-category of perfect complexes,  as the full subcategory of dualizable objects in $\Qcoh(X)$. As proved in Proposition 3.6 of \cite{BFN} this is equivalent to  the ordinary definition of perfect complexes as objects that are locally equivalent to complexes of vector bundles.  
 Let $I$ be a small cofiltered category, and let $\{X_i\}_{i\in I}$ be a pro-object in stacks. 
\begin{definition}
\label{properf}
We set $\Db{\{X_i\}_{i\in I}}:=\varinjlim_i \Db{X_i}$ as an $\infty$-category.
\end{definition}
 
Let $X$ be a log scheme. We will apply Definition \ref{properf} 
to the pro-object in stacks 
given by the root stacks of 
$X$ together with the root maps between them, $\{\radice[r]{X}\}_{r \in \mathbb{N}}.$ 
As we prove in Proposition 2.25 of \cite{scherotzke2016logarithmic}, which we recall below, under appropriate assumptions there is an equivalence  
$\Db{\radice[\infty]{X}}\simeq \varinjlim_r \Db{\radice[r]{X}}.$ Although 
in \cite{scherotzke2016logarithmic} 
Proposition 2.25 was stated for the dg categories of perfect complexes,  the proof given there works without variations for $\infty$-categories of perfect complexes over an arbitrary ground ring.

\begin{proposition}[{\cite[Proposition 2.25]{scherotzke2016logarithmic}}]\label{prop.db.irs}
Let $X$ be a noetherian fine saturated log algebraic stack with locally free log structure over $\kappa$. Then there is an equivalence of $\infty$-categories
$$
\Db{\radice[\infty]{X}}\simeq \varinjlim_r \Db{\radice[r]{X}}.
$$
\end{proposition}

\begin{proposition}
\label{kflinf}
Let $X$ be a noetherian fine saturated log algebraic stack with locally free log structure over $\kappa$ (this holds in particular if the log structure comes from a normal crossings divisor). Then there is an equivalence of $\infty$-categories
$
\Perf(X_{\mathrm{Kfl}}) \simeq \Db{\radice[\infty]{X}}.
$
\end{proposition}
\begin{proof}
The proof is very similar to the proof of 
Proposition \ref{prop.db.irs} given in \cite{scherotzke2016logarithmic}. By  Corollary 6.17 of \cite{TV} 
there is an equivalence of abelian categories  $
\mathrm{coh}(X_{\mathrm{Kfl}}) \simeq \mathrm{coh}(\radice[\infty]{X}),
$
which yields an equivalence between the stable 
$\infty$-categories 
$
 \Coh(X_{\mathrm{Kfl}})   \simeq 
 \Coh(\radice[\infty]{X})  . 
$
Also, the structure sheaves $\cO_{X_{\mathrm{Kfl}}}$ 
and $\cO_{\radice[\infty]{X}}$ are coherent, see Proposition 4.9 of \cite{TV}. Thus, in the terminology of 
\cite[Section 1.5]{Ga}, the ringed topoi $X_{\mathrm{Kfl}}$ and $\radice[\infty]{X}$ are eventually coconnective. As in the proof 
of  Proposition 1.5.3 in \cite{Ga}, this implies that there are fully-faithful inclusions
$$
\Perf(X_{\mathrm{Kfl}}) \subseteq 
 \Coh(X_{\mathrm{Kfl}}), \quad \Perf(  \radice[\infty]{X}) \subseteq 
 \Coh(\radice[\infty]{X}).
$$ Further, the categories of perfect complexes are the full subcategories of dualizable objects. 
In formulas, we can write 
$$
\Perf(X_{\mathrm{Kfl}}) \simeq  
 \Coh(X_{\mathrm{Kfl}}) ^{\mathrm{dual}}, \quad 
\Perf(\radice[\infty]{X}) \simeq  
 \Coh(\radice[\infty]{X}) ^{\mathrm{dual}}. 
$$
We obtain the following chain of equivalences, which implies the statement  
$$
\Perf(X_{\mathrm{Kfl}}) \simeq  
 \Coh(X_{\mathrm{Kfl}})  ^{\mathrm{dual}} \simeq \Coh(\radice[\infty]{X})  ^{\mathrm{dual}} 
\simeq  
\Perf(\radice[\infty]{X}).
$$
 \end{proof}

\begin{remark}
Let $X$ be a noetherian fine saturated log algebraic stack with locally free log structure over $\kappa$. In  this paper we take $\Perf(\radice[\infty]{X})$ as our chosen model for the category of perfect complexes over the log scheme $X$. By Proposition \ref{kflinf} this coincides with the category of perfect complexes on the Kummer flat topos of $X$ and is therefore well suited to our applications to Kummer flat invariants of $X$.  
However we could have obtained an a priori different model of perfect complexes on $X$ by working with the abelian category of quasi-coherent parabolic sheaves with rational weights on $X$, 
$\mathrm{Par}(X)$, and then taking compact objects inside its $\infty$-categorical derived category $\mathrm{D}(\mathrm{Par}(X))$. It turns out that this actually gives an equivalent category. Let us  briefly explain why.
\begin{itemize}
\item By  Theorem 7.3 of \cite{TV} the abelian categories $\mathrm{Par}(X)$ and $\mathrm{qcoh}(\radice[\infty]{X})$ are equivalent.
\item  From Remark \ref{GR} and Proposition \ref{prop.db.irs} it easily follows that the subcategories of bounded-below objects in $\mathrm{D}(\mathrm{Par}(X)) \simeq \mathrm{D}( \mathrm{qcoh}(\radice[\infty]{X}))$ and in 
$\mathrm{QCoh}(\radice[\infty]{X})$ are equivalent. Since compact objects are bounded, this implies that the subcategory of compact objects in 
$\mathrm{D}(\mathrm{Par}(X))$ coincides with $\Perf(\radice[\infty]{X})$, as desired. 
\end{itemize} 
\end{remark}

\subsection{Exact sequences of $\infty$-categories}
\label{esoinfca}
 Let $\cC$ be a stable $\infty$-category.  We say that two objects $A$ and $A'$ are \emph{equivalent} if there is a   map $A\to A'$ that becomes an isomorphism in the homotopy category of $\cC$.  If $\iota\colon \cC' \to \cC$ is a fully faithful functor, we often view $\cC'$  as a subcategory of $\cC$: accordingly, we will usually denote the image under  $\iota$ of an object $A$ of $\cC'$  simply by $A$ rather than $\iota(A).$ 
We will always assume that subcategories are closed under equivalence. That is, if $\cC'$ is a full subcategory of $\cC,$ $A$ is an object of $\cC',$ and $A'$ is an object of $\cC$ which is equivalent to $A,$ we will always assume that $A'$ lies in $\cC'$ as well.

Recall that if $\cD$ is a full subcategory of $\cC,$ $(\cD)^\bot$ denotes the \emph{right orthogonal} of $\cD$, i.e. the full subcategory of $\cC$ consisting of the objects $B$ such that the Hom-space $\Hom_\cC(A,B)$ is contractible for every object $A \in \cD$. 
Let 
$\{\cC_1, \ldots, \cC_n\}$ be a finite collection of stable subcategories of $\cC$ such that, for all $1 \leq i \leq n-1,$ $\, \cC_{i} \subseteq (\cC_{i+1})^\bot.$ 
Then we denote  by 
$
\langle \cC_1  , \ldots , \cC_n \rangle
$
 the smallest 
 stable subcategory 
of $\cC$ containing all the  subcategories $\cC_i$. 

An \emph{exact sequence} of stable $\infty$-categories is a sequence  
\begin{equation}
\label{dgverdier}
\cA \stackrel{F} \longrightarrow \cB \stackrel{G} \longrightarrow \cC \end{equation}
which is both a fiber and a cofiber sequence in $\SCat.$  This concept  captures  the classical notion of Verdier localization of triangulated categories in the setting of $\infty$-categories: as shown in Section 5.1 of \cite{BGT}, (\ref{dgverdier}) is an exact sequence  if and only if 
the sequence of homotopy categories
$$
\mathrm{Ho}(\cA) \xrightarrow{\mathrm{Ho}(F)}  
\mathrm{Ho}(\cB) 
\xrightarrow{\mathrm{Ho}(G)} 
\mathrm{Ho}(\cC)
$$  
is a classical Verdier localization of triangulated categories (up to idempotent-completion of $\mathrm{Ho}(\cC)$). This implies in particular that the fully-faithfulness of a functor 
between stable $\infty$-categories can be checked at the level of homotopy categories.

The functor $F$ admits a right adjoint $F^R$ exactly if $G$ admits a right adjoint  $G^R$, and similarly for left adjoints. This is proved in   \cite[Proposition 4.9.1]{krause2010localization} for triangulated categories but the extension to $\infty$-categories is straightforward. If $F$ (or equivalently $G$) admits a right adjoint we say that (\ref{dgverdier}) is a \emph{split  exact sequence}. In this case the functor $G^R$ is fully faithful 
and we have that  
$
\cB=\langle G^R(\cC)  ,  \cA \rangle. 
$
As we  indicated earlier, since $G^R$ is fully faithful we will drop it from our notations  whenever this is not likely to create confusion: thus we will denote  
$G^R(\cC)$ simply by $\cC,$  and 
write 
$
\cB=\langle \cC ,   \cA \rangle. 
$

\subsection{Preordered semi-orthogonal decompositions}
\label{psodsec}

 \label{gluing psods}
In this section we  introduce \emph{preordered  semi-orthogonal decompositions}
of $\infty$-categories. This concept was also discussed in \cite{bergh2016geometricity}. See \cite{kuznetsov2015semiorthogonal} for a  survey of  semi-orthogonal decompositions (sod-s). Let $\cC$ be a stable $\infty$-category, and let $P$ be a preordered set. Consider a collection 
of full stable  subcategories $\iota_x\colon \cC_x  
\longrightarrow \cC$ indexed by $x \in P$.

\begin{definition}\label{def psod} 
We say that the subcategories $\cC_x$ form a    \emph{preordered semi-orthogonal decomposition (psod) of type $P$}  if they satisfy the following three properties.
\begin{itemize}
\item For all $x \in P,$ $\cC_x$ is a non-zero  \emph{admissible} subcategory: that is, the embedding  $
\iota_x $ admits a right adjoint and a left adjoint, which we denote by
$$
r_x\colon \cC \longrightarrow \cC_x \quad \text{and} \quad 
l_x\colon \cC \longrightarrow \cC_x .
$$
\item If $y<_P x$, i.e. $y \leq_P  x,$ and $ x \not =y,$ then $\cC_y \subseteq \cC_x^{\bot}.$ 
\item $\cC$ is the smallest stable subcategory of $\cC$ containing all the subcategories $\cC_x,$  $x \in P.$
\end{itemize}
\end{definition}

\begin{definition}
If $\cC$ is equipped with a psod of type $(P,\leq)$,  we write 
$
\cC = \langle  \cC_x,  x \in (P,\leq)  \rangle.
$
\end{definition}

\begin{definition}
Let $\cC = \langle  \cC_x,  x \in (P,\leq)  \rangle$ and $\cD = \langle  \cD_y,  y \in (Q,\leq)  \rangle$ be categories equipped with psod-s. Let 
$F: \cC \to \cD$ be a fully-faithful functor. We say that $F$ is \emph{compatible with the psod-s} if for all $x \in P$ there exists $y \in Q$ such that $F(\cC_x)=\cD_y$. \end{definition}

\begin{remark}
\label{stabletriang}
The definition of psod makes  sense also for classical triangulated categories, and not just for $\infty$-categories. We mentioned in Section \ref{esoinfca} that a sequence of stable 
$\infty$-categories is exact if and only if the corresponding sequence of homotopy categories is a Verdier localization. Also, $F\colon \cA \leftrightarrows \cB\cocolon G$ is an adjoint pair of functors between stable $\infty$-categories if and only if 
$$\mathrm{Ho}(F)\colon\mathrm{Ho}(\cA) \leftrightarrows \mathrm{Ho}(\cB)\cocolon \mathrm{Ho}(G)$$
is an adjoint pair between the homotopy categories. 
As a consequence psod-s can be checked at the level of homotopy categories: that is, a collection of admissible subcategories $\cC_x$ gives a psod of type $P$ of the $\infty$-category $\cC$ if and only if $\mathrm{Ho}(\cC_x)$ gives a psod of type $P$ of the triangulated category $\mathrm{Ho}(\cC).$
\end{remark}

\begin{definition}
Let $P$ and $Q$ be preordered sets. 
The \emph{join  
of $P$ and $Q$}  is the preordered set 
$$ 
P * Q:=(P \coprod Q, \leq_{P * Q})
$$  
defined as follows. Let $x$ and $y$ be in $P \coprod Q$:
\begin{itemize}
\item if $x,y \in P,$ then $x \leq_{P * Q}y$ if and only if $x \leq_P y$;
\item  if $x,y \in Q,$ then $x \leq_{P * Q}y$ if and only if $x \leq_Q y$;
\item  if $x \in P$ and $y \in Q,$ then $x \leq_{P * Q}y$.
\end{itemize} 
\end{definition}

\begin{lemma}
\label{lem:preord}
Let 
$ \,\, 
\cA \stackrel{F}\longrightarrow \cB \stackrel{G}\longrightarrow \cC
\, \, $
be an exact sequence of stable idempotent-complete $\infty$-categories. Assume that $\cA$ is admissible, and that $\cA$ and $\cC$ carry  psod-s of type $P_\cA$ and $P_\cC,$ respectively.   
Then 
$\cB$ carries a canonical psod of type 
$ \, \, 
P_\cB:=P_\cA * P_\cC
$ 
such that for all $x \in P_\cB$ the component $\cB_x$ is equivalent to $\cA_x,$ if $x$ is in $P_\cA,$ and to $\cC_x,$ if $x$ is in $P_\cC.$ 
\end{lemma}
\begin{proof}
It follows from the assumptions that the functor $G$ has a right adjoint $G^R\colon\cC \to \cB$. The  admissible subcategories $\cC_x,$ $x \in P_\cC,$ are  admissible subcategories of $\cB$ under the image of $G^R,$ and they are all right orthogonal to the image of $\cA$ under $F$. In particular, they are right orthogonal to  $F(\cA_x),$ $x \in P_\cA,$ and this concludes the proof. 
\end{proof}

\subsection{The Chern character}
\label{dgchern}
In section \ref{ncmols} we will study  the Chern character in the logarithmic setting.  
For later reference, we give below an account of the Chern character that follows closely \cite{BGT}.

 Let $\cS_\infty$ be the $\infty$-category of spectra. It follows from \cite{BGT} that the Chern character can be defined in the abstract setting of $\infty$-categories as a natural transformation between additive invariants 
\begin{equation}
\label{chernchar}
\mathrm{ch}\colon K(-) \Rightarrow \mathrm{HH}(-)\colon \mathrm{Cat}^{\mathrm{perf}}_{\infty, \kappa} \longrightarrow \cS_\infty 
\end{equation}
where:
\begin{itemize}
\item $K(-)$  
is the algebraic K-theory,
\item $\mathrm{HH}(-)$ 
is the Hochschild complex viewed as an object in  spectra.
\end{itemize}
As explained in Section 10 of \cite{BGT}  the Chern character (\ref{chernchar}) is uniquely determined by the choice of the element 
$1 \in \mathrm{HH}(\Perf(\kappa)) \simeq \kappa$. 
Assume now that $\kappa$ is a field of characteristic $0.$ Then the Chern character (\ref{chernchar}) captures the ordinary de Rham Chern character, we refer  to \cite{cualduararu2005mukai} for a thorough discussion of these aspects.  More precisely, let $X$ be a smooth and proper scheme over $\kappa$.   Denote by $\mathrm{H}_{\mathrm{dR}}(X)$ be the de Rham cohomology of 
$X,$ which is  the hypercohomology of the de Rham complex. The 
HKR theorem gives an isomorphism $ \, 
\mathrm{HH}_0(\Perf(X)) \cong  \bigoplus_{k\geq 0} \mathrm{H}_{\mathrm{dR}}^{2k}(X) \,  
$. 
Then the composition
$$
\mathrm{ch}_{\mathrm{dR}}\colon K_0(X) \stackrel{\mathrm{ch}  }\longrightarrow \mathrm{HH}_0(\Perf(X)) \stackrel{\cong} \longrightarrow \bigoplus_{k\geq 0} \mathrm{H}_{\mathrm{dR}}^{2k}(X), 
$$
recovers the ordinary Chern character  taking values in the even de Rham cohomology of $X$.

\section{Perfect complexes over infinite root stacks}
\label{GC}
In this section we construct sod-s  for infinite root stacks. In \ref{rooteff} and \ref{sncd}   we    treat separately the case of root stacks of a single Cartier divisor, and of a simple normal crossings divisor with an arbitrary number of components. We start  by reviewing the results 
obtained  in   \cite{ishii2011special} and \cite{bergh2016geometricity} for  finite root stacks of simple normal crossings divisors. We construct recursively compatible sod-s for the $\infty$-categories of perfect complexes of these root stacks, for a cofinal subset of indices. This will be key to constructing sod-s  on the infinite root stack. 

 In Section \ref{beyond} we explain how to extend these results beyond the normal crossing case: in Section \ref{ncd} we extend our investigation  to  root stacks of general (not necessarily simple) normal crossings divisors, and in Section \ref{reducetoncviacan} we discuss the case of log stacks with simplicial log structure. 

\subsection{Root stacks of a regular divisor}
\label{rooteff}
Let $X$ be an algebraic stack, and $D\subset X$ a regular Cartier divisor. We use the notations of Section \ref{sec:smooth} in the preliminaries. Recall in particular that we denote by by $g_{r',r}$ the projection $\radice[r']{(X,D)}\to \radice[r]{(X,D)}$ for $r\mid r'$, by $D_r$ the universal Cartier divisor on $\radice[r]{(X,D)}$, i.e. the reduction of the preimage $g_{r,1}^{-1}(D)\subset \radice[r]{(X,D)}$. Further, we denote by $i_r\colon D_r\to \radice[r]{(X,D)}$ the closed embedding. 

\begin{lemma}
\label{verdierdivisor} 
The category $
\Perf(D_r)$ splits as the direct sum of $r$ copies of 
$\Perf(D).$ More precisely, if $\mathbb{Z}_r$ is the Cartier dual 
of $\mu_r,$  there are natural equivalences  
$$
\Perf(D_r) \stackrel{(1)}\simeq  \bigoplus_{\chi \in \mathbb{Z}_r} \Perf(D_r)_\chi  \stackrel{(2)}\simeq  \Perf(D) \otimes \Perf(\mathscr{B} \mu_r) \simeq \Perf(D) \otimes \bigoplus_{\chi \in \mathbb{Z}_r} \Perf(\kappa).
$$
\end{lemma}
\begin{proof}
This is a well-known fact, that applies more generally to gerbes banded by a diagonalizable group scheme, so we limit ourselves to a brief sketch. Equivalence $(1)$ comes from the character decomposition of $\Perf(D_r).$  If $\chi$ is in $\mathbb{Z}_r,$ let $\kappa(\chi)$ be the corresponding $\mu_r$-representation. Then equivalence $(2)$ 
maps the twisted structure sheaf $\cO_{D}^\chi$ to $\cO_{D} \otimes \kappa(\chi).$  \end{proof}

\begin{lemma}
Let $r, r' \in \mathbb{N},$ and $r \mid r'.$ Then the pull-back functors 
$$
g_{r',r}^*\colon \Perf(\radice[r]{(X,D)}) \longrightarrow 
 \Perf (\radice[r']{(X,D)})
$$
are fully faithful.  
\end{lemma}
\begin{proof}
Since $g_{r',r}$ is a relative coarse moduli space map, the natural map
$
 \cO_{ \radice[r]{(X,D)}} \longrightarrow g_{r',r,*}\cO_{ \radice[r']{(X,D)}} 
$
is an equivalence in $\Perf( \radice[r]{(X,D)}).$ 
This easily implies that $g_{r',r}^*$ is fully faithful by the adjunction with ${g_{r',r}}_*$ and the projection formula (see for instance \cite[Lemma 4.4]{bergh2016geometricity}). 
\end{proof}

Let $\mathbb{Z}_r^* = \mathbb{Z}_r \setminus \{0\}$ be the set of non-trivial characters of $\mu_r.$ 
In order to keep track of indices it will be convenient to identify  $\mathbb{Z}_r$ and 
$\mathbb{Z}_r^*$ with 
subsets of $\mathbb{Q}/\mathbb{Z}  = \mathbb{Q} \cap (-1,0]$ as follows:
$$
\mathbb{Z}_r \cong \left\{-\frac{r-1}{r}, \hdots, -\frac{1}{r}, 0\right\} \subset \mathbb{Q} \cap (-1,0], \quad  
\mathbb{Z}_r^* \cong \left\{-\frac{r-1}{r}, \hdots, -\frac{1}{r}\right\} \subset \mathbb{Q} \cap (-1,0].
$$ 
We equip $\mathbb{Z}_r$ and $\mathbb{Z}_r^*$ with the total order $\leq$ given by 
$$-\frac{r-1}{r} < -\frac{r-2}{r} < \hdots < -\frac{1}{r} < 0.$$ Here $-\frac{k}{r}$ should be thought of as the element $r-k$ in $\{1,\hdots, r\}$, equipped with the standard ordering. This (perhaps unusual) identification will be convenient when we pass to the limit for $r\to \infty$.

Following Theorem 4.7 of  \cite{bergh2016geometricity}, for every $\chi \in \mathbb{Z}_r^*$ we consider the fully faithful embedding
$$
\Phi_\chi\colon \Perf(D) \to \Perf( \radice[r]{(X,D)})
$$
given by the composite
$$
\Perf(D) \stackrel{\simeq} \longrightarrow \Perf(D_r)_\chi \stackrel{\subset} \longrightarrow  \Perf(D_r) \stackrel{i_{r,*}}\longrightarrow  \Perf(\radice[r]{(X,D)}).
$$
Note that if we see $\chi$ as an element of $\{1,\hdots, r\}$, in the notation of \cite{bergh2016geometricity} the objects of the image of $\Phi_\chi$ are actually equipped with the action corresponding to the character $-\chi\equiv r-\chi\in \bZ_r$.
\begin{remark}
Each individual summand $\Perf(D_r)_\chi$ of $\Perf(D_r)$ embeds fully faithfully in the category $\Perf(\radice[r]{(X,D)})$ via $\Phi_\chi,$ which is a restriction of $ 
i_{r, *} 
$ to $\Perf(D_r)_\chi.$ However  
the functor 
$ 
i_{r, *} 
$ 
itself is not fully faithful.
\end{remark}
We introduce the following notations: 
\begin{itemize}
\item Let $\chi$ be in $\mathbb{Z}_r.$ If $\chi \neq 0$  we denote by 
$\mathcal{A}_\chi \subset \Perf(\radice[r]{(X,D)})$ the image of $\Perf(D)$ under $\Phi_\chi,$ and we denote   by $\mathcal{A}_0 \subset \Perf(\radice[r]{(X,D)})$ the image of $\Perf(X)$  under $g_{r,1}^*.$
\item We denote by $Q_{r }$ the subcategory of $\Perf(\radice[r]{(X,D)})$  generated by the subcategories $\mathcal{A}_\chi$ for $\chi \in \mathbb{Z}_r^*$.
\end{itemize}
 The next theorem is proved in \cite{bergh2016geometricity} using the language of classical triangulated categories, but the proof applies without variations to the $\infty$-setting.
\begin{proposition}[{\cite[Theorem 4.7]{bergh2016geometricity}}]\label{prop: sod}
 \mbox{ }
\begin{enumerate}[leftmargin=*] 
\item The category $\mathcal{A}_0$ is an admissible subcategory of  $\Perf(\radice[r]{(X,D)})$.
\item  The right orthogonal of $\cA_0$ inside $\Perf(\radice[r]{(X,D)})$ is the subcategory $Q_{r}$ introduced above. There is a psod of type 
$(\mathbb{Z}_r^*, \leq)$ 
$$
Q_{r } =  \langle  \cA_\chi ,  \chi \in  (\mathbb{Z}_r^*, \leq)  \rangle.
$$ 
\item By items (1) and (2), the category $\Perf(\radice[r]{(X,D)})$ has a psod of type $(\mathbb{Z}_r, \leq)$ 
$$\Perf(\radice[r]{(X,D)})= \langle  Q_r, \cA_0   \rangle = \langle \cA_\chi, \chi \in  (\mathbb{Z}_r, \leq)  \rangle.$$
\end{enumerate}
\end{proposition}
Consider the directed system of root stacks $\{\radice[n!]{(X,D)}\}_{n\in \bN}$ whose indices are the factorials, with the natural projections 
$$
\cdots \longrightarrow \radice[3!]{(X,D)}  \xrightarrow{g_{3!,2!}} \radice[2!]{(X,D)}  \xrightarrow{g_{2!,1!}}  \radice[1!]{(X,D)} = X.
$$
We point out once again that this subsystem of indices is identified with $\bN$ with its standard ordering, i.e. if $n\leq m$ we have a map $\radice[m!]{(X,D)}\xrightarrow{g_{m!,n!}} \radice[n!]{(X,D)}$, and these are compatible in the obvious sense.
We will inductively construct a tower of compatible 
sod-s on the categories of perfect complexes on these root stacks. This is done in Proposition \ref{sod!} below by applying iteratively Proposition \ref{prop: sod}. For $n\geq 2$, the sod-s we will construct on 
$\Perf(\radice[n!]{(X,D)})$ will be different from the sod-s on 
 the category of perfect complexes on the root stacks of 
 $(X,D)$ 
 given directly by 
Proposition \ref{prop: sod} (see Example \ref{ex:weird.psod}).

We will equip the set 
$\mathbb{Q}/\mathbb{Z} = \mathbb{Q} \cap (-1,0]$ with a  total order $\leq^ !$ which is not the restriction of the usual ordering of the real numbers.
First of all we define the order $\leq^!$ on $\bZ_{n!}$ recursively, as follows.
\begin{itemize}
\item On $\bZ_{2!}=\{-\frac{1}{2},0\}$ we set $-\frac{1}{2}<^! 0$.
\item Having defined $\leq^!$ on $\bZ_{(n-1)!}$, let us consider the natural short exact sequence
$$
0\to \bZ_{(n-1)!}\to \bZ_{n!}\xrightarrow{\pi_n} \bZ_n\to 0,
$$
where $\bZ_n=\{-\frac{n-1}{n},\hdots, -\frac{1}{n},0\}$ is equipped with the standard order $\leq$ described above. Given two elements $a,b\in \bZ_{n!}$, we set $a\leq^! b$ if either $\pi_n(a)<\pi_n(b)$, or $\pi_n(a)=\pi_n(b)$ and $a \leq^! b$ in $\bZ_{(n-1)!}$, where we are identifying the fiber $\pi_n^{-1}(\pi_n(a))\subseteq \bZ_{n!}$ with $\bZ_{(n-1)!}$ in the canonical manner.
\end{itemize}

For example, on $\bZ_{3!}=\{-\frac{5}{6}, -\frac{4}{6},-\frac{3}{6},-\frac{2}{6},-\frac{1}{6},0\}$, the resulting ordering is described by
$$
-\frac{5}{6} <^! -\frac{2}{6} <^! -\frac{4}{6} <^!-\frac{1}{6} <^! -\frac{3}{6} <^! 0.
$$
Now every element 
in $\mathbb{Q} \cap (-1,0]$ can be written as 
$-\frac{p}{n!}$ for some $p \in \mathbb{N}$ and $n \in \mathbb{N}\setminus\{0\}.$ This expression is unique if we require $n$ to be as small as possible, and we call this the \emph{normal factorial form}.  Let $\chi = -\frac{p}{n!}, \chi'= -\frac{q}{m!} \in \mathbb{Q} \cap (-1,0]$ be in normal factorial form. 
We write $\chi <^! \chi'$ if:
\begin{itemize}
\item $n > m$, or
\item $n=m$ and $-\frac{p}{n!} <^! -\frac{q}{n!}$ in $\bZ_{n!}$.
\end{itemize}
For example, with this ordering we have $-\frac{1}{24}<^!-\frac{2}{6}<^!-\frac{4}{6}<^!-\frac{1}{2}$. 

\begin{proposition}\label{sod!}
 \mbox{ }
For every $n \in \mathbb{N}$ the category $\Perf(\radice[n!]{(X,D)})$ has a psod of type $(\mathbb{Z}_{n!}, \leq^!)$
$$
\Perf(\radice[n!]{(X,D)}) = \langle  \cA_\chi^!, \chi \in (\mathbb{Z}_{n!}, \leq^!)  \rangle,
$$
where   
\begin{itemize} 
\item  $\mathcal{A}_0^! \simeq \Perf(X)$, and
\item  $\mathcal{A}_\chi^! \simeq \Perf(D)$ for all $\chi \in \mathbb{Z}_{n!}^*$.
\end{itemize}
Further, for all $n \in \mathbb{N}$ the functor 
$$g_{(n+1)!, n!}^*: \Perf(\radice[n!]{(X,D)}) \longrightarrow \Perf(\radice[(n+1)!]{(X,D)})
$$
is compatible with the  psod-s. 
\end{proposition}

We remark again that for a given $\chi \in \mathbb{Z}_{n!}^*$, it is not necessarily the case that $\cA_\chi=\cA_\chi^!$ as subcategories of $\Perf(\radice[n!]{(X,D)})$ (see Example \ref{ex:weird.psod} below).

\begin{proof}
We construct the sod with 
the desired properties on $\Perf(\radice[n!]{(X,D)})$ inductively.

{\bf Basis: $n=2.$} We take the sod on 
$\Perf(\radice[2!]{(X,D)})$ given by Proposition \ref{prop: sod}. We take $\cA_0$  and $\cA_{\frac{1}{2}}$ to be the same as in Proposition \ref{prop: sod}. This clearly gives a psod of $\Perf(\radice[2]{(X,D)})$ of type $(\mathbb{Z}_2, \leq^!)$ satisfying the properties of the claim.  

{\bf Inductive step: $n-1 \to n.$}  Recall that there is a natural identification 
\begin{equation}
\label{eq!}
\radice[n!]{(X,D)} \simeq \radice[n]{(\radice[(n-1)!]{(X,D)}, D_{(n-1)!})}.
\end{equation}
Further, the projection
$$
\radice[n!]{(X,D)} \simeq \radice[n]{(\radice[(n-1)!]{(X,D)}, D_{(n-1)!})} \longrightarrow \radice[(n-1)!]{(X,D)}
$$
coincides under this identification with the map $g_{n!, (n-1)!}.$ 
We then apply Proposition \ref{prop: sod} to (\ref{eq!}). This yields a psod of type 
$(\mathbb{Z}_n, \leq)$
$$
\Perf(\radice[n!]{(X,D)}) = \langle  Q_n, \cB_0     \rangle
=  \langle  \cB_\zeta, \zeta \in \mathbb{Z}_n   \rangle,
$$
where:
\begin{itemize}
\item the subcategory $\cB_0 \simeq \Perf(\radice[(n-1)!]{(X,D)})$ is given by the image of $g_{n!, (n-1)!}^*$, and
\item the subcategory $\cB_\zeta \simeq 
\Perf(D_{n!})_\zeta \simeq \Perf(D_{(n-1)!})$ is given by the image of $\Phi_\zeta$.
\end{itemize}
Now note that by Lemma \ref{verdierdivisor} for all $\zeta \in \mathbb{Z}_n^*$ the category $\cB_\zeta$ splits as a direct sum of categories labelled by characters in $\mathbb{Z}_{(n-1)!},$ that is: 
$$
\cB_\zeta  \simeq \Perf(D_{(n-1)!}) \simeq \bigoplus_{\xi \in { \mathbb{Z}_{(n-1)!}}} \Perf(D)_\xi 
=: \bigoplus_{\xi \in { \mathbb{Z}_{(n-1)!}}} \cB_{\zeta, \xi}.
$$
We identify the 
subcategories $\cB_{\zeta, \xi}$ with  the factors $\cA_\chi^!$ for $\chi \in \mathbb{Z}_{n!}^*$  appearing in the statement of the proposition by setting 
$$\cA^!_{\frac{\zeta}{(n-1)!} + \xi}:= \cB_{\zeta, \xi}.$$
Note that we can make sense of the expression 
$\frac{\zeta}{(n-1)!} + \xi$ as an element of $\mathbb{Z}_{n!}^*$ because we have identified the sets of characters $\mathbb{Z}_{n!}$, $\mathbb{Z}_{(n-1)!}$ and $\mathbb{Z}_n$ with subsets of $ \mathbb{Q} \cap  (-1,0].$ It is easy to verify that we can write as sums of the form $\frac{\zeta}{(n-1)!} + \xi$ exactly the elements of $\mathbb{Z}_{n!}$ which are in the complement of $\mathbb{Z}_{(n-1)!},$ that is
$$
\left\{ \frac{\zeta}{(n-1)!} + \xi \, \Bigm\vert \, \zeta\in  \mathbb{Z}_n^* , \, \, \xi \in { \mathbb{Z}_{(n-1)!} } \right\} = \mathbb{Z}_{n!} \setminus \mathbb{Z}_{(n-1)!} \subset \mathbb{Q} \cap (-1,0]. 
$$ 
Thus we conclude that the subcategory $Q_n$ has a psod of type $(\mathbb{Z}_{n!} \setminus\mathbb{Z}_{(n-1)!}, \leq^!)$
$$
Q_n = \langle  \cA_\chi^!, \chi \in (\mathbb{Z}_{n!}\setminus\mathbb{Z}_{(n-1)!}, \leq^!) \rangle,
$$
such that $\mathcal{A}_\chi^! \simeq \Perf(D)$ for all $\chi \in \mathbb{Z}_{n!}\setminus\mathbb{Z}_{(n-1)!}.$ 

 By the inductive hypothesis  $\cB_0 \simeq \Perf(\radice[(n-1)!]{(X,D)})$ carries a psod of type $(\mathbb{Z}_{(n-1)!}, \leq^!) $
$$
\cB_0 \simeq \Perf(\radice[(n-1)!]{(X,D)}) = \langle  \cA_\chi^!, \chi \in (\mathbb{Z}_{(n-1)!}, \leq^!)  \rangle  
$$
such that:
\begin{itemize} 
\item  $\mathcal{A}_0^! \simeq \Perf(X)$, and
\item  $\mathcal{A}_\chi^! \simeq \Perf(D)$ for all $\chi \in \mathbb{Z}_{(n-1)!}^*$.
\end{itemize}  
We can thus write 
$$
\Perf(\radice[n!]{(X,D)}) = \langle  Q_n, \cB_0    \rangle =  \langle \langle  \cA_\chi^!, \chi \in (\mathbb{Z}_{n!}\setminus\mathbb{Z}_{(n-1)!}, \leq^!) \rangle  ,  \langle \cA_\chi^!, \chi \in (\mathbb{Z}_{(n-1)!}, \leq^!) \rangle  \rangle. 
$$
Note that the we have a canonical isomorphism of  ordered sets
$$
(\mathbb{Z}_{n!}\setminus\mathbb{Z}_{(n-1)!}, \leq^!) *(\mathbb{Z}_{(n-1)!}, \leq^!)  \cong (\mathbb{Z}_{n!}, \leq^!). 
$$ 
Thus, by Lemma  \ref{lem:preord} the category $
\Perf(\radice[n!]{(X,D)})$ carries a psod of type 
$(\mathbb{Z}_{n!}, \leq^!),$
$$
\Perf(\radice[n!]{(X,D)}) = \langle  \cA_\chi^!, \chi \in (\mathbb{Z}_{n!}, \leq^!) \rangle. 
$$
The compatibility with the psod-s of the pull-back along root maps follows by construction.  This concludes the proof. 
\end{proof}

\begin{example}\label{ex:weird.psod}
For $n=2$, by construction our psod coincides with the one given by Proposition \ref{prop: sod}. For $n=3$ though, the psod given by that proposition for $\radice[3!]{(X,D)}$ looks like
\begin{equation}
\label{sodsodsob}
\Perf(\radice[6]{(X,D)})=\langle \cA_{-\frac{5}{6}}, \cA_{-\frac{4}{6}}, \cA_{-\frac{3}{6}}, \cA_{-\frac{2}{6}}, \cA_{-\frac{1}{6}}, \Perf(X)\rangle
\end{equation}
(where $\cA_{-\frac{k}{n}}$ is the factor $\Phi_{k}$ in Theorem 4.7 of \cite{bergh2016geometricity}).
The psod that was constructed in the previous proposition, on the other hand, has the form
$$
\langle \cB_{-\frac{2}{3}},\cB_{-\frac{1}{3}}, \Perf(\radice[2]{(X,D)}) \rangle,
$$
where $\cB_{-\frac{k}{3}}\simeq \Perf(D_2)\simeq \Perf(D)_{-\frac{1}{2}}\oplus\Perf(D)_0$, and $\Perf(\radice[2]{(X,D)})=\langle \cA^!_{-\frac{1}{2}}, \Perf(X)\rangle$, embedded in $\Perf(\radice[6]{(X,D)})$ via pullback along the projection $\radice[6]{(X,D)}\to \radice[2]{(X,D)}$.

Following the proof of the previous proposition, the first term $\cB_{-\frac{2}{3}}$ gives us $\cA^!_{-\frac{5}{6}}\oplus \cA^!_{-\frac{2}{6}}$, the second term $\cB_{-\frac{1}{3}}$ gives $\cA^!_{-\frac{4}{6}}\oplus \cA^!_{-\frac{1}{6}}$, and $\cA^!_{-\frac{3}{6}}$ is defined as the image of $\cA^!_{-\frac{1}{2}}\subseteq \Perf(\radice[2]{(X,D)})$. Overall, the psod looks like
$$
\langle \cA^!_{-\frac{5}{6}}, \cA^!_{-\frac{2}{6}}, \cA^!_{-\frac{4}{6}}, \cA^!_{-\frac{1}{6}}, \cA^!_{-\frac{3}{6}}, \Perf(X)\rangle
$$
(note that this reflects exactly the ordering $\leq^!$ on $\bZ_{3!}$).

In fact we could swap $\cA^!_{-\frac{2}{6}}$ and $\cA^!_{-\frac{4}{6}}$, and we have $\cA_{-\frac{k}{6}}=\cA^!_{-\frac{k}{6}}$ for all values of $k$ except $3$: the subcategory $\cA^!_{-\frac{3}{6}}$ does not contain the structure sheaf of $D$ equipped with character $-\frac{3}{6}$, but only thickened versions of it.
\end{example}

Set $\mathbb{Q}/\mathbb{Z}^*:=\mathbb{Q}/\mathbb{Z}\setminus\{0\}.$ 
\begin{proposition}\label{sod!Q}
 \mbox{ }
The category $\Perf(\radice[\infty]{(X,D)})$ has a psod of type $(\mathbb{Q}/\mathbb{Z}, \leq^!)$
$$
\Perf(\radice[\infty]{(X,D)}) = \langle  \cA_\chi^!, \chi \in (\mathbb{Q}/\mathbb{Z}, \leq^!)  \rangle,
$$
where:
\begin{itemize} 
\item  $\cA_0^! \simeq \Perf(X)$, and
\item  $\cA_\chi^! \simeq \Perf(D)$ for all $\chi \in \mathbb{Q}/\mathbb{Z}^*$.
\end{itemize}
\end{proposition}
\begin{proof}
Factorials are cofinal in the filtered set of natural numbers ordered by divisibility. This together with Proposition \ref{prop.db.irs}  implies that 
$\Perf(\radice[\infty]{(X,D)})$ is the colimit of the directed system of fully-faithful embeddings 
\begin{equation}
\label{directed!}
\Perf(X) \xrightarrow{g_{2!,1!}^*} \Perf(\radice[2!]{(X,D)}) 
\xrightarrow{g_{3!,2!}^*} \Perf(\radice[3!]{(X,D)}) 
\xrightarrow{g_{4!,3!}^*} \cdots
\end{equation}
By Proposition \ref{sod!} 
the $n$-th category in the directed system (\ref{directed!}) carries a psod of type $(\mathbb{Z}_{n!}, \leq^!).$ Further the structure functors are compatible with the  sod-s: the compatibility is witnessed by the inclusion of preordered sets 
$$
(\mathbb{Z}_{(n-1)!}, \leq^!) \subset (\mathbb{Z}_{n!}, \leq^!),
$$ 
which embeds the indexing set of the sod of $\Perf(\radice[(n-1)!]{(X,D)})$ into the indexing set of the sod of $\Perf(\radice[n!]{(X,D)})$. 

By Proposition \ref{colim}, $\Perf(\radice[\infty]{(X,D)})$ is given by the union of the  categories making up the directed system (\ref{directed!}).   
This implies that $\Perf(\radice[\infty]{(X,D)})$ carries a psod of type
$$
\bigcup_{n \in \mathbb{N}} (\mathbb{Z}_{n
!}, \leq^!) = (\mathbb{Q}/\mathbb{Z}, \leq^!).
$$
It is immediate to see that this sod satisfies the  properties $(1)$ and $(2)$ from the statement. This concludes the proof. 
\end{proof}

     \subsection{Root stacks of simple normal crossings divisors} 
     \label{sncd}  

As shown in \cite{bergh2016geometricity}, the  categories of perfect complexes over the  root stacks of simple normal crossing divisors carry a canonical  
sod. We 
review that result 
in a slightly different formulation, 
which is better adapted to our purposes. 
Further, we extend it to the 
infinite root stack. 
We start by fixing  notations. 

Let $X$ be an algebraic stack and $D$ be a simple normal crossings divisor on $X$. We denote by $D_1,\hdots, D_N$ its irreducible components. As recalled in the preliminaries, this gives rise to a log stack $(X,D)$. 
Denote by $I$ the set $\{1, \ldots, N\}$.
The  stack $X$ carries a canonical stratification, given by the closed substacks of  $D$  obtained as intersections of the $D_i$. 
If $J$ is a subset of $I$, we denote  
$D_J := \cap_{j \in J} D_j$ if $J \neq \varnothing,$ and we set $D_\varnothing := X$ otherwise. Let 
$ \, \, 
S_I := \{J \subseteq I \}
$
be the power set of $I$, and $S_I^*$ denote the subset $S_I\setminus\{\varnothing\}$.  We often equivalently regard $I$ as the set of irreducible components of $D$, $S_I$ as the set of strata of $(X,D)$, and $S_I^*$ as the set of the strata of positive codimension.  We equip 
$S_I$ with a preorder $\leq$ that keeps track of the inclusions of strata, but it is finer than that. 
Namely we let  $\leq$ be the coarsest preorder on 
$S_I$ with the following two properties: 
\begin{itemize}
\item if $J \subseteq J',$ then $J' \leq J$, and
\item if $d$ is the dimension of $X$, then  the assignment $(S_I, \leq) \to (\mathbb{N}, \leq)$ mapping  
a subset $J$ to  $d \, -|J|$  is an order-reflecting map. 
\end{itemize} 
By ``order-reflecting map'' we mean a map between preordered sets $f:P \to Q$ such that $f(p) \le f(p')$ implies $p \le p'$. Note in particular that these conditions impose both $J\leq J'$ and $J'\leq J$ if $| J|=| J'|$.

Recall that the root stacks of $(X,D)$ are indexed by multi-indices 
$ \,  
\vec{r} = (r_1, \ldots, r_N) \in \mathbb{N}^N.  
\, $   
For elements $\vec{r},\vec{r'}$ of $\mathbb{N}^N$ we write 
$
\vec{r}=(r_1, \ldots, r_N)\; \mid \;(r_1', \ldots, r_N')=\vec{r'} $ if 
$  r_1 \mid r_1', \ldots, r_N \mid r_N'. 
$
Recall also that the root stack $\radice[\vec{r}]{(X,D)}$ in this case can be realized as the limit of the diagram of stacks
$$
\xymatrix{
\radice[r_1]{(X,D_1)} \ar[rrd] & \radice[r_2]{(X,D_2)} 
\ar[rd] & \ldots & \radice[r_{N-1}]{(X,D_{N-1})} \ar[ld] & \radice[r_N]{(X,D_N)} 
\ar[lld] \\
&& X . &&
}
$$
If $\vec{r}, \vec{r'}\in \mathbb{N}^N,$ $\vec{r} \mid \vec{r'},$ 
we denote the natural maps between root stacks  
by
$$
g_{\vec{r'},\vec{r}}: \radice[\vec{r'}]{(X,D)} \longrightarrow \radice[\vec{r}]{(X,D)}.$$

\subsubsection{Root stacks and the strata of $(X,D)$}
 Let $\vec{r} \in \mathbb{N}^N$. We will use the following notation. 
\begin{itemize}
\item 
Let $J = \{j_1, \ldots, j_k \}\subseteq I$ be non-empty. We denote by $\vec{r}_J \in \mathbb{N}^N$ the index vector obtained from $\vec{r}$ by setting to $1$  all the entries whose index is not in $J$. In formulas, letting $\vec{e}_j$ be the size $N$ vector with $j$-th entry $1$ and all other entries equal to $0$, we can write 
$$
\vec{r}_J=\sum_{j \in J} r_{j} \vec{e_{j}} + \sum_{j \notin J}   \vec{e_{j}}.
$$
Note that $\vec{r}_J \mid \vec{r}$. 
\item  With slight abuse of notation, for all $i \in I$, we denote by $D_{i, r_i}$ both the universal effective Cartier divisor on the stack $\radice[r_{i}]{(X,D_{i})}$ (obtained as reduction of the preimage of $D_i$), and its pull-back to $\radice[\vec{r}]{(X,D)}$. Whenever we use this notation  we will specify which of the two meanings is the intended one. 
\end{itemize} 
We denote by
$D_{J,\vec{r}}$ the limit of the diagram of stacks 
\begin{equation}
\begin{gathered}
\label{defDJr}
\xymatrix{
D_{j_1,r_{j_1}} \ar[rrd] & D_{j_2,r_{j_2}} 
\ar[rd] & \ldots & D_{j_{k-1},r_{j_{k-1}}} \ar[ld] & D_{j_k ,r_{j_k}}\ar[lld] \\
&& \radice[\vec{r}]{(X,D)}&&
}
\end{gathered}
\end{equation}
 where the arrows are the embeddings $D_{j_l, r_{j_l}} \subset \radice[\vec{r}]{(X,D)}$. In general $D_{J,\vec{r}}$ is not a gerbe over $D_J$, as $D_{J,\vec{r}}$ will have larger isotropy groups along the higher codimensional strata $S \in S_I$ contained in $D_J$. However $D_{J,\vec{r}_J}$ is always a  gerbe  over $D_J$.    
We denote by $i_{J,\vec{r}}: D_{J,\vec{r}} \to \radice[\vec{r}]{(X,D)}$ the embedding.

We set 
$ \, \, 
|J, \vec{r}|:=\prod_{j \in J} r_j ,
\, $ and 
$$
\mu_{J,\vec{r}} := \bigoplus_{j \in J} \mu_{r_j}, \quad \mathbb{Z}_{J,\vec{r}} := \bigoplus_{j \in J} \mathbb{Z}_{r_j}, \quad 
\mathbb{Z}_{J,\vec{r}}^* := \bigoplus_{j \in J} \Big ( \mathbb{Z}_{r_j}\setminus \{0\} \Big ), $$
$$
(\mathbb{Q}/\mathbb{Z})_{J} := \bigoplus_{j \in J} \mathbb{Q}/\mathbb{Z}, \quad 
(\mathbb{Q}/\mathbb{Z})_{J}^*:=   \bigoplus_{j \in J} \Big ( \mathbb{Q}/\mathbb{Z}\setminus \{0\} \Big ).  
$$
Note that $\mathbb{Z}_{J,\vec{r}}$ 
is the Cartier dual of 
$\mu_{J,\vec{r}}$. 
\begin{remark}
\label{remremnotation}
By definition, the set of strata of a pair $(X, D)$, where $D$ is simple normal crossing, is the disjoint union of the intersections of the irreducible components of $D$. Thus strata are in bijection with the subsets of $I$. However it is sometime convenient to label the index sets we have introduced above via the strata themselves, without making an explicit reference to the corresponding subsets $J \subset I$. Thus if $Z$ is a stratum of $(X, D)$, then $Z= \cap_{j \in J} D_j$ for some $J \subset I$, and we will sometime denote by 
$$
\mu_{Z,\vec{r}}, \quad \mathbb{Z}_{Z,\vec{r}}, \quad 
\mathbb{Z}_{Z,\vec{r}}^* ,  \quad (\mathbb{Q}/\mathbb{Z})_{Z}, \quad 
(\mathbb{Q}/\mathbb{Z})_{Z}^*
$$
the index sets $
\mu_{J,\vec{r}}$ , $\mathbb{Z}_{J,\vec{r}}$ ,  
$\mathbb{Z}_{J,\vec{r}}^*$ ,  $(\mathbb{Q}/\mathbb{Z})_{J}$ , and $
(\mathbb{Q}/\mathbb{Z})_{J}^*$. 
\end{remark}

\begin{lemma}\label{lemma: E}
Let $J \subseteq I$ be a non-empty subset. Then the category $
\Perf(D_{J,\vec{r}_J})$ splits as the direct sum of $|J, \vec{r}|$ copies of 
$\Perf(D_J).$ More precisely  
there are natural equivalences  
$$
\Perf(D_{J,\vec{r}_J}) \simeq  \bigoplus_{\chi \in  \mathbb{Z}_{J,\vec r}} \Perf(D_{J,\vec r})_\chi  \simeq  \Perf(D_J) \otimes \Perf(\mathscr{B}\mu_{J,\vec r}) \simeq \Perf(D_J) \otimes \bigoplus_{\chi \in \mathbb{Z}_{J,\vec r}} \Perf(\kappa).
$$
\end{lemma}
\begin{proof}
The proof is similar to the proof of Lemma \ref{verdierdivisor}. 
\end{proof}

\subsubsection{The main theorem}
Similarly to what we did in Section \ref{rooteff} for every $J \subset I$ we identify  $\mathbb{Z}_{J,\vec r},$ $\mathbb{Z}_{J,\vec r}^*$ with 
subsets of $(\mathbb{Q}/\mathbb{Z})_I  = \mathbb{Q}^{I }  \cap (-1,0] ^{ I  }$ via 
$$
\mathbb{Z}_{J,\vec r}^* \subset \mathbb{Z}_{J,\vec{r}} = \bigoplus_{j \in J}   \mathbb{Z}_{r_j}  \subset  \bigoplus_{i \in I}   \mathbb{Z}_{r_i} = \mathbb{Z}_{I,\vec r}  \subset  \bigoplus_{i \in I} \mathbb{Q}/\mathbb{Z}  = 
\mathbb{Q}^{I}  \cap (-1,0] ^{  I  },
$$
where, on each factor, the inclusion of sets $\mathbb{Z}_{r_i} \subset \mathbb{Q}/\mathbb{Z} = \mathbb{Q} \cap (-1,0]$ is the one we considered in Section \ref{rooteff}. We also consider the inclusions 
$$
(\mathbb{Q}/\mathbb{Z})_J^* \subset (\mathbb{Q}/\mathbb{Z})_J \subset  \bigoplus_{j \in J} \mathbb{Q}/\mathbb{Z}   \subset  \bigoplus_{i \in I} \mathbb{Q}/\mathbb{Z}  = 
\mathbb{Q}^{I}  \cap (-1,0] ^{  I  }.$$
The following lemma is immediate. 
\begin{lemma}
\label{lemimm}
We have decompositions as disjoint union  of sets
$$
\mathbb{Z}_{I,\vec r} = \coprod_{J \subset I} 
\mathbb{Z}_{J,\vec r}^*   \quad 
 (\mathbb{Q}/\mathbb{Z})_I = \coprod_{J \subset I} (\mathbb{Q}/\mathbb{Z})_J^*. \vspace{-27.5pt}
$$
\qed\vspace{15pt}
\end{lemma}
By Lemma \ref{lemimm} if $\chi$ is in $\mathbb{Z}_{I,\vec r}$   there is a unique $J \subset I$ such that 
$\chi$ is in $\mathbb{Z}_{J,\vec r}^*,$ and similarly for 
 $(\mathbb{Q}/\mathbb{Z})_I$  and $(\mathbb{Q}/\mathbb{Z})_J^*$.
 
 \begin{definition}
 \label{defpreorder<}
We  equip $(\mathbb{Q}/\mathbb{Z})_I= \mathbb{Q}^ I \cap (-1,0]^ I$ with the product partial order $\leq$: let  
$$
\chi=(\chi_1, \ldots, \chi_N) \quad \text{and} \quad 
\chi'=(\chi'_1, \ldots, \chi'_N) \quad \text{be in} \quad \mathbb{Q}^ I \cap (-1,0]^ I 
$$ 
then we set $\chi \leq \chi'$ if  
$\chi_{l} \leq \chi'_{l}$ for all  $l=1, \ldots, N$ for the restriction of the standard ordering of the real numbers. 
This restricts to a preorder on 
 $ \mathbb{Z}_{J,\vec{r}}, \,  \mathbb{Z}_{J,\vec{r}}^*, \, (\mathbb{Q}/\mathbb{Z})_J$ and $(\mathbb{Q}/\mathbb{Z})_J^*$. 
 \end{definition}
 
  We introduce the following notations: 
\begin{itemize}
\item Let $\vec{0} = (0, \ldots, 0) \in \mathbb{Z}_{I, \vec{r}}.$  We denote   by $\cA_{\vec{0}} \subset \Perf(\radice[r]{(X,D)})$ the image of $\Perf(X)$  under $g_{\vec{r},1}^*.$ 
\item Let $J \in S_I^*$ and let $\chi$ be in $\mathbb{Z}_{J,\vec{r}}^*.$ We denote by 
$\mathcal{A}_{\chi}^J \subset \Perf(\radice[\vec{r}]{(X,D)})$ the image of $\Perf(D_J)$ under the functor 
$\Phi_{J,\chi}$, which is defined as the composite
$$
\Phi_{J,\chi}: \Perf(D_J) \stackrel{(a)} \simeq \Perf(D_{J,\vec r})_\chi \stackrel{(b)} \rightarrow   \Perf(D_{J,\vec{r}_J})  \xrightarrow{ {(i_{J,\vec{r}_J})_*}} 
 \Perf(\radice[\vec{r}_J]{(X,D)}) \xrightarrow{ (g_{\vec{r}, \vec{r}_J})^ *} 
  \Perf(\radice[\vec{r}]{(X,D)}) 
 $$
 where equivalence $(a)$ and   inclusion $(b)$ are explained in Lemma \ref{lemma: E}.  
\item We denote by $\cA_{J}$ the subcategory of $ \Perf(\radice[\vec{r}]{(X,D)})$  generated by the subcategories $\mathcal{A}_{\chi}^J$ for $\chi \in \mathbb{Z}_{J,\vec{r}}^*$. \end{itemize}
 The following statement is a rephrasing of \cite[Theorem 4.9]{bergh2016geometricity}. 
 \begin{proposition}
\label{prop: sodvsimplnc}
 \mbox{ }
\begin{enumerate}[leftmargin=*] 
\item The category $\Perf(\radice[\vec r]{(X,D)})$ has a  psod
 of type $S_I$,  
$ \, 
\Perf(\radice[\vec r]{(X,D)}) = \langle \cA_J,  J \in S_I  \rangle 
\, ,$ such that: 
\begin{itemize}[leftmargin=*]
\item For all $J$, the subcategory $\cA_J$ is admissible. 
\item  For all $J \in S_I^*$, the subcategory $\cA_J$  has a psod of type $(\mathbb{Z}_{J,\vec r}^*, \leq)$ 
$$
\cA_J = \langle  \cA_{ \chi}^J, \chi \in (\mathbb{Z}_{J,\vec r}^*, \leq)  \rangle
$$ 
and, for all $\chi \in \mathbb{Z}_{J,\vec r}^*$, there is an equivalence $\cA_{ \chi}^J \simeq \Perf(D_J)$. 
\end{itemize}
\item The  category $\Perf(\radice[\vec r]{(X,D)})$ has a psod of type $(\mathbb{Z}_{I, \vec{r}}, \leq)$ 
$$\Perf(\radice[\vec r]{(X,D)})=  \langle  \cA_\chi^J, \chi \in  (\mathbb{Z}_{I, \vec{r}}, \leq)   \rangle,$$
where $\cA_{\vec 0}\simeq \Perf(X)$ and for all $\chi \in \mathbb{Z}_{J,\vec r}^*$ there is an equivalence $\cA_{ \chi}^J \simeq \Perf(D_J).$
\end{enumerate}
\end{proposition}

We will equip the set 
$(\mathbb{Q}/\mathbb{Z})_I = \mathbb{Q}^ I \cap (-1,0]^ I$ with a  total order $\leq^ !$ (different from the one of  Definition \ref{defpreorder<}) which generalizes the preorder 
$(\mathbb{Q}/\mathbb{Z}, \leq^ !)$ that we introduced in Section \ref{rooteff}. 
We can write every element $\chi$ in $\mathbb{Q}^ I \cap (-1,0]^ I$ as 
$$
\chi = \left(-\frac{p_1}{n!}, \ldots, -\frac{p_N}{n!}\right)
$$
 for some $p_1, \ldots, p_N$ in $\mathbb{N}$ and $n \in \mathbb{N}.$ This expression is unique if we require $n$ to be as small as possible, and we call this the
\emph{normal factorial form}.  
\begin{definition}
\label{defpartord!}
Let 
$$
\chi = \left(-\frac{p_1}{n!}, \ldots, -\frac{p_N}{n!}\right), \quad \chi'= \left(-\frac{q_1}{m!}, \ldots, -\frac{q_N}{m!}\right) \in \mathbb{Q}^ I \cap (-1,0]^ I
$$ be in normal factorial form. 
We write $\chi \leq^! \chi'$ if:
\begin{itemize}
\item $n > m$, or
\item $n=m$ and $-\frac{p_i}{n!} \leq^! -\frac{q_i}{n!}$ in $\bZ_{n!}$ for all $i = 1, \hdots, N$, where $\leq^!$ is the ordering defined in Section \ref{rooteff}. 
\end{itemize}
For all $\vec r \in \mathbb{N}^N$ we obtain an induced ordering $\leq^!$ on $\mathbb{Z}_{I, \vec{r}}$ and $\mathbb{Z}_{I, \vec{r}}^*$.
\end{definition}
If $n \in \mathbb{N}$ we set $\vec{n}:=(n, \ldots, n)$ and  $\vec{n}!:=(n!, \ldots, n!)\in \bN^N$.

\begin{proposition}
\label{sod!snc}
 \mbox{ }
\begin{enumerate}[leftmargin=*] 
\item The category $\Perf(\radice[\vec{n}!]{(X,D)})$ has a collection of subcategories $\cA_J^!$, $J \in S_I$, such that: 
\begin{itemize}[leftmargin=*]
\item  For all $J$, the subcategory $\cA_J^!$ is admissible. 
\item  For all $J \in S_I^*$, the subcategory $\cA_J^!$  has a psod of type $(\mathbb{Z}_{J,\vec{n}!}^*, \leq^!)$ 
$$
\cA_J^! = \langle  \cA_{ \chi}^{J,!}, \chi \in (\mathbb{Z}_{J,\vec{n}!}^*, \leq^!) \rangle
$$
and, for all $\chi \in \mathbb{Z}_{J,\vec{n}!}^*$, there is an equivalence $\cA_{ \chi}^{J,!} \simeq \Perf(D_J).$
\end{itemize}
\item  The  category  $\Perf(\radice[\vec{n}!]{(X,D)})$ has a psod of type $(\mathbb{Z}_{I, \vec{n}!}, \leq^!)$ 
$$\Perf(\radice[\vec n !]{(X,D)})=  \langle  \cA_\chi^{J,!}, \chi \in  (\mathbb{Z}_{I, \vec{n}!}, \leq^!)   \rangle,$$
where $\cA_{\vec 0}^!\simeq \Perf(X)$ and for all $\chi \in \mathbb{Z}_{J,\vec n!}^*$ there is an equivalence $\cA_{ \chi}^{J,!} \simeq \Perf(D_J).$
\item For all $n \in \mathbb{N}$, $ \, g_{\overrightarrow{(n+1)}!, \overrightarrow{n}!}^*\colon \Perf(\radice[\vec n !]{(X,D)}) \longrightarrow \Perf(\radice[\overrightarrow{(n+1)}!]{(X,D)})
\, $
is compatible with the  psod-s. 
\end{enumerate}
\end{proposition}

The proof of Proposition \ref{sod!snc}  that we give below 
depends on a somewhat involved inductive argument. A much simpler proof is possible, at the price of a mild reduction of generality which still covers most examples of interest. We sketch it in Remark \ref{goodproof} below. 
\begin{proof}
By Lemma \ref{lemimm} 
$\mathbb{Z}_{I,\vec r} = \coprod_{J \subset I} 
\mathbb{Z}_{J,\vec r}^*$ and thus the  Proposition gives a description of all the factors appearing in the psod of 
$
\Perf(\radice[\vec{n}!]{(X,D)})$.

It is actually more convenient to prove first part $(2)$ of the Proposition, and deduce from there part  $(1).$ The compatibility with the psod-s, part $(3)$, follows   automatically. The proof involves a nested induction, first on the number $N$ of irreducible components of $D$, and then on the index $n$ appearing in the statement of Proposition \ref{sod!snc}. 
 Let us clarify the structure of the induction.
 \begin{itemize}
 \item[(a)] The basis step of the induction on $N$ consists in the proof of the statement of 
 Proposition \ref{sod!snc} for $N=1$ and arbitrary $n$. This is given by Proposition \ref{sod!}.  
 \item[(b)]  The inductive step involves proving Proposition \ref{sod!snc} in the case of a divisor $D$ with $N$ irreducible components:  as inductive hypothesis  we assume that  Proposition \ref{sod!snc} holds, for all $n \in \mathbb{N}$, in the case of a  divisor  $D$ with $M$ irreducible components, where $M$ is any integer smaller than $N$.  
 \item[(c)]  
 We establish inductive step $(b)$ via a second induction, this time on the index $n$.  We will spend the rest of the proof explaining the basis step $n=2$ and the inductive step $n-1 \to n$. This will imply inductive step $(b)$ and conclude the proof.
    \end{itemize}

{\bf Basis: $n=2.$} We have the sod on 
$\Perf(\radice[\vec{2}!]{(X,D)})$ given by Proposition \ref{prop: sodvsimplnc}, with the required subcategories $
\cA_J^!$ for $J\in S_I^*$ and $\cA_\chi^{J,!}$ with $\chi \in \mathbb{Z}_{I,\vec{2}}$. 

{\bf Inductive step: $n-1 \to n.$} The proof is similar to the one of Proposition \ref{sod!}, thus we limit  ourselves to an abbreviated treatment of the argument. We use the natural identification 
\begin{equation}
\label{eq!vec}
\radice[\vec{n}!]{(X,D)} \simeq \radice[\vec{n}]{\left(\radice[{\overrightarrow{(n-1)}!}]{(X,D)}, D_{\overrightarrow{(n-1)}!}\right)}.
\end{equation}
Applying Proposition \ref{prop: sodvsimplnc} to (\ref{eq!vec}) yields a psod of type  $(S_I, \leq)$ 
$$\Perf(\radice[\vec n]{(X,D)})=  \langle  \cB_J, J  \in  (S_I, \leq)  \rangle.$$
Additionally each summand $\cB_J$ for $J\in S_I^*$ carries a psod 
$$
\cB_J = \langle \cB_{\zeta}, \zeta \in (\mathbb{Z}_{J,\vec n}^*, \leq)  \rangle,
$$ 
where for all $\zeta \in \mathbb{Z}_{J,\vec n}^*$ there is an equivalence $\cB_{\zeta} \simeq \Perf\left(D_{J, \overrightarrow{(n-1)}!}\right)$.   We will realize the summands $\cA^ {J,!}_{\chi}$ appearing in the statement of Proposition \ref{sod!snc} as semi-orthogonal factors of the categories $\cB_{\zeta}$.

Fix $J$ in $S_I^*$. If $D_J=\varnothing$ then  $\cB_J = 0$, and we set   $\cA^ {J,!}_{\chi}:=0$ for all $\chi \in \mathbb{Z}_{J,\vec{n}!}^*$.  
Assume next that $D_J$ is non-empty. Set $L:=I\setminus J$ and $M:=| L |$.  
The stratum $D_J$ carries a simple normal crossing divisor
\footnote{Some, or even all, the intersections $D_i \cap D_J$, $i \in L$,   might be empty. We can nonetheless argue as if $D^L$ had $M$ distinct irreducible components: the empty strata will give rise to zero categories, and therefore will give no contribution to 
the  psod-s.} 
$$
D^L:=  \bigcup_{i \in L} (D_i \cap D_J) \subset D_J
$$ 
Denote by $\overrightarrow{(n-1)}!^ L \in \mathbb{N}^L$ the diagonal vector with entries all equal to $(n-1)!$ 
$$
\overrightarrow{(n-1)}!^ L :=   ((n-1)!, \ldots, (n-1)!  ) \in \mathbb{N}^L. 
$$ 
The substack 
$D_{J, \overrightarrow{(n-1)}!}$  is a $\mu_{J,\overrightarrow{(n-1)}!}$-gerbe over 
the $\overrightarrow{(n-1)}!^L$-th root stack of $D_J$ 
with respect to $D^ L$. Thus by Lemma \ref{lemma: E} we have a decomposition
$$
\Perf(D_{J, \overrightarrow{(n-1)}!})  \simeq \bigoplus_ { \mathbb{Z}_{J,\overrightarrow{(n-1)}!}}\Perf\left(\radice[\overrightarrow{(n-1)}!^L ]{(D_J,D^ L)}\right) =: \bigoplus_{\xi \in \mathbb{Z}_{J,\overrightarrow{(n-1)}!}}\cB_{\zeta, \xi}.
$$
Additionally, since $M < N$, we can assume  by the inductive hypothesis  that Proposition \ref{sod!snc} applies to the root stack 
$ \radice[\overrightarrow{(n-1)}!^L ]{(D_J,D^ L)}$. This gives us a psod
$$\cB_{\zeta, \xi} \simeq  \Perf\left(\radice[\overrightarrow{(n-1)}!^L ]{(D_J,D^ L)}\right) =   \left \langle  \cB_{\zeta, \xi, \rho} ,  \rho \in  \left (\mathbb{Z}_{L, \overrightarrow{(n-1)}! }  ,  \leq^! \right)  \right \rangle.$$
Similarly to the proof of Proposition  \ref{sod!}, we identify the 
subcategories $\cB_{\zeta, \xi, \rho}$ with   factors $\cA_\chi^{J,!}$ for $\chi \in \mathbb{Z}_{I, \vec{n}!}$  appearing in the statement of the proposition by setting $$
\cA_{{\frac{\zeta}{(n-1)!} + \xi + \rho}}^{J,!}:= \cB_{\zeta, \xi, \rho}
$$
 where we are identifying $\zeta$, $\xi$ and $\rho$ with elements of $\bZ_{I,\vec n !}$ in the natural manner. Note the value of $\frac{\zeta}{(n-1)!} + \xi + \rho$ ranges exactly over the set $\bZ_{I,\overrightarrow n !}\setminus \bZ_{I,\overrightarrow{(n-1)}!}$ for $\zeta \in \mathbb{Z}_{J,\vec n}^*$, $\xi \in \mathbb{Z}_{J,\overrightarrow{(n-1)}!}$ and  $\rho \in  \mathbb{Z}_{L, \overrightarrow{(n-1)}! }$.  

Now, by the inductive hypothesis  $\cB_{\vec{0}} \simeq \Perf\left(\radice[\overrightarrow{(n-1)}!]{(X,D)}\right)$ carries a psod of type $(\mathbb{Z}_{I, \overrightarrow{(n-1)}!}, \leq^!) $
$$
\cB_{\vec{0}} \simeq \Perf\left(\radice[\overrightarrow{(n-1)}!]{(X,D)}\right) = \langle  \cA_\chi^{J,!}, \chi \in (\mathbb{Z}_{I, \overrightarrow{(n-1)}!}, \leq^!)  \rangle  ,
$$
such that:
\begin{itemize}
\item  $\mathcal{A}_{\vec{0}}^! \simeq \Perf(X)$, and 
\item  $\mathcal{A}_\chi^{J,!} \simeq \Perf(D_J)$ for all $\chi \in \mathbb{Z}_{J, \overrightarrow{(n-1)}!}^*$.
\end{itemize}  
As in the proof of Proposition \ref{sod!} we conclude the categories 
$\mathcal{A}_\chi^{J,!}, \chi \in \mathbb{Z}_{I, \vec{n}!}$ that we have just constructed  make up a psod of $\Perf(\radice[\vec{n}!]{(X,D)})$ of  type 
$(\mathbb{Z}_{I, \vec{n}!}, \leq^!)$ and this concludes the proof of part $(2)$ of the Proposition.

Now it is easy to proceed backwards and prove part $(1)$. For every $J \in S_I$ we define $\cA_J^!$   as the subcategory of $\Perf(\radice[\vec{n}!]{(X,D)})$ generated by the subcategories 
$
\cA_\chi^{J,!}$ {as $\chi$ varies in} $\mathbb{Z}_{J,\vec{n}!}^* \subset \mathbb{Z}_{I, \vec{n}!}$.
By construction the subcategories $\cA_J^!$ have the properties required by part $(1)$ of the Proposition, and this concludes the proof.  
\end{proof}

\begin{remark}
\label{goodproof}
It is possible to give a much simpler proof of Proposition \ref{sod!snc}  leveraging 
formal properties of the category of perfect complexes established 
in \cite{BFN}. This however requires to reduce generality, and assume that 
$X$ is a \emph{perfect stack} \cite[Definition 3.2]{BFN}: this is a large class of stacks containing for instance all quasi-compact schemes with affine diagonal. 

Let us sketch the argument assuming for simplicity that $D=D_1 \cup D_2$ has two components. 
Recall that  there is an equivalence
$$\radice[\vec{n}!]{(X,D)} \simeq \radice[n!]{(X,D_1)} \times_{X} \radice[n!]{(X,D_2)}.  
$$
Since $X$ is perfect, so are its root stacks. Then by \cite[Theorem 1.2]{BFN} there is an equivalence of categories
\begin{equation}
\label{eqtensorallproof}
\Perf \left ( \radice[\vec{n}!]{(X,D)} \right ) \simeq  \Perf \left (\radice[n!]{(X,D_1)} \right ) \otimes_{\Perf(X)}  \Perf \left ( \radice[n!]{(X,D_2)} \right ).
\end{equation}
 Now by Proposition \ref{prop: sodvsimplnc} we have psod-s 
\begin{equation}
\label{tensoffactors}
 \Perf \left (\radice[n!]{(X,D_1)} \right ) = \langle  \cB_{ \xi}^!, \xi \in (\mathbb{Z}_{n!}, \leq^!) \rangle, \quad  \Perf \left (\radice[n!]{(X,D_2)} \right ) = \langle  \cC_{ \xi'}^!, \xi' \in (\mathbb{Z}_{n!}, \leq^!) \rangle.
\end{equation}
 Equivalence (\ref{eqtensorallproof}) then implies that 
 $\Perf \left ( \radice[\vec{n}!]{(X,D)} \right )$ carries a psod whose semi-orthogonal factors are the tensor products of the factors appearing in (\ref{tensoffactors}): more precisely, for all 
 $\chi = (\xi, \xi' )  \in  \mathbb{Z}_{n!}\oplus\mathbb{Z}_{n!}$   set 
 $ \, 
 \cA_\chi := \cB_{\xi} \otimes_{\Perf(X)} \cC_{\xi'} 
 $.  
Then $\Perf \left ( \radice[\vec{n}!]{(X,D)} \right )$ carries a psod with the categories 
 $\cA_\chi$ as factors
\begin{equation}
\label{tensoffactorsII}
\Perf \left ( \radice[\vec{n}!]{(X,D)} \right ) =  \langle \cA_\chi = \cB_{\xi} \otimes_{\Perf(X)} \cC_{\xi'} \, , \,  \chi = (\xi, \xi') \in (\mathbb{Z}_{n!} \oplus \mathbb{Z}_{n!}, \leq^!) \rangle
\end{equation}
  For all $\xi, \xi' \in \mathbb{Z}_{n!}$ Theorem 1.2 of 
 \cite{BFN} yields equivalences \begin{itemize}
 \item $\cA_{(0,0)} = \cB_0 \otimes_{\Perf(X)} \cC_0 = \Perf(X) \otimes_{\Perf(X)} \Perf(X) \simeq \Perf(X)$,
 \item $\cA_{(\xi,0)} = \cB_{\xi} \otimes_{\Perf(X)} \cC_0 \simeq \Perf(D_1) \otimes_{\Perf(X)} \Perf(X) \simeq \Perf(D_1)$,  
 \item $\cA_{(0,\xi')} = \cB_{0} \otimes_{\Perf(X)} \cC_{\xi'} \simeq \Perf(X) \otimes_{\Perf(X)} \Perf(D_2) \simeq \Perf(D_2)$,
 \item $\cA_{(\xi,\xi')} = \cB_{\xi} \otimes_{\Perf(X)} \cC_{\xi'}\simeq \Perf(D_1) \otimes_{\Perf(X)}\Perf(D_2)\simeq \Perf(D_{\{12\}}).$  
  \end{itemize}
 Thus psod (\ref{tensoffactorsII}) has the same properties required by Proposition  \ref{sod!snc} and in fact it is easy to see that it coincides with it.  
 \end{remark}

\begin{theorem}
\label{mainsgncdiv}
 \mbox{ }
\begin{enumerate}[leftmargin=*] 
\item The category  $\Perf(\radice[\infty]{(X,D)})$ has a collection of subcategories $\cA_J^!$,  $J \in S_I$, such that: 
\begin{itemize}[leftmargin=*]
\item For all $J$, the subcategory $\cA_J^!$ is admissible. 
\item  For all $J \in S_I^*$ the category $\cA_J^!$  has a psod of type $((\mathbb{Q}/\mathbb{Z})_{J}^*, \leq^!)$ 
$$
\cA_J^! = \langle  \cA_{ \chi}^{J,!}, \chi \in ((\mathbb{Q}/\mathbb{Z})_{J}^*, \leq^!)  \rangle,
$$ 
and, for all $\chi \in (\mathbb{Q}/\mathbb{Z})_{J}^*$, there is an equivalence $\cA_{ \chi}^{J,!} \simeq \Perf(D_J)$.
\end{itemize}
\item  The  category $\Perf(\radice[\infty]{(X,D)})$ has a psod of type $((\mathbb{Q}/\mathbb{Z})_{I}, \leq^!)$ 
$$
\Perf(\radice[\infty]{(X,D)})=  \langle \cA_\chi^{J,!}, \chi \in  ((\mathbb{Q}/\mathbb{Z})_{I}, \leq^!)  \rangle,
$$
and for all $\chi \in (\mathbb{Q}/\mathbb{Z})_{J}^ *$ there is an equivalence $\cA_{ \chi}^{J,!} \simeq \Perf(D_J)$.
\end{enumerate}
\end{theorem}

  \begin{proof}
  The proof is the same as the one of  Proposition \ref{sod!Q}. 
  \end{proof}

  \begin{remark}
  \label{reminifnitelymanypsods}
The psod-s constructed in Proposition \ref{sod!Q} and Theorem \ref{mainsgncdiv} depend on the choice of a directed system which is cofinal in $\mathbb{N}^N$ preordered by divisibility. We chose the set of 
  diagonal vectors with factorial entries, but different choices were possible and would have given rise to different psod-s. At the level of additive inviariants, and in particular K-theory, all choices yield identical splitting formulas. All these different psod-s should be connected via mutation patterns which give rise to canonical identifications of semi-orthogonal factors: we   leave this to future investigation.   \end{remark}
 
\section{Beyond  simple normal crossing divisors}
\label{beyond}
In this section we study sod-s of infinite root stacks of general normal crossing divisors, 
and of divisors with simplicial singularities. In both cases we will be able to reduce to the simple normal crossing setting. 
At the same time  genuinely new phenomena will arise.

In the general normal crossing case, the factors making up the psod on the infinite root stacks are not equivalent to the category of perfect complexes on the strata, but on the \emph{normalization of the strata}. In \cite{scherotzke2016logarithmic} it is proven that the categories of perfect complexes on infinite root stacks are invariant under some class of log blow-ups. This is the key ingredient in the construction of these psod-s. However since the results in \cite{scherotzke2016logarithmic} require working over a field of characteristic $0$, we are bound to make the same assumption here.

Constructing sod-s for infinite root stacks in the setting of divisors with what we call ``simple simplicial singularities'' is more straightforward. We do this in Section \ref{reducetoncviacan}. The proof depends on a cofinality argument which allows us to reduce directly to the (simple) normal crossing case.  
\subsection{Root stacks of normal crossings divisors}
\label{ncd}
Throughout this section we work over a field $\kappa$ of characteristic zero.

We will establish sod-s for  root stacks of normal crossings divisors which are not necessarily simple. 
As explained in the preliminaries (Section \ref{sec:pre.non.simple}) root stacks of non-simple normal crossing divisors cannot be defined via an iterated root construction, as in \cite{bergh2016geometricity}. This has to do with the fact that the self-intersections of the divisors create higher codimensional strata which are not correctly accounted for if we just  take iterated roots of the divisors themselves.  We rely instead on the general definition of root stacks of logarithmic schemes introduced by Borne and Vistoli \cite{borne-vistoli}.

 Let $X$ be an algebraic stack, and $D$ a normal crossings divisor in $X$.
 In this section we consider the $r$-th root stacks $\radice[r]{(X,D)}$ for $r\in \bN$ described in Section \ref{sec:pre.non.simple}, and the infinite root stack $\radice[\infty]{(X,D)}=\varprojlim_r \radice[r]{(X,D)}$ of the log stack $(X, D)$. Note that if $D$ happens to be simple normal crossings, then $\radice[r]{(X,D)}$ coincides with the root stack $\radice[\vec r]{(X,D)}$ of the previous section, where $\vec r$ is the vector $(r,\hdots, r)\in \bN^N$ and $N$ is the number of irreducible components  of $D$. 
 
 \subsubsection{Strictification of normal crossing divisors}
 \label{secstrictify}
 Note that $D$ equips $X$ with a canonical stratification in locally closed substacks. This stratification is most easily expressed by saying that the locally closed strata are the connected substacks of $X$ where the rank of the log structure (i.e. of the sheaf $\overline{M}=M/\cO_X^\times$, using standard notation for log structures) remains constant. These are also the connected substacks where the number of points in the fiber of the normalization map $\widetilde{D}\to D$ remains constant.  We denote $S_D$ the set of the closures of these strata, and we set $S_D^*:=S_D - \varnothing.$

 \begin{lemma}
 \label{resolnsnc}
Let $X$ be an algebraic stack, and $D$ a normal crossings divisor in $X$. Then there exists a finite sequence of log blow-ups 
$$
(\widetilde{X}, \widetilde{D}):=(X_n, D_{X_n}) \xrightarrow{\pi_n} \ldots \xrightarrow{\pi_2} (X_1, D_{X_1}) \xrightarrow{\pi_1} (X_0, D_{X_0}):=(X,D)
$$
with the following properties: \begin{enumerate}
\item  For all $0 \leq i \leq n$,
 $X_i$ is an algebraic stack with a normal crossings divisor $D_i$. 
\item Denote by $S_{{D_{X_i}}}$ the set of closures of the strata.  Then the map 
$\pi_i: (X_i, D_{X_i}) \to (X_{i-1}, D_{X_{i-1}})$ is  the blow-up of a regular  stratum $S \in S_{D_{i-1}}$.  
\item The divisor $\widetilde{D}$ of $\widetilde{X}$ is simple normal crossing. 
\end{enumerate}
\end{lemma}
We say that $(\widetilde{X}, \widetilde{D})$ is a \emph{strictification} of $(X,D)$.  A proof of Lemma \ref{resolnsnc} 
is given in \cite{conrad}.  Let us explain briefly how to construct a strictification $(\widetilde{X}, \widetilde{D})$,  referring to 
\cite{conrad} for further details.

Let $S \in S_D$  be a stratum of codimension $d$.  We say that  $S$ is \emph{non simple} if $S$ is not a connected component of the intersection of $d$ distinct irreducible components of $D.$ Denote by $Z_d$ be the substack of $X$ given by the union of the non simple strata of $X$ of codimension at most $d.$ Let $m$ be the maximal index such that $Z_m \neq \varnothing.$ 
The substack $Z_m$ is regular. There is a sequence of inclusions
$$
Z_m \subset Z_{m-1} \subset \ldots \subset Z_1 = D
$$
Let $\pi_1: (X_1, D_1) \to (X,D)$ be the blow-up at $Z_m$. 
The strict transform of $Z_{m-1}$ under $\pi_1$ is regular, we denote it by $\widetilde{Z}_{m-1}$. We denote by $\pi_2:(X_2, D_2) \to (X_1, D_1)$ the blow-up at $\widetilde{Z}_{m-1}$.   
The strictification $(\widetilde{X}, \widetilde{D})$ is obtained by iterating this procedure for $m-1$ steps.

Let $\widetilde{I}$ be the set of irreducible components of 
$(\widetilde{X}, \widetilde{D})$, and let $S_{\widetilde{I}}$ be the set of strata. The iterated blow-up 
$$\widetilde{\pi}:=\pi_n \circ \ldots \circ \pi_1:(\widetilde{X}, \widetilde{D}) \longrightarrow (X,D)$$ 
maps strata  of $(\widetilde{X}, \widetilde{D})$  to strata of $(X,D)$.  
Let $S$ in $S_{D}$ be the image of the stratum $\widetilde{S}$ in $S_{\widetilde{I}}$. The restriction of 
$\widetilde{\pi}$  to $\widetilde{S}$ 
 $$
\widetilde{\pi}|_{\widetilde{S}}: \widetilde{S} \longrightarrow S
$$
can be described explicitly in terms of the geometry of iterated 
blow-ups. This requires some combinatorial book-keeping which, although elementary, quickly becomes quite intricate.

For  simplicity we will limit ourselves instead to give a qualitative description of the geometry of the strata of $(\widetilde{X}, \widetilde{S}).$ 
 We introduce an auxiliary class of stacks whose geometry is related in a simple way to the geometry of the strata of $(X,D).$ This is done in Definition \ref{auxiliarystacks}. It will be clear that 
 all strata of $(\widetilde{X}, \widetilde{S})$ are of this form, and this will be sufficient for our applications.

 \begin{definition}
 \label{auxiliarystacks} 
  Let $Y$ be an algebraic stack. We define recursively what it means for $Y$ to be \emph{of type $Z_i$}, where $m \geq i \geq 1$, starting from $i=m$: 
  \begin{enumerate}
 \item We say that $Y$ is \emph{of type $Z_m$} if there exists a stratum $S \in S_D,$ $\, S \subset Z_m \subset X\,$, and  maps 
$$
Y=Y_\gamma \xrightarrow{\gamma} Y_\beta \xrightarrow{\beta} Y_\alpha \xrightarrow{\alpha} S 
$$
where $\alpha,$ $\beta,$  and $\gamma$ are morphisms of the following type.
\begin{enumerate}
\item $Y_\alpha$ is the disjoint union of finitely many copies of the normalization of $S$. The map  
$\alpha : Y_\alpha \to S$ 
restricts to the normalization map on each copy.
 \item $\beta: Y_\beta \to Y_\alpha$ is a projective bundle.
 \item $\gamma: Y_\gamma \to Y_\beta$ is an iterated blow-up having the following two properties:
 \begin{itemize}
 \item it factors as a composite of blow-ups along regular centers, and
 \item  each of these centers is isomorphic to a projective bundle over a  stratum $S' \in S_D$ such that $S' \subseteq S$.
 \end{itemize}
\end{enumerate}
 \item  
We say that $X$ is \emph{of type $Z_i$} if there exists a stratum $S \in S_D,$ $\, S \subset Z_i \,$, and  maps 
$$
Y=Y_\gamma \xrightarrow{\gamma} Y_\beta \xrightarrow{\beta} Y_\alpha \xrightarrow{\alpha} S 
$$
where $\alpha,$ $\beta$  are as above, and  
$\gamma: Y_\gamma \to Y_\beta$ factors as a composite of blow-ups along regular centers, and each of these centers is a stack of type 
$Z_j$ for some $m \geq j >i.$ 
\end{enumerate} 
We say that $Y$ is \emph{adapted to $S_D$} if it is of type 
$Z_i$ for some $m \geq i \geq 1$.
 \end{definition}
It follows from the definition that if $X$ is  adapted to $S_D$ then $X$ is regular.  
\begin{lemma}
\label{adaptedcarsod}
If $X$ is adapted to $S_D$, $\Perf(X)$ admits a  sod such that all its semi-orthogonal factors are of the form 
$ \, 
\Perf(S^\vee)
\, $, 
where $S^\vee$ is the normalization of a stratum $S$ in $S_D.$
\end{lemma}
\begin{proof}
The category of perfect complexes over a disjoint union decomposes as a direct sum of the categories corresponding to each connected component. If $E \to X$ is a projective bundle then $\Perf(E)$ admits a 
sod where all semi-orthogonal factors are equivalent to $\Perf(X)$ \cite[Example 3.2]{kuznetsov2015semiorthogonal}. By Orlov blow-up formula, if $X$ is regular and 
$Y \to X$ is a blow-up along a regular center $Z \subset X$,  
$\Perf(Y)$ admits a sod whose factors are equivalent to either 
$\Perf(X)$ or $\Perf(Z)$  \cite[Theorem 3.4]{kuznetsov2015semiorthogonal}.  
Then the statement follows immediately from Definition \ref{auxiliarystacks}.
\end{proof}

\begin{lemma}
 \label{strataareadapted}
All strata $\widetilde{S}$ in $S_{\widetilde{I}}$ are adapted to $S_D.$ \end{lemma} 
 \begin{proof}
This follows because the 
strictification $(\widetilde{X}, \widetilde{D})$ is an iterated blow-up of $(X,D)$.   
 \end{proof}
 
\subsubsection{Reduction to the simple normal crossing case}
The next result, which was proved in \cite{scherotzke2016logarithmic}, allows us to 
reduce  to the simple normal crossing case, which was studied in Section  \ref{sncd}. We stress that, as in \cite{scherotzke2016logarithmic}, we need to assume that the ground ring $\kappa$ is a field of characteristic zero.  
 \begin{proposition}[Proposition 3.9, \cite{scherotzke2016logarithmic}]
 \label{excisionpara}
 Let $(X',D') \to (X,D)$ be a log blow-up such that $(X', D')$ is a again an algebraic stack with a normal crossing divisor. Then there is an equivalence of $\infty$-categories
\begin{equation}
\label{equationparabolicexcision}
\Perf(\radice[\infty]{(X',D')}) \simeq \Perf(\radice[\infty]{(X,D)}).
\end{equation}
 \end{proposition}
 \begin{proof}
 Proposition 3.9 of  \cite{scherotzke2016logarithmic} is formulated in terms of bounded derived categories and assumes that $X$ is regular: however the same proof works in this more general setting. \end{proof}
 
Let $X$ be an algebraic stack and let $D$ be a normal crossing divisor. Let 
$(\widetilde{X}, \widetilde{D})$ be the strictification of $(X,D)$ constructed in Section \ref{secstrictify}. 
As before let $\widetilde{I}$ be the set of divisors of $(\widetilde{X}, \widetilde{D})$, and let $S_{\widetilde{I}}$ 
be the preordered set of strata, and $S_{\widetilde{I}}^* = S_{\widetilde{I}} - \{\varnothing\}$.

\begin{theorem}
\label{mainncdivsod}
\mbox{ }
\begin{enumerate}[leftmargin=*]
\item
 The category  $\Perf(\radice[\infty]{(X,D)})$ has a collection of subcategories $\cA_J^!$, $J \in S_{\widetilde{I}}$, such that: 
\begin{itemize}[leftmargin=*]  
\item For all $J$, $\cA_J^!$ is admissible.
\item The category $\cA_{\varnothing}^! = \Perf(\widetilde{X})$ has a psod 
whose semi-orthogonal factors are 
\begin{itemize}
\item[(a)]  $\Perf(X)$, and 
\item[(b)]  factors of the form $ \, 
\Perf(S^\vee)
\, $, 
where $S^\vee$ is the normalization of a stratum $S$ in $S_{D}^*$.
\end{itemize}
\item  For all $J \in S_{\widetilde{I}}^*$ the category $\cA_J^!$  has a psod of type $((\mathbb{Q}/\mathbb{Z})_{J}^*, \leq^!)$ 
$$
\cA_J^! = \langle  \cA_{ \chi}^{J,!}, \chi \in ((\mathbb{Q}/\mathbb{Z})_{J}^*, \leq^!)  \rangle. 
$$ 
Additionally, for all $\chi \in (\mathbb{Q}/\mathbb{Z})_{J}^*$, $\cA_{ \chi}^{J,!}$ has  a psod 
whose semi-orthogonal factors are of the form  
$ \, 
\Perf(S^\vee)
\, $, 
where $S^\vee$ is the normalization of a stratum $S$ in $S_{D}^*$.
\end{itemize}
\item The  category $\Perf(\radice[\infty]{(X,D)})$ has a psod whose semi-orthogonal factors are given by the 
factors of the psod-s of $\cA_{\varnothing}^! = \Perf(\widetilde{X})$ and $\cA_{ \chi}^{J,!}$ described above.
\end{enumerate}
\end{theorem}
\begin{proof}
By Proposition \ref{excisionpara} there is an equivalence 
$ \, 
\Perf(\radice[\infty]{(\widetilde{X},\widetilde{D})}) \simeq \Perf(\radice[\infty]{(X,D)})
\, $. 
Since $(\widetilde{X},\widetilde{D})$ is simple normal crossing, we can apply Theorem 
\ref{mainsgncdiv}. Let us use the notations of Theorem 
\ref{mainsgncdiv}: recall that the semi-orthogonal factors $\cA_{ \chi}^{J,!}$, $J \subset \widetilde{I}$, considered  there are equivalent to 
$\Perf(\widetilde{D}_J)$. The stack $\widetilde{D}_J$ is a stratum in $S_{\widetilde{I}}$ and thus, by Lemma \ref{strataareadapted}, 
 is adapted to $S_D$. By Lemma \ref{adaptedcarsod}, $\Perf(\widetilde{D}_J)$ carries a psod whose factors are of the form 
$ \, 
\Perf(S^\vee)
\, $, 
where $S^\vee$ is the normalization of a stratum $S$ in $S_D$. This concludes the proof. 
 \end{proof}
 
 \subsection{Root stacks of divisors with simplicial singularities}
 \label{reducetoncviacan}

Let $X$ be a log scheme with a \emph{simplicial} log structure. This means that it is fine and saturated, and the stalks of the sheaf $\overline{M}=M/\cO_X^\times$ are simplicial monoids. Recall that a sharp fine saturated monoid $P$ is simplicial if the extremal rays of the rational cone $P_\bQ\subset P^\gp\otimes_\bZ \bQ$ are linearly independent.

\begin{proposition}
Let $X$ be a log scheme with simplicial log structure. Then there is a canonical minimal Kummer extension $\overline{M}\subseteq \cF$ where $\cF$ is a coherent sheaf of monoids on $X$ with free stalks.
\end{proposition}

The minimality in the statement means that every Kummer extension $\overline{M}\to N$ where $N$ is coherent with free stalks factors uniquely as $\overline{M}\to\ \cF\to N$.

\begin{proof}
Because of the uniqueness part of the statement, it suffices to do the construction locally. Assume therefore that we have a global chart $X\to [\Spec \kappa[P]/D(P^\gp)]$ for $X$ (in the sense of \cite[Section 3.3]{borne-vistoli}), where $P$ is a simplicial monoid, and $D(P^\gp)$ denotes the Cartier dual $\Hom(P^\gp, \bG_m)$ of $P^\gp$.

Let $p_1,\hdots, p_n$ be the primitive generators in the lattice $P^\gp$ of the extremal rays of the rational cone $P_\bQ\subset P^\gp\otimes_\bZ\bQ$ generated by $P$. These are indecomposable elements of $P$. Let $q_1,\hdots, q_m$ be the remaining indecomposable elements of $P$.
By simpliciality of $P$, for every index $i=1,\hdots, m$ we can express $q_i$ uniquely as a rational linear combination of the $p_j$. Let us write
$
q_i=\sum_{j=1}^n \frac{a_{ij}}{b_{ij}}\cdot p_j
$
where for every pair of indices $\{i,j\}$, $a_{ij}$ and $b_{ij}$ are coprime non-negative integers. For every $j=1,\hdots, n$ let $c_j$ be the lcm of the set $\{b_{1j},\hdots, b_{mj}\}$. Consider then the submonoid of $P^\gp\otimes_\bZ\bQ$ generated by the vectors $p_1/c_1,\hdots, p_n/c_n\in P_\bQ$. This is a free monoid $\bN^n$ containing $P$, and the inclusion $P\to \bN^n$ is a Kummer morphism. It is easy to check that it is minimal among Kummer morphisms from $P$ to a free monoid.

Now consider the map $P\to \overline{M}(X)$ corresponding to the chart for the log structure of $X$ that we fixed above. Recall from \cite[Section 3.3]{borne-vistoli} that this map being a chart exactly means that the induced morphism $\phi\colon \underline{P}\to \overline{M}$ from the constant sheaf $\underline{P}$ is a cokernel, i.e. it induces an isomorphism $\underline{P}/\ker \phi\cong \overline{M}$, where $\ker \phi$ denotes the preimage of the zero section (this is not always true in the category of monoids). In the same way, the map $P_\bQ\to \overline{M}(X)_\bQ$ gives a chart for the sheaf $\overline{M}_\bQ$ over $X$.
For the Kummer extension $P\subseteq \bN^n\subset P_\bQ$ constructed above, let us consider the image $\cF$ of the subsheaf $\underline{\bN^n}\subset \underline{P_\bQ}$ in $\overline{M}_\bQ$. It is not hard to check that the natural map $\bN^n\to \cF(X)$ is a chart for $\cF$ (i.e. the corresponding $\underline{\bN^n}\to \cF$ is a cokernel), and that the induced morphism $\overline{M}\to \cF$ is a Kummer extension with the universal property of the statement.
\end{proof}
 
\begin{remark}
The previous proposition is also true for log algebraic stacks with simplicial log structure, by passing to a smooth presentation and using uniqueness of the Kummer extension to produce descent data.
\end{remark}

\begin{definition}
\label{canonicalstack}
Let $X$ be a log algebraic stack with simplicial log structure, and let $\overline{M}\to \cF$ be the canonical Kummer extension constructed above. We call the root stack $\radice[\cF]{X}$ the \emph{canonical root stack} of $X$.
\end{definition} 
 
\begin{remark}
Assume that $\kappa$ is a field of characteristic $0$. If $X=\Spec \kappa[P]$ where $P$ is a simplicial monoid, then we can consider the minimal Kummer extension $P\subseteq \bN^n$ constructed in the proof of the previous proposition. The canonical root stack in this case is the quotient $[\Spec \kappa[\bN^n]/D(\bZ^n/P^\gp)]$. Note that the quotient $\bZ^n/P^\gp$ is a finite group, so this quotient is a smooth Deligne--Mumford stack. In fact, in this case it coincides with the \emph{canonical stack} of Fantechi--Mann--Nironi \cite{fantechi2010smooth}, which justifies its name.
\end{remark}

Note that the natural log structure of the canonical root stack $\radice[\cF]{X}$ is locally free by construction. 

\begin{definition}\label{def:simpl.sing}
Let $X$ be a scheme over $\kappa$, and $D\subset X$ be an effective Cartier divisor. We say that $D$ has  \emph{simple simplicial singularities} if:
\begin{itemize}
\item the compactifying log structure $M_D$ (whose definition is recalled in Section \ref{sec:logstr}) is simplicial, 
\item the tautological log structure of the canonical root stack $X'=\radice[\cF]{(X,D)}$ is given by a  simple  normal crossings divisor $D'\subset X'$ (in the sense of Definition \ref{def:nc}).
\end{itemize}
\end{definition}

\begin{remark}
In Definition \ref{def:simpl.sing} we have made the assumption that $D' \subset X'$ is a \emph{simple} normal crossing divisor in order to simplify the exposition. However, in characteristic $0$, the results from this section could be formulated more generally for the case when $D'$ is a general normal crossing divisor. We leave it to the interested reader 
to recast Theorem \ref{maininfpsodsimpli} below in this greater generality using as input the psod constructed in the general normal crossing setting in Theorem \ref{mainncdivsod}. 

In the rest of the paper, for convenience we will abbreviate ``simple simplicial singularieties'' by just ``simplicial singularities''.
\end{remark}

If $X$ is an algebraic stack and $D\subseteq X$ is an effective Cartier divisor, we say that $D$ has simplicial singularities if the pull-back of $D$ to some smooth presentation $U\to X$, where $U$ is a scheme, has simplicial singularities in the sense of the previous definition.

\begin{remark}
Assume that 
$D$ is an effective Cartier divisor on $X$ and that for every $x \in D$ the pair $(X,D)$ is \'etale locally around $x$ isomorphic to the pair $(\Spec \kappa[P] \times \bA^n,\Delta_P\times \bA^n)$ for some simplicial monoid $P$ and $n\in \bN$. Then $D$ has simplicial singularities in the sense Definition \ref{def:simpl.sing}.
\end{remark}

Now assume that $D\subset X$ has simplicial singularities, and consider the canonical root stack $(X',D')=\radice[\cF]{(X,D)}$. Since $(X',D')\to (X,D)$ is a root stack morphism we have a canonical isomorphism $\radice[\infty]{(X',D')}\simeq \radice[\infty]{(X,D)}$ and therefore in order to study the category of perfect complexes on $\radice[\infty]{(X,D)}$, we can pass to $(X',D').$ For future reference we state this as the following theorem.
We will use it in Section \ref{sec:simplicial.applications} to obtain a formula for the Kummer flat K-theory of $X$.

Let $I$ be the set of irreducible components of $D'$ and let $S_{I'}$ be the preorder of strata of $(X',D')$. 

\begin{theorem}
\label{maininfpsodsimpli}
 \mbox{ }
\begin{enumerate}[leftmargin=*] 
\item The category  $\Perf(\radice[\infty]{(X,D)})$ has a collection of subcategories $\cA_{J'}^!$,  $J' \in S_{I'}$, such that: 
\begin{itemize}[leftmargin=*]
\item For all $J'$, the subcategory $\cA_{J'}^!$ is admissible. 
\item  For all $J' \in S_{I'}^*$ the category $\cA_{J'}^!$  has a psod of type $((\mathbb{Q}/\mathbb{Z})_{J'}^*, \leq^!)$ 
$$
\cA_{J'}^! = \langle  \cA_{ \chi}^{J',!}, \chi \in ((\mathbb{Q}/\mathbb{Z})_{J'}^*, \leq^!)  \rangle,
$$ 
and, for all $\chi \in (\mathbb{Q}/\mathbb{Z})_{J'}^*$, there is an equivalence $\cA_{ \chi}^{J',!} \simeq \Perf(D_{J'})$.
\end{itemize}
\item  The  category $\Perf(\radice[\infty]{(X,D)})$ has a psod of type $((\mathbb{Q}/\mathbb{Z})_{I'}, \leq^!)$ 
$$
\Perf(\radice[\infty]{(X,D)})=  \langle \cA_\chi^{J',!}, \chi \in  ((\mathbb{Q}/\mathbb{Z})_{I'}, \leq^!)  \rangle,
$$
and for all $\chi \in (\mathbb{Q}/\mathbb{Z})_{J'}^ *$ there is an equivalence $\cA_{ \chi}^{J',!} \simeq \Perf(D_{J'})$.
\end{enumerate}
\end{theorem}
\begin{proof}
This follows immediately from  the equivalence
$
\Perf(\radice[\infty]{(X,D)}) \simeq \Perf(\radice[\infty]{(X',D')})
$
and the fact that by 
Theorem \ref{mainsgncdiv} we can equip $\Perf(\radice[\infty]{(X',D')})$ with a psod of type $S_{I'}.$   
\end{proof}

\begin{remark}
It would be very interesting to extend Theorem \ref{maininfpsodsimpli} to the general log smooth setting, however new ideas would be required. A promising avenue for further investigation might come from work of Satriano \cite{satriano2013canonical}, which essentially extends the theory of canonical root stacks given in Definition \ref{canonicalstack} to general (non-necessarily simplicial) toric singularities. 
\end{remark}

\section{Non-commutative motives of log schemes}
\label{ncmols}
In this section we  associate to log stacks  objects in the category of non-commutative motives. We start by giving a brief summary of the theory which  follows closely the treatment given in \cite[Section 5]{HSS}. The reader can find  accounts of  the theory of non-commutative motives  in \cite{BGT} and \cite{HSS}. 

Let $\cT_\infty$ be the $\infty$-category of spaces, which is the homotopy coherent nerve of the simplicial category of Kan complexes. Let $\cS_\infty$ be the $\infty$-category of spectra. The  
category $\cS_\infty$ is the stabilization of $\cT_\infty,$ and we denote by 
$
\Sigma_+^\infty\colon \cT_\infty \rightarrow \cS_\infty
$
the stabilization functor. 
\begin{definition}
  Let $\cC$ be a small $\infty$-category.  We denote:
 \begin{itemize}[leftmargin=*]
 \item by $\mathrm{PSh}(\cC)=\mathrm{Fun}(\cC^{\mathrm{op}}, \cT_\infty)$ the $\infty$-category of presheaves of $\infty$-groupoids over $\cC$,
 \item by $\mathrm{PSh}_{\cS_\infty}(\cC)= \mathrm{Fun}(\cC^{\mathrm{op}}, \cS_\infty)$ the $\infty$-category of presheaves of spectra over $\cC$,
 \item by $\Sigma_+^\infty\colon \mathrm{PSh}(\cC)  \longrightarrow \mathrm{PSh}_{\cS_\infty}(\cC)$ 
 the  functor given, on objects, by stabilization.
 \end{itemize}
\end{definition}
Let 
$(\mathrm{Cat}_{\infty, \kappa}^{\mathrm{perf}})^\omega$ be the subcategory of compact objects in $\mathrm{Cat}_{\infty, \kappa}^{\mathrm{perf}}.$ Let $\phi$ be  the composite  
$$
\phi\colon \mathrm{Cat}_{\infty, \kappa}^{\mathrm{perf}} \longrightarrow   \mathrm{PSh}((\mathrm{Cat}_{\infty, \kappa}^{\mathrm{perf}})^\omega) \xrightarrow{\Sigma_+^\infty} 
\mathrm{PSh}_{\cS_\infty}((\mathrm{Cat}_{\infty, \kappa}^{\mathrm{perf}})^\omega),
$$
where the first arrow is the restriction of  the Yoneda to the subcategory $(\mathrm{Cat}_{\infty, \kappa}^{\mathrm{perf}})^\omega.$ 

\begin{definition}
\label{defaddmot}
The \emph{category of additive motives} $\mathrm{Mot}^{\mathrm{add}}$ is the localization 
of $\mathrm{PSh}_{\cS_\infty}((\mathrm{Cat}_{\infty, \kappa}^{\mathrm{perf}})^\omega)$ at the class of morphisms 
$
\phi(\cB) / \phi(\cA) \rightarrow \phi(\cC)
$
which are induced by split exact sequences 
$ \,   \cA \to \cB \to \cC   \,$    in $\mathrm{Cat}_{\infty, \kappa}^{\mathrm{perf}}.$ 
\end{definition}

Let $\cU$ be the composite
$ \, \mathrm{Cat}_{\infty, \kappa}^{\mathrm{perf}} \stackrel{\phi} \rightarrow 
\mathrm{PSh}_{\cS_\infty}((\mathrm{Cat}_{\infty, \kappa}^{\mathrm{perf}})^\omega)  \rightarrow \mathrm{Mot}^{\mathrm{add}} 
\, $,  
where the second arrow is given by the localization functor. An \emph{additive invariant}  
is a functor  $\mathrm{H}\colon\mathrm{Cat}_{\infty, \kappa}^{\mathrm{perf}} \to \cP$, where $\cP$ is a stable presentable $\infty$-category, that preserves zero objects and filtered colimits, and
that maps split exact sequences to cofiber sequences. The functor $\cU$ is the universal additive invariant. We formulate the precise statement below. 

\begin{proposition}[Theorem 5.12 \cite{HSS}] 
Let $\cP$ be a presentable and stable $\infty$-category, and let 
$
\mathrm{H}\colon \mathrm{Cat}_{\infty, \kappa}^{\mathrm{perf}} \longrightarrow \cP
$
be an additive invariant. Then $\mathrm{H}$  factors uniquely as a composition $$
 \xymatrix{
\mathrm{Cat}_{\infty, \kappa}^{\mathrm{perf}} \ar[r]^-{\mathrm{H}} \ar[d]_-{\cU} & \cP \\ 
 \mathrm{Mot}^{\mathrm{add}} \ar[ur]_-{\mathrm{\overline H}}
 }
 $$
where $\mathrm{\overline H}$ is a colimit-preserving functor of presentable categories.   
\end{proposition}

\begin{remark}
\label{splitsplit}
Let $ 
\mathrm{H}\colon \mathrm{Cat}_{\infty, \kappa}^{\mathrm{perf}}  \to \cP
$ be an additive invariant. If 
$
\cA \stackrel{F} \longrightarrow 
\cB \stackrel{G} \longrightarrow \cC
$
is a split exact sequence in $\mathrm{Cat}_{\infty, \kappa}^{\mathrm{perf}},$ then there is a canonical splitting  
$ \, 
\mathrm{H}(\cB) \simeq 
\mathrm{H}(\cA) \oplus \mathrm{H}(\cC) \, . 
$ 
Indeed, let $(F)^R$  be the right adjoint  of $F.$ Since $\mathrm{H}$ is additive 
\begin{equation}
\label{eqeqsplitsplit}
\mathrm{H}(\cA) \stackrel{\mathrm{H}(F)} \longrightarrow 
\mathrm{H}(\cB) \stackrel{\mathrm{H}(G)} \longrightarrow \mathrm{H}(\cC)
\end{equation}
is a fiber sequence in $\cP.$ Further 
$\mathrm{H}((F)^R)$ is a section of $\mathrm{H}(F).$ Thus  (\ref{eqeqsplitsplit}) splits, and  $\mathrm{H}(\cB)$ decomposes as the direct sum of 
 $\mathrm{H}(\cA)$ and $\mathrm{H}(\cC).$
\end{remark}

\begin{lemma}
\label{motsplitlem}
Let $\cC= \langle  \cC_x,  x \in P \rangle$ be a stable $\infty$-category equipped with a psod of type 
$P,$ and assume that $P$ is finite and directed (i.e. it admits an order-reflecting map to the natural numbers $\bN$, ordered in the standard manner). Then there is an equivalence 
$
\cU(\cC) \simeq \bigoplus_{x \in P }\cU(\cC_x).
$ 
\end{lemma}
\begin{proof}
Note that if $P$ is directed we can choose a numbering 
$
\{p_0, \ldots, p_m \}
$
of its elements with the property that, if $i < j,$ then $\cC_{p_i} \subseteq \cC_{p_j}^\bot.$  Thus we can write down a 
sod $\langle \cC_{p_0}, \ldots, \cC_{p_m} \rangle$ for $\cC$.
Then the second statement is a simple consequence of  Remark \ref{splitsplit}.  
\end{proof}

\subsection{The non-commutative motive of a log stack}
We introduce the following notations.
\begin{itemize}
\item If $X$ is a stack we set 
$\cU(X):=\cU(\Perf(X))$.  
\item If $X$ is a log algebraic stack, we denote by
$ X_{\mathrm{Kfl}}$ the ringed Kummer flat topos  over $X$. We set 
$\cU(X_{\mathrm{Kfl}}):=\cU(\Perf(X_{\mathrm{Kfl}}))$.
\end{itemize}
We will apply Lemma \ref{motsplitlem} to the psod-s we constructed in sections \ref{sncd} and \ref{ncd}.     
Let $(X, D)$ be a log stack given by an algebraic stack $X$ equipped with   a  normal crossings divisor 
$D.$ Let $S_I$ be the preorder of strata of $(X,D)$. In the  statement below we  use the same notations as  in 
Section \ref{sncd}.

\begin{corollary}
\label{ncmotdirectsum}
Let $(X, D)$ be a log stack given by an algebraic stack $X$ equipped with a simple normal crossings divisor 
$D.$ 
Then there is an equivalence $$
\cU((X,D)_{\mathrm{Kfl}}) \simeq \cU(X) \bigoplus \Big ( \bigoplus_{S \in S_I^*} \Big ( \bigoplus_{\chi \in (\mathbb{Q}/\mathbb{Z})_I^*} \cU(S) \Big ) \Big ). 
$$ 
\end{corollary}
\begin{proof}
Using Lemma \ref{kflinf}  we obtain  an equivalence
$
\Perf((X,D)_{\mathrm{Kfl}}) \simeq \Perf(\radice[\infty]{(X,D)})
$.
The non-commutative motive of  $\Perf(\radice[n!]{(X,D)})$ has a decomposition as in the statement (except the indexing set $(\mathbb{Q}/\mathbb{Z})_I^*$ has to be replaced by $\mathbb{Z}_{I, n!}^*$): this follows from   Proposition \ref{sod!snc} and 
Lemma \ref{motsplitlem}. The category $\Perf(\radice[\infty]{(X,D)})$ is a filtered colimit of the categories $\Perf(\radice[n!]{(X,D)})$. Also, by Theorem \ref{mainsgncdiv}, it carries a psod that is the colimit of the psod-s of the categories 
$\Perf(\radice[n!]{(X,D)}).$ The statement follows because, by construction, $\cU(-)$ commutes with filtered colimits.
\end{proof}

Formulas exactly paralleling Corollary \ref{ncmotdirectsum}  can be obtained in the general normal crossing setting. 
This is straightforward, as explained in Section \ref{ncd}, but involves  messy combinatorics.  
For this reason we give instead a simplified statement, which is contained in   Corollary \ref{ncmotdirectsumII} below. 
 
In the following statement, if $S$ is a stratum of $(X,D)$, we denote by $S^\vee$ its normalization. 
\begin{corollary}
\label{ncmotdirectsumII}
Assume that the ground ring $\kappa$ is a field of characteristic $0$. Let $(X, D)$ be a log stack given by an algebraic stack $X$ equipped with a normal crossings divisor 
$D$.  
Then for each $S \in S_D^*$ there exists an infinite countable set $\cI_S$, such that there is an equivalence 
 $$ 
\cU((X,D)_{\mathrm{Kfl}}) \simeq \cU(X) \bigoplus \Big ( \bigoplus_{S \in S_D^*} \Big ( \bigoplus_{j \in \cI_S} \cU(S^\vee) \Big ) \Big ). 
$$  
\end{corollary}

By the universal property of $\cU,$ Corollary  \ref{ncmotdirectsum} and \ref{ncmotdirectsumII} 
imply uniform direct sum decompositions across all additive invariants. As the case of algebraic K-theory is especially important we formulate it explicitly in the following corollary: this generalizes Hagihara and Nizio{\l}'s  as we drop the simplicity assumption on $D$, $X$ can be a stack,  and $X$ need not be regular away from $D$.   
 
\begin{corollary}
\label{ncmotdirectsum1}
\mbox{}
\begin{itemize}
\item 
Let $(X, D)$ be a log stack given by an algebraic stack $X$ equipped with a simple normal crossings divisor 
$D.$  
Then there is a direct sum decomposition of spectra \begin{equation}
\label{formula1234}
K((X,D)_{\mathrm{Kfl}}) \simeq K(X) \bigoplus \Big ( \bigoplus_{S \in S_D^*} \Big ( \bigoplus_{\chi \in (\mathbb{Q}/\mathbb{Z})_S^*} K(S) \Big ) \Big )   .
\end{equation}
\item Assume that the ground ring $\kappa$ is a field of characteristic $0$. Let $(X, D)$ be a log stack given by an algebraic stack $X$ equipped with a normal crossings divisor 
$D.$  
Then there is a direct sum decomposition of spectra
$$
K((X,D)_{\mathrm{Kfl}}) \simeq K(X) \bigoplus \Big ( \bigoplus_{S \in S_D^*} \Big ( \bigoplus_{j \in \cI_S} K(S^\vee) \Big )\Big )  .$$ 
\end{itemize}
\end{corollary}

\begin{remark}\label{rmk:ket}
The previous result has an analogue for the Kummer \'etale topos of $(X,D)$, parallel to the second part of the statement of Theorem 1.1 of \cite{Ni1} and the Main Theorem of \cite{hagihara}. In characteristic zero there is no difference, so this comment is relevant only if $\kappa$ has positive or mixed characteristic, and, assuming that $D$ is equicharacteristic as in  \cite{Ni1}, $\bQ/\bZ$ has to be replaced by $(\bQ/\bZ)'=\bZ_{(p)}/\bZ$ (where $p$ is the characteristic over which $D$ lives) in the formulas above.
 
This analogous formula for the Kummer \'etale K-theory follows from our methods, starting from the analogue of Proposition \ref{kflinf} for the Kummer \'etale site and a restricted version $\radice[\infty']{(X,D)}$ of the infinite root stack, where we take the inverse limit only of root stacks $\radice[r]{(X,D)}$ where $r$ is not divisible by $p$.  
 This statement in turn follows from the same argument used in the proof of \ref{kflinf}, after proving Theorem 6.16 and Corollary 6.17 of \cite{TV} for the Kummer \'etale site and this restricted root stack. We leave the details to the interested reader.
 \end{remark}
 
 \subsection{Log schemes with simplicial log structure} \label{sec:simplicial.applications}  
 
 Let $D$ be a divisor with simplicial singularities in an algebraic stack $X$. Consider the associated log stack $(X,D)$ with simplicial log structure and let 
 $\radice[\mathcal{F}]{(X,D)}$ be its canonical root stack, which  is of the form $(X',D')$, where $D'$ is a normal crossings divisor on $X'$. 
 Our techniques allow us to derive a decomposition formula for the Kummer flat K-theory of  $X$ in terms of the geometry of $(X',D')$. In fact, we can formulate two such results.  
 
 Let $S_{I'}$ be the preorder of strata of $(X',D')$. By Theorem \ref{maininfpsodsimpli} $\Perf((X,D)_{\mathrm{Kfl}})$ carries a canonical psod of type $S_{I'},$ and this yields a decomposition of the noncommutative motive $\cU((X,D)_{\mathrm{Kfl}}).$  In particular, we obtain an equivalence of spectra  
\begin{equation}
\label{KKalg}
K((X,D)_{\mathrm{Kfl}}) \simeq K((X',D')_{\mathrm{Kfl}})
\simeq K(X') \bigoplus \Big ( \bigoplus_{J \in S_{I'}^*} \Big ( \bigoplus_{\chi \in (\mathbb{Q}/\mathbb{Z})_{J}^*} K(S') \Big ) \Big ).
\end{equation}  Under some additional assumptions on 
 $(X,D)$ however, we can do better. We can refine (\ref{KKalg}) to a second decomposition formula for the (\emph{complexified}) Kummer flat K-theory of $(X,D)$ which is formulated in terms of the  \emph{G-theory} of the strata of $X$ determined by the divisor $D$ via the associated log structure. We do this in Proposition \ref{propKKalg} below.

We will make use of results proved in \cite{krishna2017atiyah}. 
Let $(X, D)$ be a log scheme 
where $D$ has simplicial singularities. We assume that \begin{itemize}
\item[($\star$)]
$\kappa=\bC$, $X$ is quasi-projective, and $(X, D)$ has a global chart $X\to [\Spec \bC[P]/D(P^\gp)]$ for a simplicial monoid $P$, which is a smooth morphism. 
\end{itemize}
This implies that the canonical root stack  $\radice[\cF]{(X,D)}$  
is a   quotient stack $[(Y,E)/G]$ where $Y$ is a smooth quasi-projective scheme, $E\subset Y$ is a simple normal crossings divisor and $G$ is a finite group acting on the pair $(Y,E)$. Also, $(X,D)$ is obtained by taking the coarse quotient for the action of $G$. We denote $\radice[\cF]{(X,D)}$ by $(X',D')$, where $X'=[Y/G]$ and $D'$ is the induced simple normal crossings divisor $[E/G]$.
In particular, $X'$ is smooth and has a quasi-projective coarse moduli space.

Let $L$, $I'$ and $I$ be, respectively, the set of irreducible components of the divisors 
$ \, 
E \subset Y, \, D' \subset X', \,$ and $D \subset X$. As usual we denote the corresponding sets of strata by 
$S_L$, $S_{I'}$ and $S_I$. There is a canonical bijection between the sets $S_{I'}$ and $S_{I}$. The group  
$G$ acts on  $ S_L,$ and there is a map  
$
 p\colon S_L \to S_L/G \cong S_{I'} \cong S_{I} 
 $
  induced by the quotient 
 $Y \to [Y/G] \simeq X'.$ 
 Let $F$ be the \emph{disjoint} union of the sets of irreducible components of the fixed loci  $Y^g \subset Y,$ as $g$ ranges over $G \setminus \{1_G\}.$  The  fixed loci are  strata of $Y$, and this gives a map $F \to S_L$. In general this  is not an injection, as the same stratum of 
 $Y$ might appear more than once in $F$  if it is   fixed by  several distinct  group elements.

 If $U$ is a  stratum  of $Y$ we introduce the following notations,  \begin{itemize}
 \item  $
 F_U := \{T \in F \mid U \subseteq T \} \subseteq F ,
 $
 \item $
 (\mathbb{Q}/\mathbb{Z})_{F_U}^* := (\mathbb{Q}/\mathbb{Z})_U^* 
 \coprod \Big ( \coprod_{T \in F_U} (\mathbb{Q}/\mathbb{Z})_T^* \Big ) 
 $,  where the index sets on the right hand side are written according to the convention explained in Remark \ref{remremnotation} : that is, they are labeled by strata, rather than by subsets of $L$. We will follow this convention throughout Section \ref{sec:simplicial.applications}. 
 \end{itemize}
We extend this to $X$ using the map 
$p\colon S_E \to S_D.$ Namely, if $S$ is in $S_D$ we set:  
\begin{itemize}
 \item  $
 F_S := \coprod_{U \in 
 p^{-1}(S)}   p(F_U),    
 $  $F_S $ is the \emph{disjoint} union of the sets $p(F_U)$,
 \item $
 (\mathbb{Q}/\mathbb{Z})_{F_S}^* := (\mathbb{Q}/\mathbb{Z})_S^* 
 \coprod \Big ( \coprod_{T \in F_S} (\mathbb{Q}/\mathbb{Z})_T^* \Big ) .
 $
 \end{itemize}
 If $X$ is an algebraic stack, in the statement of Proposition \ref{propKKalg}, and throughout its proof, we denote  $G_i(X)$ the $i$-th G-theory group of $X$ with \emph{complex coefficients}, i.e.
 $
 G_i(X):=K_i(\mathrm{Coh}(X)) \otimes \mathbb{C}.
 $ 
\begin{proposition}
\label{propKKalg}
 Let $(X, D)$ be a log scheme given by a divisor $D$ with simplicial singularities, satisfying   assumption  \emph{(}$\star$\emph{)}.    
Then for all $i \in \mathbb{N}$ there is a   direct sum decomposition 
\begin{equation}
\label{complex-decompo}
K_i((X,D)_{\mathrm{Kfl}}) \otimes \mathbb{C} \cong G_i(X)  \bigoplus \Big ( \bigoplus_{S \in S_I^*} \Big ( \bigoplus_{\chi \in (\mathbb{Q}/\mathbb{Z})_{F_S}^*} G_i(S)  \Big ) \Big ). 
\end{equation}
\end{proposition} 
Proposition \ref{propKKalg} follows from (\ref{KKalg}) and an Atiyah--Segal-type formula  expressing the G-theory of a stack in terms of the G-theory of the coarse moduli of its inertia. The most general version of such a formula in the   literature was obtained  in  \cite{krishna2017atiyah} (and holds over $\mathbb{C}$). The assumptions we 
impose on $(X, D)$ mirror  the assumptions made in \cite{krishna2017atiyah}: they can be relaxed if more general versions 
of the Atiyah--Segal decomposition  will become available in the future. 

\begin{proof}
For simplicity we assume that  $X$ is an affine toric variety with simplicial singularities. Then  $Y=\mathbb{A}^n$,  
and $G$ is a finite group acting torically. 
The proof in the general case is the same, except the book-keeping of the summands on the right-hand side of 
(\ref{complex-decompo}) requires some extra care. 

Throughout the proof, if  $X$ is a stack we denote by
$IX$  the inertia of $X$, and  by 
$\widetilde{IX}$ its coarse moduli space.  For all $i \in \mathbb{N},$ formula (\ref{KKalg}) yields an isomorphism of abelian groups 
\begin{equation}
\label{prelstepeq}
K_i((X,D)_{\mathrm{Kfl}})  \cong K_i(X') \bigoplus \Big ( \bigoplus_{S' \in S_{I'}^*} \Big ( \bigoplus_{\chi \in (\mathbb{Q}/\mathbb{Z})_{S'}^*} K_i(S') \Big ) \Big ) 
\end{equation}
All the strata $S' \in S_{I'}$ are smooth, and thus their G-theory and K-theory are the same. By Theorem 1.1 of \cite{krishna2017atiyah}, for all $i \in \mathbb{N}$ there is an isomorphism  
$$K_i(X') \otimes \mathbb{C} = G_i(X')   \cong G_i(\widetilde{IX'})  \cong \bigoplus_{g \in G} G_i(Y^g/G)  = 
G_i(X) \bigoplus \Big ( \bigoplus_{T \in F} G_i(T/G)  \Big ).
$$ 
Similarly, if $S' = [U/G] \in S_{I'}$ is a stratum, we have 
$$  K_i(S') \otimes \mathbb{C} = G_i(S') \cong G_i(\widetilde{IS'}) \cong \bigoplus_{g \in G} G_i(Y^g \cap U /G)= 
G_i(U/G) \bigoplus \Big (\bigoplus_{T \in F} G_i(T \cap U/G) \Big ).$$
Thus, if we complexify formula (\ref{prelstepeq}),  we find that $K_i((X,D)_{\mathrm{Kfl}})  \otimes \mathbb{C}$ is isomorphic to  \begin{equation}
\label{eqeqdecompdecomp}
  G_i(X) \bigoplus \Big  ( \bigoplus_{T \in F} G_i(T/G) \Big ) \bigoplus \Big ( \bigoplus_{U \in S_{L}, U \neq Y} \Big ( \bigoplus_{\chi \in (\mathbb{Q}/\mathbb{Z})_U^*} G_i(U/G) \bigoplus 
\Big ( \bigoplus_{T \in F} G_i(T \cap U/G) \Big ) \Big ) \Big ).  
\end{equation}
The main difference between the statement we need to prove  and the decomposition  (\ref{formula1234}) 
which holds in the simple normal crossings case is that, in general,   the indexing set  $(\mathbb{Q}/\mathbb{Z})_{F_S}^*$ corresponding to a stratum $S \in S_D$ is larger than the indexing set 
$(\mathbb{Q}/\mathbb{Z})_{S}^*$ which appears in 
(\ref{formula1234}).  The reason is that bigger strata  containing $S$ might 
split off extra factors  of the form $G_i(S)$ owing to the Atiyah--Segal decomposition encoded in (\ref{eqeqdecompdecomp}).  
Formula (\ref{complex-decompo}) is then obtained by rearranging the factors on the right-hand side of (\ref{eqeqdecompdecomp}) so as to group  together all factors of the form $G_i(S).$

More precisely, let $S=U/G$ be a stratum  of    $X.$   
Assume that 
 there exists a pair $T \in F,$  $V \in S_E$ such that $U=T \cap V.$   Then the summand of (\ref{prelstepeq}) corresponding to the stratum  $[V/G] \in S_{D'}$ is $
 \bigoplus_{\chi \in (\mathbb{Q}/\mathbb{Z})_{[V/G]}^*} K_i([V/G]) \otimes \mathbb{C}$ and  
can be rewritten as
$$
 G_i(U/G) \bigoplus \Big (
 \bigoplus_{\chi \in (\mathbb{Q}/\mathbb{Z})_V^*} 
 G_i(T  \cap V/G) \Big ) \bigoplus \Big ( \bigoplus_{\chi \in (\mathbb{Q}/\mathbb{Z})_V^*} \Big (
\bigoplus_{T' \in F, T' \neq T} G_i(T' \cap V/G) \Big ) \Big )
 \cong
$$ 
$$
\cong G_i(U/G) \bigoplus  \Big ( \bigoplus_{\chi \in (\mathbb{Q}/\mathbb{Z})_V^*} 
 G_i(S) \Big ) \bigoplus \Big (\bigoplus_{\chi \in (\mathbb{Q}/\mathbb{Z})_V^*} 
\Big (\bigoplus_{T' \in F, T' \neq T} G_i(T' \cap V/G) \Big ) \Big ).
$$
Thus 
it splits  off   
a summand 
$\bigoplus_{\chi \in (\mathbb{Q}/\mathbb{Z})_V^*} 
 G_i(S).$
 Taking into account the contributions coming from all  pairs $T \in F,$  $V \in S_E$ such that $U=T \cap V,$ yields the summand  
 $
 \bigoplus_{\chi \in (\mathbb{Q}/\mathbb{Z})_{F_S}^*} G_i(S)
 $ 
 which appears in (\ref{complex-decompo}). This concludes the proof. 
\end{proof}

\subsection{Logarithmic Chern character}
\label{logchchar}
In this last section we sketch one additional application of our techniques. Namely, we  define a \emph{logarithmic Chern 
character} and  
explain some of its basic properties. We conclude by formulating a Grothendieck-Riemann-Roch statement for the logarithmic Chern character. For simplicity in this section $\kappa$ will be a field of characteristic $0$.

Recall from Section \ref{dgchern} the definition of the Chern character morphism $\mathrm{ch}$ in the setting of $\infty$-categories.

\begin{definition}
\label{deflogchernchar}
Let $X$ be a log algebraic stack. We define the  \emph{logarithmic Chern character}  to be the morphism
$
\mathrm{ch}\colon K(X_{\mathrm{Kfl}}  )\longrightarrow \mathrm{HH}( X_{\mathrm{Kfl}}). 
$
\end{definition}

To emphasize the fact that we are in the logarithmic setting, we will denote the logarithmic Chern character   by $\mathrm{ch}^{\mathrm{log}}.$  
The next statement follows immediately from Corollary \ref{ncmotdirectsum}. 
\begin{proposition}
\label{logcherncomm}
Let $(X, D)$ be a log stack given by 
an algebraic stack $X$ equipped with a simple normal crossings divisor $D.$ Let $I$ be the set of irreducible components of $D$ and denote by $S_I$ the set of strata.  Then there  is a commutative diagram   
$$
\xymatrix{
K((X,D)_{\mathrm{Kfl}}) \ar[r]^-{\mathrm{ch}^{\mathrm{log}}} \ar[d]_-\simeq & \mathrm{HH}((X,D)_{\mathrm{Kfl}}) \ar[d]^-\simeq \\ 
K(X) \bigoplus \Big ( \bigoplus_{S \in S_I^*} \Big ( \bigoplus_{\chi \in (\mathbb{Q}/\mathbb{Z})_S^*} K(S) \Big ) \Big ) \ar[r]^-{\oplus \mathrm{ch}} & \mathrm{HH}(X) \bigoplus \Big ( \bigoplus_{S \in S_I^*} \Big ( \bigoplus_{\chi \in (\mathbb{Q}/\mathbb{Z})_S^*} \mathrm{HH}(S) \Big ) \Big ) }
$$
where $\oplus \mathrm{ch}$ denotes the direct sum of the Chern character maps  
$ \mathrm{ch}\colon K(S) \longrightarrow \mathrm{HH}(S)
$ for $S \in S_I$.
\end{proposition}

\begin{definition}
\label{logchernrem1}
Let $(X, D)$ be a log scheme given by a smooth and proper scheme $X$ together with a simple normal crossings divisor $D$. Then we define the \emph{de Rham logarithmic Chern character} $\mathrm{ch}^{\mathrm{log}}_{\mathrm{dR}}$ 
to be the composite
$$
\xymatrix{K_0((X,D)_{\mathrm{Kfl}}) \ar[r]^-{\mathrm{ch}^{\mathrm{log}}} \ar[dr]_ -{\mathrm{ch}^{\mathrm{log}}_{\mathrm{dR}}} & \mathrm{HH}_0 ((X,D)_{\mathrm{Kfl}}) \ar[d]^-\cong \\ 
 & \bigoplus_{k \geq 0}\mathrm{H}_{\mathrm{dR}}^{2k}(X) \bigoplus \Big ( \bigoplus_{S \in S_I^*} \Big ( \bigoplus_{\chi \in (\mathbb{Q}/\mathbb{Z})_S^*} \bigoplus_{k \geq 0}\mathrm{H}_{\mathrm{dR}}^{2k}(S) \Big ) \Big ) }
$$
\end{definition}

\begin{remark} 
\label{logchernrem2}
The morphism $\mathrm{ch}^{\mathrm{log}}_{\mathrm{dR}}$ is closely related to the parabolic Chern character considered in \cite{iyer2007relation}. One  difference is that 
the authors in  \cite{iyer2007relation} work with finite rather than  infinite root stacks.  
\end{remark}
We conclude by stating a Grothendieck--Riemann--Roch theorem for the logarithmic Chern character. We will place ourselves  under quite restrictive assumptions. We will return to the problem of extending this logarithmic GRR formalism to a larger class of log stacks in future work.   
Let $f\colon (Y, E) \longrightarrow (X, D)$ be a strict map of log schemes having the following properties: 
\begin{itemize}
\item the underlying schemes $Y$ and $X$ are smooth and proper, and $E$ and $D$ are simple normal crossings divisors; 
\item the morphism between the underlying schemes $f\colon Y \to X$ is flat and proper.
\end{itemize}
 Let $L$ and $I$ be the irreducible components of $E$ and $D$ and denote by $S_L$ and $S_I$ the sets of strata. 
Note that each stratum $S_Y \in S_L$ is mapped by $f$ to a stratum $S_X \in S_I.$ Further, for each stratum $S_Y \in S_L,$  the classical Grothendieck--Riemann--Roch theorem gives a 
commutative diagram 
\begin{equation}
\label{grrstratum}
\begin{gathered}
\xymatrix{
K_0(S_Y) \ar[r]^-{\mathrm{ch}^{\mathrm{dR}}} \ar[d] & \bigoplus_{k \geq 0}\mathrm{H}_{\mathrm{dR}}^{2k}(S_Y) \ar[d]^-{f_*(-  \wedge \mathrm{Td_{S_Y/S_X}} )} \\ 
K_0(S_X) \ar[r]^-{\mathrm{ch}^{\mathrm{dR}}} & \bigoplus_{k \geq 0}\mathrm{H}_{\mathrm{dR}}^{2k}(S_X) 
}
\end{gathered}
\end{equation}
where  $\mathrm{Td_{S_Y/S_X}}$ is
the Todd class of the relative tangent bundle. 
Taking the direct sum of the vertical morphism  on the right of  (\ref{grrstratum})  over all strata we obtain a morphism 
$$
\xymatrix{ 
 \bigoplus_{k \geq 0}\mathrm{H}_{\mathrm{dR}}^{2k}(Y)    
 \bigoplus \Big ( \bigoplus_{S  \in S_L^*} \Big ( \bigoplus_{\chi \in (\mathbb{Q}/\mathbb{Z})^*_{S }} \bigoplus_{k \geq 0}  \mathrm{H}_{\mathrm{dR}}^{2k}(S ) \Big )  
\ar[d]^-{\bigoplus f_*(-\wedge\mathrm{Td})} \Big ) 
\\ 
\bigoplus_{k \geq 0}\mathrm{H}_{\mathrm{dR}}^{2k}(X)    
 \bigoplus \Big ( \bigoplus_{S  \in S_I^*} \Big ( \bigoplus_{\chi \in (\mathbb{Q}/\mathbb{Z})^*_{S }} \bigoplus_{k \geq 0}  \mathrm{H}_{\mathrm{dR}}^{2k}(S ) \Big )  \Big )
}
$$
which we denote for simplicity   
$\bigoplus f_*(-\wedge \mathrm{Td}),$ dropping the indices from the Todd classes. 
\begin{proposition}
Let $f\colon (Y, E) \longrightarrow (X, D)$ be a map of log schemes satisfying the properties above. Then:
\begin{enumerate}[leftmargin=*]
\item  
There is a commutative diagram in $\cS_\infty$
$$
\xymatrix{
K((Y,E)_{\mathrm{Kfl}}) \ar[r]^-{\mathrm{ch}^{\mathrm{log}}} \ar[d]_-{f_*} & \mathrm{HH}((Y,E)_{\mathrm{Kfl}}) \ar[d]^-{f_*} \\ 
K((X,D)_{\mathrm{Kfl}}) \ar[r]^-{\mathrm{ch}^{\mathrm{log}}} & \mathrm{HH}((X,D)_{\mathrm{Kfl}}).
}
$$
\item There is a commutative diagram of abelian groups
$$
\xymatrix{
K_0((Y,E)_{\mathrm{Kfl}}) \ar[r]^-{\mathrm{ch}^{\mathrm{log}}_{\mathrm{dR}}} \ar[d]_-{f_*} & \bigoplus_{k \geq 0}\mathrm{H}_{\mathrm{dR}}^{2k}(Y)    
 \bigoplus \Big ( \bigoplus_{S  \in S_L^*} \Big ( \bigoplus_{\chi \in (\mathbb{Q}/\mathbb{Z})^*_{S }} \bigoplus_{k \geq 0}  \mathrm{H}_{\mathrm{dR}}^{2k}(S ) \Big ) \Big) \ar[d]^-{\bigoplus f_*(-\wedge\mathrm{Td})} \\ 
K_0((X,D)_{\mathrm{Kfl}}) \ar[r]^-{\mathrm{ch}^{\mathrm{log}}_{\mathrm{dR}}} &\bigoplus_{k \geq 0}\mathrm{H}_{\mathrm{dR}}^{2k}(X)    
 \bigoplus \Big ( \bigoplus_{S  \in S_I^*} \Big ( \bigoplus_{\chi \in (\mathbb{Q}/\mathbb{Z})^*_{S }} \bigoplus_{k \geq 0}  \mathrm{H}_{\mathrm{dR}}^{2k}(S ) \Big )\Big).}
$$
\end{enumerate}
\end{proposition}
\begin{proof}
Let us start with the first  statement. We will use  the equivalence $
\Perf((X,D)_{\mathrm{Kfl}}) \simeq \Perf(\radice[\infty]{(X,D)})
$
 from Proposition \ref{kflinf}, and the identifications 
$$
K((Y,E)_{\mathrm{Kfl}}) \simeq K(\radice[\infty]{(Y,E)}), \, \, 
K((X,D)_{\mathrm{Kfl}}) \simeq K(\radice[\infty]{(X,D)}),
$$
$$
\mathrm{HH}((Y,E)_{\mathrm{Kfl}}) \simeq \mathrm{HH}(\radice[\infty]{(Y,E)}), \, \, 
\mathrm{HH}((X,D)_{\mathrm{Kfl}}) \simeq \mathrm{HH}(\radice[\infty]{(X,D)}). 
$$
 Let $
f_r\colon \radice[r]{(Y, E)} \longrightarrow \radice[r]{(X, D)}$  and $ f_\infty\colon \radice[\infty]{(Y, E)} \longrightarrow \radice[\infty]{(X,D)},
$
be the maps between the $r$-th and the infinite root stacks   induced by $f.$ For every $r \in \mathbb{N},$  $f_r$ is   flat and proper (therefore perfect) and thus by \cite[Examples 2.2 (a)]{li-ne} it induces a  push-forward  
$
f_{r,*}\colon \Perf(\radice[r]{(Y, E)}) \longrightarrow \Perf(\radice[r]{(X, D)}) .
$
Taking the colimit over $r$ we obtain the push-forward  
$
f_{\infty,*}\colon \Perf(\radice[\infty]{(Y, E)}) \longrightarrow \Perf(\radice[\infty]{(X, D)}). 
$
Applying $\mathrm{ch}$ to $f_{\infty,*}$ yields the   commutative diagram below, which gives statement $(1)$    
$$
\xymatrix{
K(\radice[\infty]{(Y, E)}) \ar[r]^-{\mathrm{ch}} \ar[d]_-{f_{\infty,*}} & \mathrm{HH}(\radice[\infty]{(Y, E)}) \ar[d]^-{f_{\infty,*}} \\ 
K(\radice[\infty]{(X, D)}) \ar[r]^-{\mathrm{ch}} & \mathrm{HH}(\radice[\infty]{(X, D)}), 
}
$$
Let us consider  the second  statement next. 
For simplicity, we restrict to the case where $D$ and  $E$ are irreducible (and this $f^{-1}(D)=E$, by strictness). The general case is similar. We need to prove  that the push-forward  $f_{\infty,*}$ functor preserves the summands of the psod-s of $\Perf(\radice[\infty]{(Y, E)})$ and $\Perf(\radice[\infty]{(X, D)}).$   As  $f$ is strict the diagram 
\begin{equation}
\begin{gathered}
\label{eqeqfiberfiber}
 \xymatrix{
\radice[\infty]{(Y, E)} \ar[r]^-{g_{\infty,1}} \ar[d]_-{f_\infty} & Y \ar[d]^-{f} \\
  \radice[\infty]{(X, D)} \ar[r]^-{g_{\infty,1}} & X
 }
 \end{gathered}
\end{equation}
is cartesian. Further, $f$ is flat and thus base-change yields a commutative diagram
\[
\xymatrix{ \Perf (Y) \ar[d]_{f_*}  \ar[r]^<<<<<{ g^*_{\infty, 1}}&  \Perf(\radice[\infty]{(Y, E)}) \ar[d]^{f_{\infty,*}}\\
\Perf (X) \ar[r]^<<<<<{g^*_{\infty, 1}} &  \Perf (\radice[\infty]{(X, D)}).
}\]
This shows that $f_{\infty, *}$ maps the semi-orthogonal summand 
$ \Perf (Y) \subset \Perf(\radice[\infty]{(Y, E)})$ to $\Perf(X)$.   Next, let us turn  to the other summands of the sod of $\Perf(\radice[\infty]{(Y, E)}).$ Again, since $f$ is strict,   the   square below is cartesian  
\[
\xymatrix{ E_r \ar[d]_{f_r}  \ar[r]^{ }& \radice[r]{(Y, E)} \ar[d]^{f_r} \\
D_r \ar[r]  & \radice[r]{(X, D)},  }\]
where as usual $E_r$ and $D_r$ denote the universal Cartier divisors of the two root stacks. Note that flatness of $f$ implies flatness of $f_r$, and therefore $E_r$ is also the derived fiber product of the diagram.

Base change yields a commutative diagram
\[
\xymatrix{
\bigoplus_{\chi \in \mathbb{Z}_r}\Perf (E_r) _\chi  \simeq    \Perf (E_r) \ar[d]_-{f_{r,*}}  \ar[r] & \Perf (\radice[r]{(Y, E)}) \ar[d]^{f_{r,*}}  \\
\bigoplus_{\chi \in \mathbb{Z}_r} \Perf (D_r)_\chi  \simeq  \Perf (D_r) \ar[r]^{} &  \Perf   (\radice[r]{(X, D)}).
}
\]
Additionally, for every $\chi \in \mathbb{Z}_r,$ the restriction of $f_{r,*}$ to $(\Perf (E_r) )_\chi$ coincides with $f_*$: more precisely, there is a commutative  diagram 
$$
\xymatrix{
(\Perf (E_r) )_\chi     \ar[d]_-{f_{r,*}}  \ar[r]^-{\simeq} &  \Perf (E)\ar[d]^{f_{*}}  \\
(\Perf (D_r) )_\chi        \ar[r]^-{\simeq} &  \Perf (D).
}
$$
This shows that $f_*$ respects the summands of the sod-s of the $r$-th root stacks given by Proposition \ref{prop: sod}. We are actually interested in the compatibility with the sod-s of $n!$-th root stacks constructed recursively in Proposition \ref{sod!}. Note however that the latter are obtained iterating the construction from Proposition \ref{prop: sod}: thus iterating the argument above also implies that $f_*$ respects the sod given in Proposition \ref{sod!}.

This implies that the push-forward map $
 f_*\colon K((Y,E)_{\mathrm{Kfl}}) 
 \longrightarrow  K((X,D)_{\mathrm{Kfl}})
 $ decomposes as a direct sum of push-forwards along $f,$ which we denote by 
 $\bigoplus f_*,$
\begin{equation}
\label{pushdirectpush} \bigoplus f_*\colon K(Y) \bigoplus \Big ( \bigoplus_{\chi \in (\mathbb{Q}/\mathbb{Z})^*} K( E) \Big ) \longrightarrow K(Y) \bigoplus \Big (\bigoplus_{\chi \in (\mathbb{Q}/\mathbb{Z})^*} K(D) \Big ).
\end{equation} 
Then the second  statement follows by applying the ordinary Grothendiek--Riemann--Roch to each summand in (\ref{pushdirectpush}). 
\end{proof}

\bibliographystyle{plain}
\bibliography{biblio}

\providecommand\noopsort[1]{t}
\begin{thebibliography}{10}

\bibitem{AC}
Dan Abramovich and Qile Chen.
\newblock Stable logarithmic maps to {D}eligne-{F}altings pairs {II}.
\newblock {\em Asian Journal of Mathematics}, 18(3):465--488, 2014.

\bibitem{BFN}
David Ben-Zvi, John Francis, and David Nadler.
\newblock Integral transforms and {D}rinfeld centers in derived algebraic
  geometry.
\newblock {\em J. Amer. Math. Soc.}, 23(4):909--966, 2010.

\bibitem{bergh2016geometricity}
Daniel Bergh, Valery~A Lunts, and Olaf~M Schn{\"u}rer.
\newblock Geometricity for derived categories of algebraic stacks.
\newblock {\em Selecta Mathematica}, 22(4):2535--2568, 2016.

\bibitem{BGT}
A.~Blumberg, D.~Gepner, and G.~Tabuada.
\newblock A universal characterization of higher algebraic {$K$}-theory.
\newblock {\em Geom. Topol.}, 17(2):733--838, 2013.

\bibitem{bondal1990representable}
Alexei~I. Bondal and Mikhail~M. Kapranov.
\newblock Representable functors, {S}erre functors, and mutations.
\newblock {\em Izvestiya: Mathematics}, 35(3):519--541, 1990.

\bibitem{borne-vistoli}
Niels Borne and Angelo Vistoli.
\newblock Parabolic sheaves on logarithmic schemes.
\newblock {\em Advances in Mathematics}, 231(3-4):1327--1363, oct 2012.

\bibitem{cadman}
Charles Cadman.
\newblock Using stacks to impose tangency conditions on curves.
\newblock {\em American Journal of Mathematics}, 129(2):405--427, 2007.

\bibitem{cualduararu2005mukai}
Andrei C{\u{a}}ld{\u{a}}raru.
\newblock The {M}ukai pairing, {II}: the {H}ochschild--{K}ostant--{R}osenberg
  isomorphism.
\newblock {\em Advances in Mathematics}, 194(1):34--66, 2005.

\bibitem{CSST}
David Carchedi, Sarah Scherotzke, Nicol{\`o} Sibilla, and Mattia Talpo.
\newblock Kato-{N}akayama spaces, infinite root stacks, and the profinite
  homotopy type of log schemes.
\newblock {\em Geom. Topol.}, 21(5):3093--3158, 2017.

\bibitem{C}
Qile Chen.
\newblock Stable logarithmic maps to {D}eligne--{F}altings pairs {I}.
\newblock {\em Annals of Mathematics}, 180(2):455--521, 2014.

\bibitem{cisinski2011non}
Denis-Charles Cisinski and Gon{\c{c}}alo Tabuada.
\newblock Non-connective {$K$}-theory via universal invariants.
\newblock {\em Compositio Mathematica}, 147(4):1281--1320, 2011.

\bibitem{collins2010gluing}
John Collins, Alexander Polishchuk, et~al.
\newblock Gluing stability conditions.
\newblock {\em Advances in Theoretical and Mathematical Physics},
  14(2):563--608, 2010.

\bibitem{conrad}
Brian Conrad.
\newblock From normal crossings to strict normal crossings.
\newblock
  \url{http://math.stanford.edu/~conrad/249BW17Page/handouts/crossings.pdf}.

\bibitem{dhillon2015g}
Ajneet Dhillon and Ivan Kobyzev.
\newblock G-theory of root stacks and equivariant k-theory.
\newblock {\em arXiv preprint arXiv:1510.06118}, 2015.

\bibitem{fantechi2010smooth}
Barbara Fantechi, Etienne Mann, and Fabio Nironi.
\newblock Smooth toric {D}eligne-{M}umford stacks.
\newblock {\em Journal f{\"u}r die reine und angewandte Mathematik (Crelle's
  Journal)}, 2010(648):201--244, 2010.

\bibitem{Ga}
Dennis Gaitsgory.
\newblock Ind-coherent sheaves.
\newblock {\em Moscow Mathematical Journal}, 13(3):399--528, 2013.

\bibitem{gaitsgory2017study}
Dennis Gaitsgory and Nick Rozenblyum.
\newblock {\em A study in derived algebraic geometry}, volume~1.
\newblock American Mathematical Soc., 2017.

\bibitem{GS}
Mark Gross and Bernd Siebert.
\newblock Logarithmic {G}romov-{W}itten invariants.
\newblock {\em Journal of the American Mathematical Society}, 26(2):451--510,
  2013.

\bibitem{gross2006mirror}
Mark Gross, Bernd Siebert, et~al.
\newblock Mirror symmetry via logarithmic degeneration data {I}.
\newblock {\em Journal of Differential Geometry}, 72(2):169--338, 2006.

\bibitem{hagihara}
Kei Hagihara.
\newblock Structure theorem of {K}ummer \'etale {$K$}-group.
\newblock {\em $K$-Theory}, 29(2):75--99, 2003.

\bibitem{hagihara2016structure}
Kei Hagihara.
\newblock Structure theorem of {K}ummer \'etale {$K$}-group {II}.
\newblock {\em Documenta Mathematica}, 21:1345--1396, 2016.

\bibitem{hesselholt2003k}
Lars Hesselholt and Ib~Madsen.
\newblock On the k-theory of local fields.
\newblock {\em Annals of mathematics}, 158(1):1--113, 2003.

\bibitem{howell}
Nicholas~I. Howell.
\newblock Motives of log schemes.
\newblock
  \url{https://scholarsbank.uoregon.edu/xmlui/bitstream/handle/1794/22740/Howell_oregon_0171A_11948.pdf?sequence=1},
  2017.

\bibitem{HSS}
Marc Hoyois, Sarah Scherotzke, and Nicol{\`o} Sibilla.
\newblock Higher traces, noncommutative motives, and the categorified {C}hern
  character.
\newblock {\em Advances in Mathematics}, 309:97--154, 2017.

\bibitem{ishii2011special}
Akira Ishii, Kazushi Ueda, et~al.
\newblock The special {M}c{K}ay correspondence and exceptional collections.
\newblock {\em Tohoku Mathematical Journal}, 67(4):585--609, 2015.

\bibitem{ito2017log}
Tetsushi Ito, Kazuya Kato, Chikara Nakayama, and Sampei Usui.
\newblock On log motives.
\newblock Preprint, \href{http://arxiv.org/abs/1712.09815}{arXiv:1712.09815}.

\bibitem{iyer2007relation}
Jaya~NN Iyer and Carlos~T Simpson.
\newblock A relation between the parabolic {C}hern characters of the de {R}ham
  bundles.
\newblock {\em Mathematische Annalen}, 338(2):347--383, 2007.

\bibitem{kato}
Kazuya Kato.
\newblock Logarithmic structures of {F}ontaine-{I}llusie.
\newblock In {\em Algebraic analysis, geometry, and number theory ({B}altimore,
  {MD}, 1988)}, pages 191--224. Johns Hopkins Univ. Press, Baltimore, MD, 1989.

\bibitem{krause2010localization}
Henning Krause.
\newblock Localization theory for triangulated categories.
\newblock In {\em In Triangulated categories, volume 375 of London Math. Soc.
  Lecture Note Ser}. Citeseer, 2010.

\bibitem{krishna2017atiyah}
Amalendu Krishna and Bhamidi Sreedhar.
\newblock Atiyah-{S}egal theorem for {D}eligne-{M}umford stacks and
  applications.
\newblock Preprint, \href{http://arxiv.org/abs/1701.05047}{arXiv:1701.05047}.

\bibitem{kuznetsov2015semiorthogonal}
Alexander Kuznetsov.
\newblock Derived categories view on rationality problems.
\newblock Preprint,
  \href{https://arxiv.org/pdf/1509.09115.pdf}{arXiv:1509.09115}.

\bibitem{leip}
Malte Leip.
\newblock Thh of log rings.
\newblock In L.~Hesselholt and P.~Scholze, editors, {\em Arbeitsgemeinschaft:
  Topological Cyclic Homology}. Mathematisches Forschungsinstitut Oberwolfach,
  2018.

\bibitem{li-ne}
Joseph Lipman and Amnon Neeman.
\newblock Quasi-perfect scheme-maps and boundedness of the twisted inverse
  image functor.
\newblock {\em Illinois J. Math.}, 51(1):209--236, 2007.

\bibitem{lurie2009higher}
Jacob Lurie.
\newblock {\em Higher Topos Theory (AM-170)}.
\newblock Princeton University Press, 2009.

\bibitem{Lu2}
Jacob Lurie.
\newblock {\em Higher Algebra}.
\newblock 2016.
\newblock \url{http://www.math.harvard.edu/~lurie/papers/HA.pdf}.

\bibitem{Ni1}
Wies{\l}awa Nizio{\l}.
\newblock {$K$}-theory of log-schemes. {I}.
\newblock {\em Doc. Math.}, 13:505--551, 2008.

\bibitem{ogus}
Arthur Ogus.
\newblock Lectures on logarithmic algebraic geometry.
\newblock Version from December 18, 2017. To be published by Cambridge
  University Press.

\bibitem{ollsonhh}
Martin Olsson.
\newblock Hochschild and cyclic homology of log schemes.
\newblock Talk available at \url{https://www.youtube.com/watch?v=vPkSZm8DOYk}.

\bibitem{Ols}
Martin Olsson.
\newblock Logarithmic geometry and algebraic stacks.
\newblock {\em Ann. Sci. \'Ecole Norm. Sup. (4)}, 36(5):747--791, 2003.

\bibitem{Ol}
Martin Olsson.
\newblock The logarithmic cotangent complex.
\newblock {\em Math. Ann.}, 333(4):859--931, 2005.

\bibitem{robalo2015k}
Marco Robalo.
\newblock {$K$}-theory and the bridge from motives to noncommutative motives.
\newblock {\em Advances in Mathematics}, 269:399--550, 2015.

\bibitem{rognes2015localization}
John Rognes, Steffen Sagave, and Christian Schlichtkrull.
\newblock Localization sequences for logarithmic topological hochschild
  homology.
\newblock {\em Mathematische Annalen}, 363(3-4):1349--1398, 2015.

\bibitem{R}
Nick Rozenblyum.
\newblock Filtered colimits of $\infty$-categories.
\newblock Preprint,
  \url{http://www.math.harvard.edu/~gaitsgde/GL/colimits.pdf}, 2012.

\bibitem{sala2017hall}
Francesco Sala and Olivier Schiffmann.
\newblock The circle quantum group and the infinite root stack of a curve (with
  an appendix by {T}atsuki {K}uwagaki).
\newblock Preprint, \href{http://arxiv.org/abs/1711.07391}{arXiv:1711.07391}.

\bibitem{satriano2013canonical}
Matthew Satriano.
\newblock Canonical artin stacks over log smooth schemes.
\newblock {\em Mathematische Zeitschrift}, 274(3-4):779--804, 2013.

\bibitem{scherotzke2016logarithmic}
Sarah Scherotzke, Nicol{\`o} Sibilla, and Mattia Talpo.
\newblock On a logarithmic version of the derived {M}c{K}ay correspondence.
\newblock Preprint, arXiv:1612.08961, to appear in Compos. Math., 2018.

\bibitem{stacks-project}
The {Stacks Project Authors}.
\newblock Stacks project.
\newblock \url{http://stacks.math.columbia.edu}, 2018.

\bibitem{tabuada2008}
Gon{\c c}alo Tabuada.
\newblock Higher {$K$}-theory via universal invariants.
\newblock {\em Duke Math. J.}, 145(1):121--206, 10 2008.

\bibitem{talpo2014moduli}
Mattia Talpo.
\newblock Moduli of parabolic sheaves on a polarized logarithmic scheme.
\newblock {\em Trans. Amer. Math. Soc.}, 369(5):3483--3545, 2017.

\bibitem{talpo2017parabolic}
Mattia Talpo.
\newblock Parabolic sheaves with real weights as sheaves on the
  {K}ato--{N}akayama space.
\newblock {\em Adv. Math.}, 336:97--148, 2018.

\bibitem{TV}
Mattia Talpo and Angelo Vistoli.
\newblock Infinite root stacks and quasi-coherent sheaves on logarithmic
  schemes.
\newblock {\em Proc. Lond. Math. Soc.}, 117(5):1187--1243, 2018.

\bibitem{TaV}
Mattia Talpo and Angelo Vistoli.
\newblock The {K}ato–{N}akayama space as a transcendental root stack.
\newblock {\em Int. Math. Res. Not.}, 2018(19):6145--6176, 2018.

\end{thebibliography}

\end{document}